\documentclass[a4paper,leqno,12pt]{amsart}
\usepackage{amsmath,amsthm}
\usepackage{amssymb}
\usepackage[all]{xy}
\SelectTips{cm}{12}
\UseTips{}
\usepackage{xspace}
\usepackage{bm}
\usepackage{amstext}
\usepackage{amsfonts}
\usepackage[mathscr]{euscript}
\usepackage{amscd}
\usepackage{latexsym}
\usepackage{amssymb}
\usepackage{enumerate}
\theoremstyle{plain}
\swapnumbers
    \newtheorem{thm}{Theorem}[section]
    \newtheorem{prop}[thm]{Proposition}
    \newtheorem{lemma}[thm]{Lemma}
    \newtheorem{keylemma}[thm]{Key Lemma}
    \newtheorem{corollary}[thm]{Corollary}
    \newtheorem{subsec}[thm]{}
    \newtheorem*{thma}{Theorem A}
    \newtheorem*{thmb}{Theorem B}
    \newtheorem*{thmc}{Theorem C}
    \newtheorem*{thmd}{Theorem D}
\theoremstyle{definition}
    \newtheorem{defn}[thm]{Definition}
    \newtheorem{example}[thm]{Example}

\theoremstyle{remark}
        \newtheorem{remark}[thm]{Remark}
    \newtheorem{assume}[thm]{Assumption}
    \newtheorem{ack}[thm]{Acknowledgements}
\newenvironment{myeq}[1][]
{\stepcounter{thm}\begin{equation}\tag{\thethm}{#1}}
{\end{equation}}

\newcommand{\mydiagram}[2][]
{\stepcounter{thm}\begin{equation}
     \tag{\thethm}{#1}\vcenter{\xymatrix{#2}}\end{equation}}
\newcommand{\mysdiag}[2][]
{\stepcounter{thm}\begin{equation}
     \tag{\thethm}{#1}\vcenter{\xymatrix@R=25pt@C=-25pt{#2}}\end{equation}}
\newcommand{\mytdiag}[2][]
{\stepcounter{thm}\begin{equation}
     \tag{\thethm}{#1}\vcenter{\xymatrix@R=25pt@C=15pt{#2}}\end{equation}}
\newcommand{\myudiag}[2][]
{\stepcounter{thm}\begin{equation}
     \tag{\thethm}{#1}\vcenter{\xymatrix@R=27pt@C=12pt{#2}}\end{equation}}
\newenvironment{mysubsection}[2][]
{\begin{subsec}\begin{upshape}\begin{bfseries}{#2.}
\end{bfseries}{#1}}
{\end{upshape}\end{subsec}}
\newenvironment{mysubsect}[2][]
{\begin{subsec}\begin{upshape}\begin{bfseries}{#2\vsn.}
\end{bfseries}{#1}}
{\end{upshape}\end{subsec}}
\newcommand{\sect}{\setcounter{thm}{0}\section}
\newcommand{\wh}{\ -- \ }
\newcommand{\wwh}{-- \ }
\newcommand{\w}[2][ ]{\ \ensuremath{#2}{#1}\ }
\newcommand{\ww}[1]{\ \ensuremath{#1}}

\newcommand{\wwb}[1]{\ \ensuremath{(#1)}-}
\newcommand{\wb}[2][ ]{\ (\ensuremath{#2}){#1}\ }
\newcommand{\wref}[2][ ]{\ (\ref{#2}){#1}\ }

\newcommand{\hsp}{\hspace*{9 mm}}
\newcommand{\hs}{\hspace*{4 mm}}
\newcommand{\hsn}{\hspace*{1 mm}}
\newcommand{\hsm}{\hspace*{2 mm}}
\newcommand{\vsn}{\vspace{2 mm}}
\newcommand{\vs}{\vspace{7 mm}}
\newcommand{\vsm}{\vspace{4 mm}}
\newcommand{\hra}{\hookrightarrow}
\newcommand{\xra}[1]{\xrightarrow{#1}}
\newcommand{\xepic}[1]{\xrightarrow{#1}\hspace{-5 mm}\to}
\newcommand{\lra}[1]{\langle{#1}\rangle}
\newcommand{\llrra}[1]{\langle\langle{#1}\rangle\rangle}
\newcommand{\llrr}[2]{\langle\langle{#1}\rangle\rangle\sb{#2}}
\newcommand{\lrf}{\langle\langle f\lo{0,1}\rangle\rangle}
\newcommand{\lrfn}[1]{\lrf\sb{#1}}

\newcommand{\epic}{\to\hspace{-3.5 mm}\to}

\newcommand{\up}[1]{\sp{(#1)}}
\newcommand{\bup}[1]{\sp{[{#1}]}}
\newcommand{\lo}[1]{\sb{(#1)}}
\newcommand{\bp}[1]{\sb{[#1]}}
\newcommand{\rest}[1]{\lvert\sb{#1}}
\newcommand{\Hu}[3]{H\sp{#1}({#2};{#3})}

\newcommand{\HiQ}[1]{\Hu{\ast}{#1}{Q}}
\newcommand{\HiR}[1]{\Hu{\ast}{#1}{R}}
\newcommand{\HiT}[1]{H\sp{\ast}\sb{\Theta}{#1}}
\newcommand{\Cof}{\operatorname{Cof}}
\newcommand{\Coker}{\operatorname{Coker}}
\newcommand{\colim}{\operatorname{colim}}
\newcommand{\Cone}{\operatorname{Cone}}
\newcommand{\csk}[1]{\operatorname{csk}\sp{#1}}

\newcommand{\ev}{\operatorname{ev}}
\newcommand{\Fib}{\operatorname{Fib}}

\newcommand{\ho}{\operatorname{ho}}

\newcommand{\holim}{\operatorname{holim}}
\newcommand{\Hom}{\operatorname{Hom}}
\newcommand{\vth}{\vartheta}
\newcommand{\Ker}{\operatorname{Ker}}
\newcommand{\Id}{\operatorname{Id}}
\newcommand{\inc}{\operatorname{inc}}
\newcommand{\Obj}{\operatorname{Obj}\,}
\newcommand{\op}{\sp{\operatorname{op}}}
\newcommand{\prj}{\operatorname{pr}}
\newcommand{\proj}{\operatorname{proj}}

\newcommand{\sk}[1]{\operatorname{sk}\sb{#1}}
\newcommand{\Tot}{\operatorname{Tot}}

\newcommand{\Val}[1]{\operatorname{Val}({#1})}
\newcommand{\map}{\operatorname{map}}
\newcommand{\mapa}{\map\sb{\ast}}
\newcommand{\A}{\mathbf{A}}
\newcommand{\eA}{{\EuScript A}}
\newcommand{\B}{\mathcal{B}}

\newcommand{\C}{\mathcal{C}}

\newcommand{\E}{\mathcal{E}}
\newcommand{\F}{\mathcal{F}}

\newcommand{\wF}{\widetilde{F}}
\newcommand{\FQ}[1]{\F\sb{\QQ}\lra{#1}}
\newcommand{\G}{\mathcal{G}}

\newcommand{\Ink}[2]{I\sp{#1}\lo{#2}}
\newcommand{\K}{\mathcal{K}}
\newcommand{\LG}{\hat{L}\sb{\G}}

\newcommand{\Map}{{\EuScript Map}}
\newcommand{\MT}{\Map\sb{\bT}}

\newcommand{\PP}{\mathcal{P}}
\newcommand{\Pn}[1]{\PP\sp{#1}}

\newcommand{\U}{\mathcal{U}}
\newcommand{\ppp}{\hspace*{0.3mm}\sp{\prime}\hspace{-0.3mm}}
\newcommand{\pppp}{\hspace*{0.5mm}\sp{\prime}\hspace{-0.7mm}}
\newcommand{\ppppp}{\hspace*{0.5mm}\sp{\prime}\hspace{-0.3mm}}
\newcommand{\wa}{\widetilde{a}}
\newcommand{\ak}[1]{a\sp{#1}}

\newcommand{\hak}[1]{\widehat{a}\sp{#1}}

\newcommand{\uak}[1]{\underline{a}\sp{#1}}
\newcommand{\uhak}[1]{\underline{\widehat{a}}\sp{#1}}

\newcommand{\alk}[2]{a\sp{#1}\bp{#2}}

\newcommand{\Bk}[2]{B\sp{#1}\sb{#2}}
\newcommand{\Ck}[1]{C\sp{#1}}
\newcommand{\uetk}[1]{\underline{g}\sp{#1}}
\newcommand{\Fk}[1]{F\sp{#1}}

\newcommand{\Fkp}[1]{\pppp F\sp{#1}}
\newcommand{\hFk}[1]{\widehat{F}\sp{#1}}

\newcommand{\uFk}[1]{\underline{F}\sp{#1}}
\newcommand{\uhFk}[1]{\underline{\widehat{F}}\sp{#1}}

\newcommand{\Flk}[2]{F\sp{#1}\bp{#2}}

\newcommand{\ovp}{\overline{\varphi}}
\newcommand{\oph}[1]{\ovp\sp{#1}}
\newcommand{\ophl}[1]{\ovp\sb{#1}}

\newcommand{\ophin}[2]{\oph{#1}\sb{({#2})}}
\newcommand{\prn}[1]{\pi\sb{[{#1}]}}
\newcommand{\prnk}[2]{\pi\sb{[{#1}]}\sp{#2}}
\newcommand{\psn}[2]{\psi\sp{#1}\sb{[{#2}]}}
\newcommand{\orh}[1]{\overline{\rho}\sb{#1}}
\newcommand{\qk}[2]{q\sp{#1}\bp{#2}}
\newcommand{\Sn}[1]{\mathbf{s}\sp{#1}}
\newcommand{\Snk}[2]{\Sn{#1}\sb{\widehat{#2}}}
\newcommand{\varu}[1]{\vare\up{#1}}
\newcommand{\vare}{\varepsilon}
\newcommand{\varep}{\ppp\vare}

\newcommand{\vn}[1]{\vare\sb{[{#1}]}}
\newcommand{\svn}[1]{\overline{\vare\sb{[{#1}]}}}
\newcommand{\tvn}[1]{\widetilde{\vare}\sb{[{#1}]}}

\newcommand{\hy}[2]{{#1}\text{-}{#2}}
\newcommand{\Alg}[1]{\hy{#1}{\mbox{\sf Alg}}}
\newcommand{\Cha}{\mbox{\sf Ch}}
\newcommand{\Ch}[1]{\Cha\sp{#1}}
\newcommand{\Chn}[2]{\Ch{#1}\sb{\leq{#2}}}
\newcommand{\Chnm}[2]{\Ch{#1}\sb{-1\leq\ast\leq{#2}}}
\newcommand{\Mod}[1]{\hy{#1}{\mbox{\sf Mod}}}
\newcommand{\RM}{\Mod{R}}
\newcommand{\cS}{{\EuScript S}}
\newcommand{\Sa}{\cS\sb{\ast}}
\newcommand{\SQ}{\cS\sb{\QQ}}
\newcommand{\Set}{\mbox{\sf Set}}
\newcommand{\Seta}{\Set\sb{\ast}}
\newcommand{\DGA}{\mbox{\sf CDGA}}
\newcommand{\Grp}{\mbox{\sf Gp}}
\newcommand{\Top}{\mbox{\sf Top}}

\newcommand{\Del}{\mathbf{\Delta}}
\newcommand{\Deln}[1]{\Del[#1]}
\newcommand{\Delnk}[2]{\Del\sb{#2}[#1]}
\newcommand{\Dp}{\Del\sb{+}}
\newcommand{\Du}{\Del\sp{\bullet}}
\newcommand{\Htr}{{\mathcal H}}
\newcommand{\Htl}[1]{\Htr\sb{[{#1}]}}
\newcommand{\Htlp}[1]{\pppp\Htr\sb{[{#1}]}}
\newcommand{\Str}{{\mathcal S}}
\newcommand{\Stl}[1]{\Str\sb{[{#1}]}}
\newcommand{\Stlp}[1]{\pppp\Str\sb{[{#1}]}}

\newcommand{\bT}{\mathbf{\Theta}}

\newcommand{\TsR}{\bT\sb{R}}
\newcommand{\ThA}{\Theta\sb{\eA}}

\newcommand{\TQ}{\Theta\sb{\QQ}}
\newcommand{\TR}{\Theta\sb{R}}

\newcommand{\cW}{{\EuScript W}}
\newcommand{\cWp}{\ppp\!\cW}
\newcommand{\cWpp}{\ppp\cWp}
\newcommand{\cWpi}[1]{\ppppp\cW\up{#1}}

\newcommand{\ccWp}{\sp{\prime}\cW}
\newcommand{\cuW}[1]{\cW\up{#1}}

\newcommand{\Phip}[1]{\ppp\Phi\lo{#1}}
\newcommand{\FF}{\mathbb F}
\newcommand{\Fp}{\FF\sb{p}}

\newcommand{\NN}{\mathbb N}
\newcommand{\QQ}{\mathbb Q}

\newcommand{\bA}{{\mathbf A}}
\newcommand{\bB}{{\mathbf B}}

\newcommand{\bC}{{\mathbf C}}
\newcommand{\bD}{{\mathbf D}}
\newcommand{\bE}{{\mathbf E}}
\newcommand{\bv}{{\bm \vare}}
\newcommand{\bve}[1]{\bv\bp{#1}}
\newcommand{\bvep}[1]{\pppp\bv\bp{#1}}
\newcommand{\tbve}[1]{\widetilde{\bv}\bp{#1}}

\newcommand{\bG}{{\mathbf G}}

\newcommand{\bK}{{\mathbf K}}
\newcommand{\KP}[2]{\bK({#1},{#2})}
\newcommand{\KQ}[1]{\KP{\QQ}{#1}}
\newcommand{\KsQ}[1]{\Lambda[{#1}]}
\newcommand{\KR}[1]{\KP{R}{#1}}

\newcommand{\bP}{{\mathbf P}}
\newcommand{\bS}[1]{{\mathbf S}\sp{#1}}
\newcommand{\bU}{{\mathbf U}}

\newcommand{\bW}{{\mathbf W}}

\newcommand{\bX}{{\mathbf X}}
\newcommand{\bY}{{\mathbf Y}}

\newcommand{\bZ}{{\mathbf Z}}
\newcommand{\hbZ}{\widehat{\bZ}}
\newcommand{\bfp}{{\mathbf p}}
\newcommand{\obp}{\overline{\bfp}}
\newcommand{\oobp}{\overline{\obp}}
\newcommand{\bq}{{\mathbf q}}
\newcommand{\obq}{\overline{\bq}}
\newcommand{\oobq}{\overline{\obq}}
\newcommand{\br}{{\mathbf r}}
\newcommand{\obr}{\overline{\br}}
\newcommand{\oobr}{\overline{\obr}}
\newcommand{\bs}{{\mathbf s}}
\newcommand{\obs}{\overline{\bs}}
\newcommand{\oobs}{\overline{\obs}}
\newcommand{\bt}{{\mathbf t}}
\newcommand{\obt}{\overline{\bt}}
\newcommand{\oobt}{\overline{\obt}}
\newcommand{\bu}{{\mathbf u}}
\newcommand{\obu}{\overline{\bu}}

\newcommand{\oobu}{\overline{\obu}}
\newcommand{\bfv}{{\mathbf v}}
\newcommand{\obv}{\overline{\bfv}}

\newcommand{\bw}{{\mathbf w}}
\newcommand{\bx}{{\mathbf x}}
\newcommand{\by}{{\mathbf y}}
\newcommand{\bz}{{\mathbf z}}
\newcommand{\As}{A\sp{\ast}}

\newcommand{\cMl}[1]{C\sb{#1}}
\newcommand{\cM}[1]{C\sp{#1}}
\newcommand{\cMs}{\cM{\ast}}
\newcommand{\cMls}{\cMl{\ast}}
\newcommand{\bCs}{\bC\sp{\ast}}
\newcommand{\bDs}{\bD\sp{\ast}}
\newcommand{\uDs}[1]{\bD\sp{\ast}\lo{#1}}
\newcommand{\bEs}{\bE\sp{\ast}}
\newcommand{\uEs}[1]{\ppp\bEs\lo{#1}}
\newcommand{\bPs}{\bP\sp{\ast}}

\newcommand{\bKs}{\bK\sp{\ast}}
\newcommand{\uPs}[1]{\bPs\lo{#1}}
\newcommand{\cZ}[1]{Z\sp{#1}}
\newcommand{\cZl}[1]{Z\sb{#1}}
\newcommand{\dif}[1]{\delta\sp{#1}}

\newcommand{\dz}[1]{d\sp{0}\sb{#1}}
\newcommand{\od}{\overline{\partial}}
\newcommand{\ud}{\overline{d}}

\newcommand{\odz}[1]{\od\,\sp{#1}\sb{0}}

\newcommand{\udz}[1]{\ud\,\sp{0}\sb{#1}}

\newcommand{\td}{\widetilde{d}}
\newcommand{\tdz}[1]{\td\quad\hspace*{-4mm}\sp{0}\sb{#1}}
\newcommand{\cu}[1]{c({#1})\sp{\bullet}}
\newcommand{\cd}[1]{c({#1})\sb{\bullet}}
\newcommand{\Ad}{A\sb{\bullet}}
\newcommand{\Au}{A\sp{\bullet}}
\newcommand{\Bu}{B\sp{\bullet}}
\newcommand{\Bd}{B\sb{\bullet}}
\newcommand{\Cu}{\cM{\bullet}}
\newcommand{\Cd}{C\sb{\bullet}}
\newcommand{\Gd}{G\sb{\bullet}}
\newcommand{\Gu}{G\sp{\bullet}}
\newcommand{\oG}[1]{\overline{G}\sb{#1}}
\newcommand{\uuG}[1]{\overline{G}\sp{#1}}
\newcommand{\uG}[1]{\overline{\mathbf G}\sp{#1}}
\newcommand{\uH}{\underline{H}}

\newcommand{\Huk}[2]{H\sp{#1}\bp{#2}}
\newcommand{\Hupk}[2]{(H\sp{#1}\bp{#2})'}
\newcommand{\hHk}[1]{\widehat{H}\sp{#1}}
\newcommand{\vHn}[1]{\widehat{H}\bp{#1}}
\newcommand{\tHuk}[1]{\widetilde{H}\sp{#1}}
\newcommand{\hHuk}[2]{\hHk{#1}\bp{#2}}

\newcommand{\oU}[1]{\overline{U}\sb{#1}}
\newcommand{\Ud}{\bU\sb{\bullet}}

\newcommand{\Vd}{V\sb{\bullet}}
\newcommand{\Vpn}[1]{\ppp V\sb{#1}}
\newcommand{\Vdp}{\Vpn{\bullet}}

\newcommand{\Vud}[1]{\Vd\up{#1}}
\newcommand{\oV}[1]{\overline{V}\sb{#1}}
\newcommand{\oVp}[1]{\ppp\overline{V}\sb{#1}}
\newcommand{\ouS}[1]{\otimes\sp{\ast}S\sp{#1}}
\newcommand{\olS}[1]{\otimes\sb{\ast}S\sp{#1}}
\newcommand{\oW}[1]{\overline{\bW}\hspace{0.3mm}\sp{#1}}

\newcommand{\ooW}[2]{\overline{\Omega\sp{#1}\bW\sp{#2}}}
\newcommand{\PoW}[2]{\overline{P\Omega\sp{#1}\bW\sp{#2}}}

\newcommand{\ooWp}[2]{\overline{\Omega\sp{#1}\mbox{$\ppp\bW$}\sp{#2}}}
\newcommand{\PoWp}[2]{\overline{P\Omega\sp{#1}\mbox{$\ppp\bW$}\sp{#2}}}

\newcommand{\oWl}[1]{\overline{W}\sb{#1}}
\newcommand{\oWp}[1]{\ppp\overline{\bW}\sp{#1}}
\newcommand{\hWp}[1]{\ppp\widehat{\bW}\sp{#1}}

\newcommand{\ouV}[2]{\oV{#1}\up{#2}}
\newcommand{\Wu}{\bW\sp{\bullet}}
\newcommand{\Wd}{W\sb{\bullet}}
\newcommand{\hW}[1]{\widehat{\bW}\sp{#1}}

\newcommand{\Wln}[2]{W\sb{#1}\bup{#2}}
\newcommand{\Wn}[2]{\bW\sp{#1}\bp{#2}}
\newcommand{\Wni}[3]{\Wn{{#1}({#3})}{#2}}
\newcommand{\W}[1]{\Wn{\bullet}{#1}}
\newcommand{\Wi}[2]{\Wni{\bullet}{#1}{#2}}
\newcommand{\Wl}[1]{\Wln{\bullet}{#1}}
\newcommand{\tW}{\widetilde{\bW}}
\newcommand{\tve}[1]{\bv\bp{#1}}
\newcommand{\tWl}{\widetilde{W}}
\newcommand{\tWn}[2]{\tW\quad\hspace*{-4mm}\sp{#1}\bp{#2}}
\newcommand{\tWu}[1]{\tWn{\bullet}{#1}}

\newcommand{\tWln}[2]{\tWl\quad\hspace*{-4mm}\sb{#1}\bup{#2}}

\newcommand{\Wpn}[2]{\mbox{$\ppp\bW$}\sp{#1}\bp{#2}}
\newcommand{\Wp}[1]{\Wpn{\bullet}{#1}}
\newcommand{\vW}[1]{\widehat{\bW}\sp{#1}}
\newcommand{\vWn}[2]{\vW{#1}\bp{#2}}
\newcommand{\vWu}[1]{\vWn{\bullet}{#1}}

\newcommand{\oX}[1]{\overline{\bX}\sp{#1}}
\newcommand{\ooX}[2]{\overline{\Omega\sp{#1}\bX\sp{#2}}}

\newcommand{\Yu}{\bY\sp{\bullet}}

\newcommand{\Zu}{\bZ\sp{\bullet}}
\newcommand{\es}{e\sp{\#}}
\newcommand{\en}[1]{e\bp{#1}}

\newcommand{\eni}[2]{e\bp{#1}\up{#2}}
\newcommand{\enk}[2]{e\bp{#1}\sp{#2}}

\newcommand{\oen}[1]{\overline{e}\sp{#1}}

\newcommand{\oon}[2]{\overline{e}\sp{#1}\sb{#2}}
\newcommand{\Pon}[2]{\overline{Pe}\sp{#1}\sb{#2}}

\newcommand{\fu}[1]{f\sb{#1}}

\newcommand{\jnk}[2]{\overline{j}\bp{#1}\sp{#2}}

\newcommand{\pp}[1]{\overline{p}\sp{#1}}

\newcommand{\ip}[1]{\overline{\iota}\sp{#1}}
\newcommand{\rs}{r\sp{\#}}
\newcommand{\rn}[1]{r\bp{#1}}
\newcommand{\rni}[2]{\rn{#1}\up{#2}}
\newcommand{\rnk}[2]{\rn{#1}\sp{#2}}
\newcommand{\oar}[1]{\overline{r}\,\sp{#1}}
\newcommand{\oor}[2]{\overline{r}\,\sp{#1}\sb{#2}}
\newcommand{\Por}[2]{\overline{Pr}\,\sp{#1}\sb{#2}}

\newcommand{\snk}[2]{\overline{s}\bp{#1}\sp{#2}}

\newcommand{\Tal}[1][ ]{$\Theta$-algebra{#1}}
\newcommand{\pTal}[1][ ]{$\pi\sb{0}\bT$-algebra{#1}}
\newcommand{\TQal}[1][ ]{$\TQ$-algebra{#1}}
\newcommand{\TRal}[1][ ]{$\TR$-algebra{#1}}
\newcommand{\TAlg}{\Alg{\Theta}}

\newcommand{\TRA}{\Alg{\TR}}
\newcommand{\ma}[1][ ]{mapping algebra{#1}}

\newcommand{\Tma}[1][ ]{$\bT$-mapping algebra{#1}}

\newcommand{\lin}[1]{\{{#1}\}}

\newcommand{\fM}{\mathfrak{M}}
\newcommand{\fMT}{\fM\sb{\bT}}

\newcommand{\fX}{\mathfrak{X}}

\newcommand{\fY}{\mathfrak{Y}}

\newcommand{\bbk}{[\mathbf{k}]}
\newcommand{\bbn}{[\mathbf{n}]}
\begin{document}
%
%
\title{Higher cohomology operations and $R$-completion}
\author{David Blanc}
\author{Debasis Sen}
\address{Department of Mathematics\\ University of Haifa\\ 3498838 Haifa\\ Israel}
\email{blanc@math.haifa.ac.il}
\address{Department of Mathematics and Statistics\\
Indian Institute of Technology, Kanpur\\ Uttar Pradesh 208016\\India}
\email{debasis@iitk.ac.in}

\date{\today}


\begin{abstract}
Let $R$ be \w{\Fp} or a field of characteristic $0$.
For each $R$-good topological space $\bY$, we define a collection of higher
cohomology operations which, together with the cohomology algebra \w[,]{\HiR{\bY}}
suffice to determine $\bY$ up to $R$-completion. We also provide a similar
collection of higher cohomology operations which determine when two maps
\w{\fu{0},\fu{1}:\bZ\to\bY} between $R$-good spaces(inducing the same
algebraic homomorphism \w[)]{\HiR{\bY}\to\HiR{\bZ}} are $R$-equivalent.
\end{abstract}

\maketitle

\setcounter{section}{0}

%
%
\section*{Introduction}
\label{cint}

We describe complete sets of invariants for the $R$-homotopy types of $R$-good
topological spaces and maps between them, consisting of systems of higher
$R$-cohomology operations (where $R$ is either \w{\Fp} or a field of characteristic $0$).

Higher homotopy or cohomology operations
(see \cite{TodC,SpanS,SpanH,AdamsN,MaunCO}) should be thought of as
inductively-defined systems of obstructions to rectifying homotopy-commutative
diagrams.
In particular, an \emph{$n$-th order cohomology operation} is attached to
a diagram \w{F:I\to\ho\Top} indexed by a \emph{lattice} in the sense of
\cite[\S 2]{BMarkH} \wh that is, a directed category $I$ of length $n+1$
such that, for all but the initial object \w{i\sb{0}} in $I$, \w{F(i)} is
a product of $R$-module Eilenberg-Mac~Lane spaces (an \emph{$R$-GEM}).

\begin{example}\label{egtoda}
The simplest example is a Toda bracket (see \cite{BBGondH}), for a diagram
of the form
\mydiagram[\label{eqtoda}]{
\bY \ar@/^{2.0pc}/[rr]\sp{\ast} \ar[r]\sp{f} &
\bW\sb{0} \ar@/_{1.2pc}/[rr]\sb{\ast} \ar[r]\sp{g} &
\bW\sb{1} \ar[r]\sp{h} & \bW\sb{2}~,
}
\noindent with \w[,]{\bW\sb{0}} \w[,]{\bW\sb{1}} and \w{\bW\sb{2}} $R$-GEMs,
and each adjacent composition nullhomotopic (see \cite{AdamsN,HarpSC}).

This defines a secondary cohomology operation in the sense of \cite{AdamsN},
since $f$ represents a set of cohomology classes in \w[,]{\HiR{\bY}}
on which the set of primary $R$-cohomology operations
represented by $g$ vanish. The fact that \w{h\circ g\sim\ast} indicates a relation
among primary operations.

Our general strategy for rectification of any \w{F:I\to\ho\Top} is to inductively
rectify, and then make fibrant, longer and longer final segments of the given
diagram.
In our example, we first make \w{\bW\sb{2}} fibrant and change $h$ into a fibration
(so the subdiagram \w{\bW\sb{1}\xra{h}\bW\sb{2}} is fibrant). We then change $g$ up
to homotopy so that \w[,]{h\circ g=\ast} using \cite[Lemma 5.11]{BJTurnR}.
Factoring $g$ as \w[,]{\bW\sb{0}\xra{k'}\Fib(h)\stackrel{i}{\hra}\bW\sb{1}} and then
changing \w{k'} into a fibration $k$ makes
\w{\bW\sb{0}\xra{g}\bW\sb{1}\xra{h}\bW\sb{2}} fibrant. To simplify notation we
denote \w{i\circ k} simply by \w[.]{g:\bW\sb{0}\to\bW\sb{1}}

We think of the following solid diagram of vertical and horizontal fibration
sequences as the \emph{template} for our Toda bracket (depending only on
\w[):]{\bW\sb{0}\xra{g}\bW\sb{1}\xra{h}\bW\sb{2}}
\mydiagram[\label{eqtodasq}]{
\bY \ar@{.>}[rrd]\sp(0.7){\exists ?\psi} \ar@/_1.5pc/@{-->}[rrdd]\sb{f}
\ar@/^1.5pc/@{-->}[rrrrd]\sp{\varphi} && &&\\
&& \Fib(k) \ar@{^{(}->}[rr]\sb{\ell} \ar@{^{(}->}[d] &&
\Fib(\widetilde{g}) \ar@{^{(}->}[d] \ar@{->>}[rr]\sb{q} &&
\Omega\bW\sb{2} \ar@{^{(}->}[d] \\
&& \bW\sb{0} \ar@{->>}[d]\sp{k} \ar@{^{(}->}[rr]\sp{j}\sb{\simeq}
\ar[rrd]\sp{g} && \bW'\sb{0} \ar@{->>}[d]\sp{\widetilde{g}}
\ar@{->>}[rr]\sp{G} && P\bW\sb{2} \ar@{->>}[d]\sp{p} \\
&& \Fib(h) \ar@{^{(}->}[rr]\sp{i} && \bW\sb{1} \ar@{->>}[rr]\sp{h} && \bW\sb{2}
}
\noindent The nullhomotopy $G$ exists since \w{h\circ \widetilde{g} \circ j = g\circ h=\ast} and $j$ is
a weak equivalence.

Now, any map \w{f:\bY\to\bW\sb{0}} with \w{g\circ f\sim\ast}
factors up to homotopy through a map \w[.]{\varphi:\bY\to\Fib(\widetilde{g})}
To rectify \wref[,]{eqtoda} $\varphi$ and $f$ should induce a map
\w[.]{\psi:\bY\to\Fib(k)=\Fib(q)}
The obstruction to doing so \wh namely the homotopy class of the composite
\w{q\circ\varphi:\bY\to\Omega\bW\sb{2}} \wwh is called
the \emph{value} of the Toda bracket \w{\lra{f,g,h}} (for the given choices of
$k$ and $\varphi$).
\end{example}

\begin{mysubsection}{The basic construction}\label{sbasicc}
Our object in this paper is to associate to each $R$-good space $\bY$ a sequence
\w{\llrra{\bY}=(\llrr{\bY}{n})\sb{n=2}\sp{\infty}} of higher cohomology
operations which serve as a complete set of invariants for the $R$-homotopy
type of $\bY$, constructed roughly as follows\vsn:

\noindent\textbf{(a)} \ We start with \w{\HiR{\bY}} as a \emph{\TRal} \wh i.e., a
graded $R$-algebra with an action of the primary $R$-cohomology operations (such as
Steenrod squares). We then choose a \emph{CW resolution} of this \TRal[,] given
by an inductively defined simplicial \TRal \w[,]{\Vd} with the $n$-th stage obtained
from the \wwb{n-1}truncation by attaching a free \TRal \w{\oV{n}} along a map
\w{\odz{n}:\oV{n}\to V\sb{n-1}} (see \S \ref{dscwo} below).
Our goal is to realize \w{\Vd\to\HiR{\bY}} by a coaugmented cosimplicial
space \w[,]{\bY\to\Wu=\lim\sb{n}\,\W{n}} for which \w{\bY\to\Tot\Wu}
is the $R$-completion map of \cite{BKanH}\vsn.

\noindent\textbf{(b)} \ To carry out this program, we use a double induction:  if
\w{\W{n-1}} realizes \w{\Vd} through simplicial dimension \w[,]{n-1} we can choose
a map
\w{\udz{n-1}:\Wn{n-1}{n-1}\to\oW{n}} realizing \w{\odz{n}:\oV{n}\to V\sb{n-1}}
(with \w{\oW{n}} an $R$-GEM). Then \w{\udz{n-1}\circ d\sp{0}} is nullhomotopic by a
nullhomotopy \w[,]{\Fk{n-2}:\Wn{n-2}{n-1}\to P\oW{n}} with \w{\Fk{n-2}\circ d\sp{0}}
factoring through \w[,]{\ak{n-3}:\Wn{n-3}{n-1}\to\Omega\oW{n}} as in
Example \ref{egtoda} above. One can in fact choose \w{\Fk{n-2}} so that
\w{\ak{n-3}} is nullhomotopic\vsn.

\noindent\textbf{(c)} \ Step \textbf{(b)} can be repeated, and the $k$-th
obstruction \w{\ak{n-k}} again turns out to be nullhomotopic. We end up with an
$n$-th order cohomology operation \w{\llrr{\bY}{n}} with value
\w[.]{\ak{-1}:\bY\to\Omega\sp{n-1}\oW{n}} The fact that this too vanishes
allows us to extend \w{\W{n-1}} to a cosimplicial space \w{\W{n}} (which will be
$n$-coskeletal, up to homotopy). We call the system \w{\cW=(\W{n})\sb{n\in\NN}} a
\emph{sequential realization} of \w{\Vd} for the space $\bY$, and think of it as a
template for a sequence \w{\llrra{\bY}=(\llrr{\bY}{n})\sb{n=2}\sp{\infty}}
of higher cohomology operations\vsn.

\noindent\textbf{(d)} \ Finally, given a space $\bZ$ with
\w[,]{\HiR{\bY}\cong\HiR{\bZ}} the augmentation of \TRal[s]
\w{\vare:V\sb{0}\to\HiR{\bZ}} can be realized by a map \w[.]{\bve{0}:\bZ\to\Wn{0}{0}}
If we can extend this to a coaugmentation \w[,]{\bv:\bZ\to\Wu:=\lim\W{n}} we would
obtain an $R$-equivalence between $\bZ$ and $\bY$. The obstruction to extending
the \wwb{n-1}st approximation \w{\bve{n-1}:\bZ\to\W{n}} for $\bv$ to \w{\bve{n}} is
a map \w{\ak{-1}:\bZ\to\Omega\sp{n-1}\oW{n}} as above \wh the value
associated to $\bZ$ for the $n$-th order cohomology operation \w[.]{\llrr{\bY}{n}}
\end{mysubsection}

\begin{mysubsection}{Main results}\label{smainresults}
In order to apply the machinery described above, we need the following important
technical result:

\begin{thma}
Any CW resolution \w{\Vd} of the \TRal \w{\HiR{\bY}} can be realized by
a coaugmented cosimplicial space \w{\bY\to\Wu} with each
\w{\bW\sp{n}} an $R$-GEM, obtained as the limit of a sequential
realization as above.
\end{thma}
\noindent See Theorem \ref{tres} below.  This allows us to produce various
templates for the system of higher cohomology operations \w[,]{\llrra{\bY}}
based on the algebraic resolution of our choice\vsn .

The sequence of higher cohomology operations presented
here is dual to the higher homotopy operations of \cite{BlaH,BJTurnHA},
which correspond to the Andr\'{e}-Quillen cohomology obstructions for
distinguishing between different realizations \w{\pi\sb{\ast}\bY} of a
$\Pi$-algebra (see \cite{BDGoeR}).
Such Andr\'{e}-Quillen classes appear also in the dual context of
distinguishing between different realizations of a given abstract \TRal
$\Gamma$ (see \cite{BlaS,BRStelR}). However, the higher cohomology operations
of \cite{BlaS} correspond to Andr\'{e}-Quillen \emph{cocycles} (for a specific
algebraic resolution \w{\Vd} of \w[).]{\HiR{\bY}} We therefore would like to collect
together the various higher order cohomology operations corresponding to a given
Andr\'{e}-Quillen cohomology class. We do this by means of suitable (split)
weak equivalences called \emph{comparison maps} between sequential realizations
for various algebraic resolutions, and show:

\begin{thmb}
Any two sequential realizations of two CW resolutions \w{\Vd}
for the same space $\bY$ are connected by a zigzag of comparison maps.
\end{thmb}
\noindent See Theorem \ref{tcomp} below\vsm.

Our two main results may then be summarized as follows:

\begin{thmc}
For \w{R=\Fp} or a field of characteristic $0$, let $\bY$ and $\bZ$ be $R$-good
spaces and \w{\vth:\HiR{\bY}\to\HiR{\bZ}} an isomorphism of  \TRal[s.]
Then $\vth$ is realizable by a zigzag of $R$-equivalences between $\bY$ and $\bZ$
if and only if the system of higher operations associated to this initial data
vanishes.
\end{thmc}
\noindent See Theorem \ref{tvanish} below\vsm.

By extending the ideas sketched in \S \ref{sbasicc}, one can use any
sequential realization for $\bY$ to define a system of higher
cohomology operations associated to any two maps \w{\fu{0},\fu{1}:\bZ\to\bY}
which induce the same map in cohomology; although the construction is more
complicated, these operations still take values in the groups
\w[,]{[\bZ,\,\Omega\sp{n-1}\oW{n}]} and we have:

\begin{thmd}
For \w{R=\Fp} or a field of characteristic $0$, let  \w{\fu{0},\fu{1}:\bZ\to\bY}
be two maps between  $R$-good spaces which induce the
same morphism of \TRal[s] \w[.]{\HiR{\bY}\to\HiR{\bZ}}
Then \w{\fu{0}} is $R$-equivalent to \w{\fu{1}}
if and only if the associated system of higher operations vanishes.
\end{thmd}
\noindent See Theorem \ref{tvanishmap} below.
\end{mysubsection}

\begin{mysubsection}{Organization}\label{sorg}
Section \ref{crmccs} provides some background material on (co)simplicial
constructions and on sketches and their algebras. In Section \ref{crstr} we
define and study sequential realizations of algebraic resolutions. In Section
\ref{ccsr} we show how any two such sequential realizations may be connected by
a zigzag of comparison maps.
In Section \ref{chhc} we construct the higher cohomology operations
used to distinguish between spaces, including a detailed rational example in
\S \ref{sratex}, while in Section \ref{chhim} we define the analogous
invariants for maps.

Appendix \ref{apfthm} reviews the notions of enriched sketches and their
mapping algebras, which are used to generalize and prove Theorem A.
\end{mysubsection}

\begin{ack}
We would like to thank the referee for his or her careful and pertinent comments.

The research of the first author was supported by Israel Science
Foundation grants 74/11 and 770/16, and that of the second author by INSPIRE
grant No.~IFA MA-12.
\end{ack}

%
%
\sect{Background}
\label{crmccs}

In this section we present some background material on (co)simplicial theory and
algebraic theories that will be used throughout the paper\vsm.

\begin{defn}\label{dscso}
Let $\Del$ denote the category of finite ordered sets and order-preserving
maps (see \cite[\S 2]{MayS}), and \w{\Dp} the subcategory with the same objects,
but only monic maps. A \emph{cosimplicial object} \w{\Gu} in a
category $\C$ is a functor \w[,]{\Del\to\C} and a \emph{restricted}
cosimplicial object is a functor \w[.]{\Dp\to\C} More concretely,  we write
\w{G\sp{n}} for the value of \w{\Gu} at the ordered set \w[.]{\bbn=(0<1<\dotsc<n)}
The maps in the diagram \w{\Gu} are generated by the \emph{coface} maps
\w{d\sp{i}=d\sp{i}\sb{n}:G\sp{n}\to G\sp{n+1}} \wb[,]{0\leq i \leq n+1} as well
as \emph{codegeneracy} maps \w{s\sp{j}=s\sp{j}\sb{n}:G\sp{n}\to G\sp{n-1}}
\wb{0 \leq j<n} in the non-restricted case, satisfying the usual cosimplicial
identities. Dually, a \emph{simplicial object} \w{\Gd} in $\C$ is a functor
\w[.]{\Del\op\to\C} The category of cosimplicial objects over
$\C$ will be denoted by \w[,]{\C\sp{\Del}} and the category of
simplicial objects over $\C$ will be denoted by \w[.]{\C\sp{\Del\op}}
However, the category of simplicial sets is denoted simply by
\w[,]{\cS=s\Set}and  that of pointed simplicial sets by \w[.]{\Sa=s\Seta}
By a \emph{space} we always mean a pointed simplicial set.

There are natural embeddings \w{\cu{-}:\C\to\C\sp{\Del}} and
\w[,]{\cd{-}:\C\to\C\sp{\Del\op}}
defined by letting \w{\cu{A}} denote the constant cosimplicial object
which is $A$ in every cosimplicial dimension, and similarly for \w[.]{\cd{A}}
\end{defn}

\begin{mysubsection}{Latching and matching objects}
\label{slmo}
For a cosimplicial object \w{\Gu\in\C\sp{\Del}} in a complete category $\C$,
the $n$-th \emph{matching object} for \w{\Gu} is
\begin{myeq}\label{eqmatch}
M\sp{n}\Gu~:=~\lim\sb{\phi:\bbn\to\bbk}\,G\sp{k}~,
\end{myeq}
\noindent where $\phi$ ranges over the non-identity surjective maps \w{\bbn\to\bbk} in
$\Del$. There is a natural map \w{\zeta\sp{n}:G\sp{n}\to M\sp{n}\Gu} induced
by the structure maps of the limit, and any iterated codegeneracy map
\w{s\sp{I}=\phi\sb{\ast}:G\sp{n}\to G\sp{k}} factors as
\begin{myeq}\label{equnivcod}
s\sp{I}~=~\proj\sb{\phi}\circ\zeta\sp{n}~,
\end{myeq}
\noindent where \w{\proj\sb{\phi}:M\sp{n}\Gu\to G\sp{k}} is the structure
map for the copy of \w{G\sp{k}} indexed by $\phi$ (see \cite[X,\S 4.5]{BKanH}).

Note that the inclusion \w{\Del\sb{+}\hra\Del} induces a forgetful functor
\w[,]{\U:\C\sp{\Del}\to\C\sp{\Del\sb{+}}} and its right adjoint
\w{\F:\C\sp{\Del\sb{+}}\to\C\sp{\Del}} is given by
\w[,]{(\F\Gu)\sp{n}=G\sp{n}\times M\sp{n}\Gu} with codegeneracies given by
\wref{equnivcod} and coface maps by the cosimplicial identities\vsn.

The $n$-th \emph{latching object} for \w{\Gu\in\C\sp{\Del}} is
\w[,]{L\sp{n}\Gu:=\colim\sb{\theta:\bbk\to\bbn}\,G\sp{k}}
where the maps $\theta$ are now non-identity injective maps, with
\w{\sigma\sp{n}:L\sp{n}\Gu\to G\sp{n}} defined by the structure maps.

These two constructions have analogues for a simplicial object \w{\Gd} over a
(co)comp\-lete category $\C$: the \emph{latching object}
\w[,]{L\sb{n}\Gd:=\colim\sb{\theta:\bbk\to\bbn}G\sb{k}} and the
\emph{matching object} \w[,]{M\sb{n}\Gd:=\lim\sb{\phi:\bbn\to\bbk} G\sb{k}}
equipped with the obvious canonical maps.
\end{mysubsection}

\begin{defn}\label{dmco}
Let $\C$ be a pointed category.  If it is complete, the $n$-th \emph{Moore chain}
object of \w{\Gd\in\C\sp{\Del}} is
\begin{myeq}\label{eqmoor}
\cMl{n}\Gd~:=~\cap\sb{i=1}\sp{n}\Ker\{d\sb{i}:G\sb{n}\to G\sb{n-1}\}~,
\end{myeq}
\noindent with differential \w[.]{\partial\sb{n}\sp{\Gd}=\partial\sb{n}:=
(d\sb{0})\rest{\cMl{n}\Gd}:\cMl{n}\Gd\to\cMl{n-1}\Gd}
The $n$-th \emph{Moore cycles} object is
\w[.]{\cZl{n}\Gd:=\Ker(\partial\sb{n}\sp{\Gd})}

Dually, if $\C$ is cocomplete, the $n$-th \emph{Moore cochain} object of
\w{\Gu\in\C\sp{\Del\op}} is
\begin{myeq}\label{eqmoorecc}
\cM{n}\Gu ~:=~\Coker(\coprod\sb{i=1}\sp{n-1}\,G\sp{n}~\xra{\bot\sb{i}\,d\sp{i}}
~G\sp{n})
\end{myeq}
\noindent with differential \w{\delta\sp{n-1}:\cM{n-1}\Gu\to\cM{n}\Gu}
induced by \w[,]{d\sp{0}\sb{n-1}} and structure map
\w[.]{v\sp{n}:G\sp{n}\to\cM{n}\Gu}

We denote the cofiber of \w{\delta\sp{n-1}} by \w[,]{\cZ{n}\Gu}
with structure map \w[.]{w\sp{n}:\cM{n}\Gu\to\cZ{n}\Gu}
\end{defn}

\begin{mysubsection}{Cochain complexes in $\C$}
\label{scchaincx}
In general, a \emph{cochain complex} in $\C$ is a commuting diagram \w{\As}
of the form
\mydiagram[\label{eqchaincx}]{
A\sp{0}\ar[r]\sp{\dif{0}} \ar[rd] |!{[d];[r]}\hole &
A\sp{1} \ar[r]\sp{\dif{1}} \ar[rd] |!{[d];[r]}\hole &
A\sp{2}\dotsc &
\dotsc A\sp{n-1} \ar[r]^(0.65){\dif{n-1}} \ar[rd] |!{[d];[r]}\hole &
A\sp{n}\ar[r]\sp{\dif{n}} \ar[rd] |!{[d];[r]}\hole &  A\sp{n+1}\dotsc\\
0 \ar[ru] & 0 \ar[ru] & 0\dotsc & \dotsc 0\ar[ru]  & 0\ar[ru]  & 0\dotsc
}
\noindent (so \w{\dif{n}\circ\dif{n-1}=0} for all $n$).

We let \w{\Ch{\C}} denote the category of non-negatively graded cochain complexes
over $\C$, and by \w{\Chn{\C}{n}} the category of \emph{$n$-truncated}
cochain complexes \w{\As} in $\C$ (for which \w{A\sp{i}=0} unless
\w[).]{0\leq i\leq n}

The category \w{\Cha\sb{\C}} of non-negatively graded chain complexes
over $\C$ is defined analogously.

The Moore cochain functor \w{\cMs:\C\sp{\Del\sb{+}}\to\Ch{\C}} has a right
adjoint (and left inverse) \w{\E:\Ch{\C}\to\C\sp{\Del\sb{+}}} with
\w[,]{(\E\A\sb{\ast})\sp{n}=A\sp{n}} \w[,]{d\sp{0}\sb{n}=\dif{n}} and
\w{d\sp{i}\sb{n}=0} for \w[.]{i\geq 1} This holds also for \w{\Chn{\C}{n}}
if we truncate \w[,]{\C\sp{\Del\sb{+}}} too.

When $\C$ is a model category, we have several possible model category structures on
\w[,]{\C\sp{\Del}} \w[,]{\C\sp{\Del\sb{+}}} \w[,]{\Ch{\C}} and
\w{\Chn{\C}{n}} (see, e.g., \cite[\S 15.3]{PHirM}, \cite{BousC}, and
\cite[\S 12]{CScheH}).

In particular, \w{\Chn{\C}{n}} has two different Reedy model
category structures, depending on how we choose the degrees in \wref[:]{eqchaincx}
in both cases, the weak equivalences are defined levelwise. In the \emph{right
Reedy} model structure, fibrations are also defined levelwise, and an
$n$-cochain complexes \w{\As} is cofibrant if for each \w{k\leq n} the natural map
\w{\cZ{k}\As:=\Cof(\dif{k-1})\to A\sp{k+1}} is a cofibration
(with \w[).]{A\sp{-1}:=\ast}
In the \emph{left Reedy} model structure, cofibrations are defined levelwise, and an
$n$-cochain complex \w{\As} is fibrant if for each \w{k\leq n} the natural map
\w{A\sp{k}\to\Ker(\dif{k+1})} is a fibration (with \w[).]{A\sp{-1}:=\ast}
Evidently, right Reedy cofibrancy implies left Reedy cofibrancy, and
left Reedy fibrancy implies right Reedy fibrancy.

Note that the Moore cochains functor \w{\cMs:\Sa\sp{\Del}\to\Ch{\Sa}}
preserves cofibrancy and weak equivalences among cofibrant objects in the
(right) Reedy model structures (see \cite[X, Proposition 6.3]{BKanH}).
\end{mysubsection}

By analogy with the usual fiber/cone construction for (co)chain complexes we have:

\begin{mysubsect}{Fibers and cones}
\label{dchaincof}

\begin{enumerate}
\renewcommand{\labelenumi}{(\roman{enumi})~}
\item For any map \w{f:\Au\to\Bu} in \w{\C\sp{\Del\sb{+}}}
we define the restricted cosimplicial object \w{\Cu=\Fib(f)}
by setting \w[,]{C\sp{n}:=A\sp{n}\times B\sp{n-1}} with coface maps:
$$
\xymatrix@C=6pt{
C\sp{n+1} & = & A\sp{n+1} & \times & B\sp{n} &&
C\sp{n+1} & = & A\sp{n+1} & \times & B\sp{n} && \quad\\
C\sp{n} \ar[u]\sp{d\sp{0}} & =  & A\sp{n} \ar[u]\sp{d\sp{0}} \ar[urr]\sp{f\sp{n}} &
\times & B\sp{n-1} &\quad \ar @{} [u] |{\mbox{and}}&
C\sp{n} \ar[u]\sp{d\sp{i}} & = & A\sp{n} \ar[u]\sp{d\sp{i}} &
\times & B\sp{n-1} \ar[u]\sp{d\sp{i-1}} &&
\quad \ar @{} [u] |{\mbox{for \ $i\geq 1$ }}
}
$$
\noindent and a natural projection \w[.]{\ell:\Cu\to\Au} For
\w{d\sp{j}\circ d\sp{0}:C\sp{n-1}\to C\sp{n+1}} we verify that
\begin{equation*}
\prj\sb{B\sp{n}}\circ d\sp{j}\circ d\sp{0}~=~
d\sp{j-1}\sb{B\sp{n-1}}\circ f\sp{n-1}\circ\prj\sb{A\sp{n-1}}~=~
f\sp{n}\circ d\sp{j-1}\sb{A\sp{n-1}}\circ \prj\sb{A\sp{n-1}}
~=~\prj\sb{B\sb{n}}\circ d\sp{0}\circ d\sp{j-1}
\end{equation*}
\noindent for all \w[,]{j>0} while clearly
\w{\prj\sb{B\sp{n}}(d\sp{j}\circ d\sp{i})=\prj\sb{B\sp{n}}(d\sp{i}\circ d\sp{j-1})}
for \w[\vsn.]{1\leq i<j}
\item Similarly, for a map \w{f:\Ad\to\Bd} in \w{\C\sp{\Del\op\sb{+}}}
we define the restricted simplicial object \w{\Cd=\Cone(f)}
by setting \w[,]{C\sb{n}:=A\sb{n-1}\amalg B\sb{n}} with
$$
d\sb{i}\sp{C\sb{n}}~:=~
\begin{cases} f\sb{n-1}\bot\, d\sb{0}\sp{B\sb{n}} & \text{if}~i=0\\
d\sb{i-1}\sp{A\sb{n-1}}\bot\, d\sb{i}\sp{B\sb{n}} & \text{if}~i\geq 1~,
\end{cases}
$$
\noindent and a natural inclusion \w[.]{m:\Bd\hra\Cd}
\end{enumerate}
\end{mysubsect}

\begin{mysubsection}{Simplicial CW objects}
\label{dscwo}
A simplicial object \w{\Gd\in\C\sp{\Del\op}} over a pointed category $\C$
is called a \emph{CW object} if it is equipped with a \emph{CW basis}
\w{(\oG{n})\sb{n\in\NN}} in $\C$ such that
\w[,]{G\sb{n}=\oG{n}\amalg L\sb{n}\Gd} and \w{d\sb{i}\rest{\oG{n}}=0}
for \w[.]{1\leq i\leq n} In this case
\w{\odz{G\sb{n}}:=d\sb{0}\rest{\oG{n}}:\oG{n}\to G\sb{n-1}} is called the
\emph{attaching map} for \w[.]{\oG{n}} By the simplicial identities \w{\odz{G\sb{n}}}
factors as
\begin{myeq}\label{eqattach}
\odz{G\sb{n}}:\oG{n}~\to~Z\sb{n-1}\Gd~\subset~G\sb{n-1}~.
\end{myeq}
\noindent Note that we have an explicit formula
\begin{myeq}\label{eqslatch}
L\sb{n}\Gd~:=~
\coprod\sb{0\leq k<n}~\coprod\sb{0\leq i\sb{1}<\dotsc<i\sb{n-k}\leq n-1}~
\oG{k}~,
\end{myeq}
\noindent where the iterated degeneracy map
\w[,]{s\sb{i\sb{n-k}}\dotsc s\sb{i\sb{2}}s\sb{i\sb{1}}} restricted to the
basis object \w[,]{\oG{k}} is the inclusion into the copy of \w{\oG{k}} indexed by
\w[\vsn.]{(i\sb{1},\dotsc,i\sb{n-k})}
\end{mysubsection}

\begin{remark}\label{rscwo}
Given \w[,]{\oG{}\in\C} define \w{\oG{}\olS{n}} in \w{\Cha\sb{\C}}
be the chain complex with $\oG{}$ in dimension $n$ and $\ast$ elsewhere.
A CW object \w{\Gd} over $\C$ with CW basis \w{(\oG{n})\sb{n=0}\sp{\infty}} is
the colimit of an inductively constructed sequence of skeleta
\w[,]{\sk{0}\Gd\hra\sk{1}\Gd\hra\dotsc} in the usual sense
(see \cite[VII, \S 1]{GJarS}), starting with
\w[.]{\sk{0}\Gd:=\cd{\oG{0}}}

To do so, note that the attaching map \w{\odz{G\sb{n}}:\oG{n}\to Z\sb{n-1}\Gd} defines
a chain map \w[,]{\phi:\oG{n}\olS{n-1}\to\cMls\sk{n-1}\Gd} which has an
adjoint \w[,]{\widetilde{\phi}:\E'(\oG{n}\olS{n-1})\to\U'\sk{n-1}\Gd}
where \w{\E':\Cha\sb{\C}\to\C\sp{\Del\op\sb{+}}} is left adjoint to the Moore
chain functor \w{\cMls:\C\sp{\Del\op\sb{+}}\to\Cha\sb{\C}} and
\w{\U':\C\sp{\Del\op}\to\C\sp{\Del\sb{+}\op}} is the forgetful functor
(compare \S \ref{scchaincx}).

Note that \w{\U'} has a left adjoint \w{\F':\C\sp{\Del\sb{+}\op}\to\C\sp{\Del\op}}
given by \w[,]{(\F'\Gd)\sb{n}=G\sb{n}\amalg L\sb{n}\Gd} (compare \S \ref{slmo}).
If \w{\vartheta:\F'\U'\to\Id} is the counit for the adjunction, we have
\w{\sk{n}\Gd} as the pushout in:
\mydiagram[\label{eqnskeleton}]{
\ar @{} [drr]|(0.71){\framebox{\scriptsize{PO}}}
\F'\U'\sk{n-1}\Gd \ar[d]\sb{\F'm} \ar[rr]\sp{\vartheta} &&
\sk{n-1}\Gd \ar[d] \\
\F'\Cone(\widetilde{\phi}) \ar[rr] && \sk{n}\Gd~,
}
\noindent for $m$ as in \S \ref{dchaincof}(b).
\end{remark}

\begin{mysubsection}{Cosimplicial CW objects}
\label{dccwo}
A cosimplicial CW object \w{\Gu\in\C\sp{\Del}} with CW basis
\w{(\uuG{n})\sb{n\in\NN}} may be defined analogously as the limit of a tower of
coskeleta \w{\dotsc\to\csk{2}\Gu\to\csk{1}\Gu\to\csk{0}\Gu} (see
\cite[\textit{loc.\ cit.}]{GJarS}), starting with
\w[,]{\csk{0}\Gu:=\cu{\uuG{0}}} by thinking of its attaching maps
as a cochain map \w{\varphi:\cMs\csk{n-1}\Gu\to\uuG{n}\ouS{n-1}}
(where \w{\uuG{n}\ouS{n-1}} is the cochain complex with \w{\uuG{n}} in dimension
\w{n-1} and zero elsewhere).

The map $\varphi$ has an adjoint
\w[,]{\widehat{\varphi}:\U\csk{n-1}\Gu\to\E(\uuG{n}\ouS{n-1})}
and we \w{\csk{n}\Gu} as the pullback in
\mydiagram[\label{eqncoskeleton}]{
\ar @{} [drr]|(0.24){\framebox{\scriptsize{PB}}}
\csk{n}\Gu \ar[d] \ar[rr] && \F\Fib(\widehat{\varphi})\ar[d]\sp{\F\ell} \\
\csk{n-1}\Gu \ar[rr]\sp{\theta} && \F\U \csk{n-1}\Gu~,
}
\noindent with $\theta$ the unit for \w{\F\U} and $\ell$ as in
\S \ref{dchaincof}(i), using the notation of \S \ref{slmo}.
\end{mysubsection}

\begin{mysubsect}{List of functors}
\label{slistfunc}

For the reader's convenience we list the main functors we have defined for simplicial
and cosimplicial objects in a category $\C$:

\begin{enumerate}
\renewcommand{\labelenumi}{(\alph{enumi})~}
\item The Moore cochain complex functor \w[,]{\cMs:\C\sp{\Del}\sb{+}\to\Ch{\C}}
and its right adjoint (and left inverse) \w[.]{\E:\Ch{\C}\to\C\sp{\Del\sb{+}}}
\item The Moore chain complex functor \w[,]{\cMls:\C\sp{\Del\op\sb{+}}\to\Cha\sb{\C}}
and its left adjoint (and right inverse) \w[.]{\E':\Cha\sb{\C}\to\C\sp{\Del\op\sb{+}}}
\item The forgetful functor
\w[,]{\U:\C\sp{\Del}\to\C\sp{\Del\sb{+}}} and its right
adjoint \w{\F:\C\sp{\Del\sb{+}}\to\C\sp{\Del}} (adding codegeneracies).
\item  The forgetful functor \w[,]{\U':\C\sp{\Del\op}\to\C\sp{\Del\op\sb{+}}}
and its left adjoint \w{\F':\C\sp{\Del\op\sb{+}}\to\C\sp{\Del\op}}
(adding degeneracies).
\end{enumerate}

When there is no danger of confusion, we denote \w{\cMs(\U(\Wu))} simply by
\w[,]{\cMs\Wu} and \w{\cMls(\U'(\Gd))} by \w[.]{\cMls\Gd}
\end{mysubsect}

\begin{defn}\label{dsmcat}
A \emph{simplicial} category $\C$ (in the sense of Quillen) is one
in which, for each (finite) \w{K\in\cS} and \w[,]{\bX\in\C} we have objects
\w{\bX\otimes K} and \w{\bX\sp{K}} in $\C$ equipped with appropriate
adjunction-like isomorphisms. In particular, such categories are simplicially
enriched.
A \emph{simplicial model category} is a simplicial category with a model category
structure satisfying axiom SM7 (see \cite[II, \S 1-2]{QuiH}). The basic examples
are $\cS$ and \w[.]{\Sa}
\end{defn}

\begin{assume}\label{amodel}
From now on $\C$ will be a pointed simplicial model category in which all objects
are cofibrant \wh so in particular it is left proper.

The main example we shall be concerned with is \w[,]{\C=\Sa} so we shall
sometimes refer to the objects of $\C$\wh denoted by boldface letters $\bX$, $\bY$,
and so on \wh as ``spaces''.
\end{assume}

\begin{mysubsection}{$\G$-resolution model structure}
\label{sgrmc}
Let $\G$ be a class of homotopy group objects in a model category $\C$ as above,
closed under loops.
A map \w{i:\bA\to\bB} in \w{\ho\C} is called $\G$-\emph{monic} if
\w{i\sp{\ast}:[\bB,\bG]\to [\bA,\bG]} is onto for each \w[.]{\bG \in \G}
An object $\bY$ in \w{\C} is called \emph{$\G$-injective} if
\w{i\sp{\ast}:[\bB,\bY]\to [\bA,\bY]} is onto for each $\G$-monic map
\w{i:\bA\to\bB} in \w[.]{\ho\C} A fibration in $\C$ is called
\emph{$\G$-injective} if it has the right lifting property for the
$\G$-monic cofibrations in $\C$.

The homotopy category \w{\ho\C} is said to have \emph{enough $\G$-injectives}
if each object is the source of a $\G$-monic map to a $\G$-injective target.
In this case $\G$ is called a class of \emph{injective models} in \w[.]{\ho\C}

Recall that a homomorphism in the category \w{s\Grp} of simplicial groups
is a weak equivalence or fibration when its underlying map in \w{\Sa} is such.

A map \w{f:\Wu\to\Yu} in \w{\C\sp{\Del}} is called a $\G$-\emph{equivalence} if
\w{f\sp{\ast}:[\Yu,\bG]\to[\Wu,\bG]} is a weak equivalence in \w{s\Grp} for
each \w[.]{\bG\in \G}
In \cite[Theorem 3.3]{BousC}, Bousfield showed that if $\G$ is a class of
injective models in \w[,]{\ho(\C)} then \w{\C\sp{\Del}} has a left proper pointed
simplicial model category structure with such maps as weak equivalences.
\end{mysubsection}

\begin{mysubsection}{$\G$-completion}
\label{sgcge}
Given a class $\G$ of injective models in $\C$ as above,
a \emph{$\G$-resolution} of an object \w{\bY\in\C} is a
$\G$-fibrant \w{\Wu} equipped with a $\G$-trivial cofibration
\w{\cu{\bY}\hra\Wu} (see \S \ref{dscso}). In this case \w{\LG\bY:=\Tot\Wu} is
called the \emph{$\G$-completion} of $Y$, where \w{\Tot\Wu\in\C} is constructed
as in \cite[\S 2.8]{BousC}.  Moreover, a map \w{f:\bY\to\bZ} in $\C$ is
a $\G$-equivalence if and only if \w{\LG f:\LG\bY\to\LG\bZ} is a weak equivalence
in $\C$. An object \w{\bY\in\C} is called $\G$-\emph{complete} if
\w{\bY\to\LG\bY} is a weak equivalence, and it is
\emph{$\G$-good} if \w{\bY\to\LG\bY} is a $\G$-equivalence \wh so \w{\LG\bY}
is $\G$-complete (see \cite[\S 8]{BousC}).
We say that two maps \w{\fu{0},\fu{1}:\bZ\to\bY} between $\G$-good objects in
$\C$ are \emph{$\G$-equivalent} if \w{\LG\fu{0}\sim\LG\fu{1}:\LG\bZ\to\LG\bY} are
homotopic (for a suitable fibrant and cofibrant model of the $\G$-completion).

A cosimplicial object \w{\Wu\in\C\sp{\Del}} is called \emph{weakly $\G$-fibrant} if
it is Reedy fibrant (see \cite[\S 15.3]{PHirM}), and every \w{\bW\sp{n}} is in
$\G$ \wb[.]{n\geq 0} A \emph{weak $\G$-resolution} of an object \w{\bY\in\C} is a weakly
$\G$-fibrant \w{\Wu} which is $\G$-equivalent to \w[.]{\cu{\bY}} In this case
there is a natural weak equivalence \w[,]{\LG\bY\to\Tot\Wu}
by \cite[Theorem 6.5]{BousC}.
\end{mysubsection}

\begin{example}\label{eggth}
The example we have in mind is the class \w{\G=\G\sb{R}} of all $R$-GEMs in
\w[,]{\C=\Sa} for some ring $R$. In this case the $\G$-completion is the
Bousfield-Kan $R$-completion (see \cite{BKanH}).
\end{example}

\begin{mysubsect}{Sketches and their algebras}
\label{salgsk}

An (FP-)\emph{sketch} in the sense of \cite{EhreET} (cf.\ also \cite{LawvF}) is a small
category $\Theta$ with a specified set $\PP$ of (finite) products. This generalizes
Lawvere's notion of a \emph{theory}, which requires that \w[.]{\Obj(\Theta)=\NN}
A \emph{\Tal} (or $\Theta$-model) is a functor \w{\Gamma:\Theta\to\Seta} preserving
the products in $\PP$, with natural transformations as model morphisms
(see \cite[\S 4.1]{BorcH2}). We write \w{\Gamma\lin{\bB}} for the value of $\Gamma$
at \w[.]{\bB\in\Theta} The category of \Tal[s] is denoted by \w[.]{\TAlg}
Examples of categories of such models include  varieties
of universal algebras, such as groups, rings, and modules over rings.
\end{mysubsect}

\begin{remark}\label{rskhtpy}
In homotopy theory such theories often arise by choosing a collection $\eA$ of
objects in a model category $\C$ (as in \S \ref{amodel}), and letting
\w{\Theta=\ThA} denote the full subcategory of the homotopy category \w{\ho\C}
consisting of objects of $\C$ which are representable as \emph{finite type}
products of objects in $\eA$ \wh that is, products of the form
\begin{myeq}\label{eqfintype}
\prod\sb{\bA\in\eA}\ \prod\sb{i\in I\sb{\bA}}\ \bA~,
\end{myeq}
\noindent where each indexing set \w{I\sb{\bA}} is finite.
When each \w{\bA\in\eA} is a homotopy group object in $\C$, \Tal[s] have a natural
underlying group structure; in such cases we call the sketch $\Theta$
\emph{algebraic}.

For every \w{\bY\in\C} we then have a \emph{realizable} \Tal
\w[,]{\Gamma=\HiT{\bY}} with \w{\Gamma\lin{\bB}=[\bY,\bB]\sb{\C}}
for every \w[.]{\bB\in\Theta} A \Tal is called \emph{free} if is isomorphic to
\w{\HiT{\bB}} for some \w[.]{\bB\in\Theta}

Note, however, that even though the forgetful functor
\w{U:\TAlg\to\Set\sp{\eA}} to $\eA$-graded sets has a left adjoint
\w[,]{F:\Set\sp{\eA}\to\TAlg} not all \Tal[s] in its image are free by our
definition: because $F$ preserves colimits, any \Tal in the image of $F$ is
a coproduct of \emph{monogenic} free \Tal[s] (those realizable by on object of
$\eA$), and conversely. For our purposes, however (see Theorem \ref{tres} below),
it is necessary that any free \Tal be realizable in $\C$.
\end{remark}

\begin{example}\label{egrtal}
For any commutative ring $R$, let \w{\eA:=\{\KR{n}\}\sb{n=1}\sp{\infty}} in
\w[:]{\C=\Sa} we then have an FP-sketch \w{\TR} in \w{\ho\Sa} whose objects are
``finite type'' $R$-GEMs of the form
\w[,]{\prod\sb{n=1}\sp{\infty}\,\KP{V\sb{n}}{n}} where
\w{V\sb{n}=R\sp{k\sb{n}}} is a free $R$-module of dimension \w[.]{k\sb{n}<\infty}
Since each \w{\bB\in\TR} is an $R$-module object, all \TRal[s] take values in
$R$-modules.

Note that the realizable \TRal \w{\Gamma=H\sp{\ast}\sb{\TR}\bY}
has \w[,]{\Gamma\lin{\KR{n}}=\Hu{n}{\bY}{R}} so we denote it simply
by \w[:]{\HiR{\bY}} it is the $R$-cohomology algebra of $\bY$, equipped with
the action of the primary $R$-cohomology operations.
When \w[,]{R=\Fp} \w{\TR} is the homotopy category
of such \ww{\Fp}-GEMs, and a \TRal is an unstable algebra over the
mod $p$ Steenrod algebra, as in \cite[\S 1.4]{SchwU}.
When \w[,]{R=\QQ} a \TRal is just a graded commutative $\QQ$-algebra.

More generally, for any limit cardinal $\lambda$, let \w{\RM\sp{\lambda}}
denote the set of isomorphism types of free $R$-modules of dimension $<\lambda$, and
let \w[.]{\eA\sp{\lambda}:=\{\KP{V}{n}~|\ V\in\RM\sp{\lambda},\ 1\leq n\leq\infty\}}
The FP-sketch \w{\TR\sp{\lambda}} then consists of the $R$-GEMs which are finite
type products of the form \w{\prod\sb{n=1}\sp{\infty}\ \KP{V\sb{n}}{n}} with
\w{V\sb{n}} a free $R$-module of dimension $<\lambda$.

Note that since finite products in \w{\RM} are also coproducts,
\w{\prod\sb{I\sb{n}}\,V\sb{i}} is itself in \w[,]{\RM\sp{\lambda}} since
\w{I\sb{n}} is finite and $\lambda$ is infinite.
However, we must be able to distinguish the generating set $\eA$ from the resulting
sketch \w[,]{\TR} in order to define what we mean by ``finite type'' products or
coproducts.
\end{example}

We note the following two standard facts about an algebraic sketch $\Theta$:

\begin{lemma}\label{lfreeta}
If $\Gamma$ is an \Tal and \w[,]{\bB\in\Theta} there is a natural isomorphism
\w[.]{\Hom\sb{\TAlg}(\HiT\bB,\,\Gamma)\cong\Gamma\lin{\bB}}
\end{lemma}

\begin{prop}[see \protect{\cite[\S 6]{BPescF}}]\label{psimptal}
The category \w{s\TAlg} of simplicial
\Tal[s] has a model category structure, in which the weak equivalences
and fibrations are defined on the underlying (graded) simplicial groups.
\end{prop}

\begin{defn}\label{dcwres}
For any algebraic sketch $\Theta$, a \emph{CW resolution} of a
\Tal $\Gamma$ is a cofibrant replacement \w{\vare:\Vd\xra{\simeq}\cd{\Gamma}}
equipped with a CW basis \w{(\oV{n})\sb{n\in\NN}} as in \S \ref{dscwo},
with each \w{\oV{n}} a free \Tal[.]
\end{defn}

\begin{remark}\label{rcwres}
In fact, any CW object \w{\Vd} for which each \w{\oV{n}} is a free \Tal[,]
and each attaching map \w{\odz{V\sb{n}}} \wb{n\geq 0} surjects onto
\w[,]{Z\sb{n-1}\Vd} is a CW resolution. Here we set \w{Z\sb{-1}\Vd:=\Gamma} and
\w[,]{\odz{V\sb{0}}:=\vare} so that
\begin{myeq}\label{eqaugmzero}
\vare\circ \odz{V\sb{1}}~=~0~.
\end{myeq}
\noindent Thus one can easily construct (non-functorial) CW resolutions of any
\Tal $\Gamma$.

Moreover, by dualizing \cite[Proposition 3.12]{BlaS} (see also
\cite[Proposition 12]{BlaCW}) we can choose a CW basis for any free
simplicial resolution in \w[,]{\TRA} when $R$ is a field. However, we shall
not make use of this fact.
\end{remark}

%
%
\sect{Realizing simplicial \Tal resolutions}
\label{crstr}

Let $\Theta$ be an algebraic sketch obtained as in \S \ref{rskhtpy} from a set $\eA$
of homotopy group objects in a model category $\C$. In this section we show how a
CW resolution \w{\Vd} of a realizable \Tal $\Gamma$ can itself be realized
over $\C$.

Any algebraic resolution \w{\Vd\to\Gamma} is clearly realizable by a cosimplicial
object \w{\widetilde{\Wu}\in c(\ho\C)} (for which the cosimplicial identities hold
up to homotopy), but it is not clear a priori that this can be rectified to a strict
cosimplicial object \w{\Wu} in $\C$. This is in fact possible in the two cases
described in Theorem \ref{tresg} of Appendix \ref{apfthm}.
However, for the purposes of this paper, we need not only the existence of this
realization \w[,]{\Wu} but also the particular form it takes, described as follows:

\begin{defn}\label{dscr}
Let $\Theta$ be a sketch in \w{\ho\C} as in \S \ref{rskhtpy}, and  let
\w{\Vd} be a simplicial CW resolution of the \Tal \w[,]{\HiT{\bY}} for
some \w[,]{\bY\in\C} with CW basis \w[.]{(\oV{n})\sb{n=0}\sp{\infty}}
A \emph{sequential realization}
\w{\cW=\lra{\W{n},\,\tWu{n}}\sb{n\in\NN}} of \w{\Vd} for $\bY$
consists of a tower of Reedy fibrant and cofibrant $\bY$-coaugmented
cosimplicial objects:
\begin{myeq}\label{eqtower}
\dotsc~\to~\W{n}~\xra{\prn{n}}~\W{n-1}~\xra{\prn{n-1}}~\W{n-2}~\to~\dotsc~
\to~\W{0}~,
\end{myeq}
\noindent in \w[,]{\C\sp{\Del}} with each \w{\prn{n}} a Reedy fibration. We call
\w{\W{n}} the  \emph{$n$-th stage} of $\cW$, and set
\w{\Wu:=\holim\sb{n}\ \W{n}} to be the limit of the tower
\wref[.]{eqtower} This tower must satisfy the following requirements for each
\w[:]{n\geq 0}

\begin{enumerate}
\renewcommand{\labelenumi}{(\alph{enumi})~}
\item The coaugmentation \w{\bve{n}:\bY\to\W{n}} realizes \w{\Vd\to\HiT{\bY}}
through simplicial dimension $n$ \wh that is, we have a natural isomorphism
\begin{myeq}\label{eqnatisom}
\HiT{\Wn{k}{n}}\lin{\bB}~\xra{\cong}~V\sb{k}\lin{\bB}
\hspace{5mm}\text{for all} \ \bB\in\Theta \ \text{and} \ -1\leq k\leq n~,
\end{myeq}
\noindent using the notation of \S \ref{egrtal}.
\item The coaugmentation \w{\bve{n-1}:\bY\to\W{n-1}} lifts along the
\w{\prn{n}:\W{n}\to\W{n-1}} to
\w[.]{\bve{n}:\bY\to\W{n}}
\item Each \w{\W{n}}  is obtained from \w{\W{n-1}} as follows:
assume given a fibrant realization \w{\oW{n}\in\C} of \w{\oV{n}}
and let \w{\bDs\in\Ch{\C}} be a left Reedy fibrant replacement of
\w[.]{\oW{n}\ouS{n-1}} Assume that we can realize the $n$-th attaching map
\w{\odz{V\sb{n}}:\oV{n}\to V\sb{n-1}} of \w[,]{\Vd} by a cochain map
\w[,]{F:\cMs\W{n-1}\to\bDs} in the category of coaugmented cochain complexes (that is,
defined in dimensions \ \w[).]{-1\leq k\leq n-1} Note that
\w[.]{(\oW{n}\ouS{n-1})\sp{-1}=0}

Let \w{\widetilde{F}:\U\W{n-1}\to\E\bDs} the
corresponding map of restricted cosimplicial objects  as in \S \ref{dccwo}, and set
\w{\tWu{n}:=\Fib(\widetilde{F})} as in \S \ref{dchaincof}, take a pullback:
\mydiagram[\label{eqncoskeletn}]{
\ar @{} [drr]|(0.24){\framebox{\scriptsize{PB}}}
\vWu{n} \ar[d] \ar[rr] && \F\tWu{n} \ar[d]\sp{\F\ell} \\
\W{n-1} \ar[rr]\sp{\theta} && \F\U\W{n-1}~,
}
\noindent as in \wref[,]{eqncoskeleton} with \w{\W{n}} a
Reedy fibrant and cofibrant replacement (see \cite[\S 15.3]{PHirM})
for \w[.]{\hW{n}}

We start the process with \w[.]{\W{0}:=\cu{\oW{0}}}
\item For reasons which will appear later, we sometimes require the
chain map \w{F:\cMs\W{n-1}\to\bDs} in step (c) to be a levelwise
(i.e., left Reedy) cofibration. If this is true for each \w[,]{n\geq 1} we say that
the sequential realization $\cW$ is \emph{cofibrant}.
\end{enumerate}
\end{defn}

\begin{mysubsect}{Understanding the passage from \w{\W{n-1}} to \ww{\W{n}}}
\label{senso}

We now give an explicit description of the construction of \w{\W{n}}
from \w[,]{\W{n-1}} also introducing some auxiliary notation\vsm:

\noindent\textbf{(i)~~What does a left Reedy fibrant replacement for
\w{\oW{n}\ouS{n-1}} look like\vsn?}

In order for the coaugmented cochain complex \w{\bDs} in \w{\Ch{\C}} to be left
Reedy fibrant, the structure map \w{\bD\sp{k-1}\to\Ker(\dif{k}\sb{\bD})} (into
the ``$k$-cocycles'' \wh not the same as \w{\cZ{k}\bDs} of \S \ref{dmco}) must
be a fibration for each $k$.
In addition, since \w{(\oW{n}\ouS{n-1})\sp{k}=0} for \w[,]{k\neq n-1}
\w{\bD\sp{k}} must be contractible.

It is natural for such a \w{\bDs} to be described by a downward induction,
starting with \w{\bD\sp{n}=0} and \w[.]{\bD\sp{n-1}=\oW{n}} Thus
\w[,]{\Ker(\dif{n-1}\sb{\bD})=\oW{n}} too, so \w{\bD\sp{n-2}} must be a contractible
object equipped with a fibration to \w{\oW{n}} \wwh i.e., it is a path object for
\w[,]{\oW{n}} in the sense of \cite[I, \S 2]{QuiH}. We could choose it to
be the standard path object \w{P\oW{n}} of \wref[,]{eqpathloop} but this may not be
a good choice if we also want \w{\Fk{k}:\cM{k}\W{n-1}\to\bD\sp{k}} to be
a cofibration for each $k$, as in \S \ref{dscr}(d).
Thus we let \w{\dif{n-2}\sb{\bD}:\bD\sp{n-2}\to\bD\sp{n-1}} be some fibration
\w{\pp{0}:\overline{P\bW\sp{n}}\to\oW{n}} with \w{\overline{P\bW\sp{n}}}
contractible.

At the stage \w{n-3} we see that \w{\Ker(\dif{n-2}\sb{\bD})} is the fiber
\w{\ooW{1}{n}} of \w[,]{\pp{0}} and again
\w{\dif{n-3}\sb{\bD}:\bD\sp{n-3}\to\bD\sp{n-2}} must be a fibration
\w{\pp{1}:\PoW{1}{n}\to\ooW{1}{n}} (with \w{\PoW{1}{n}} contractible), followed
by the inclusion \w[.]{\ip{1}:\ooW{1}{n}\hra\PoW{0}{n}}

Proceeding by induction, for each \w{-1\leq j\leq n-1} we obtain a
fibration sequence
\mydiagram[\label{eqmodpathloop}]{
\ooW{j+1}{n}~\quad \ar@{^{(}->}[rr]\sp{\ip{j+1}} &&
\PoW{j}{n} \ar@{->>}[rr]\sp{\pp{j}} && \ooW{j}{n}
}
\noindent in $\C$ called the \emph{$j$-th modified path-loop fibration},
with \w{\PoW{j}{n}} contractible.
By convention we set \w[,]{\oW{n}:=\ooW{0}{n}=\PoW{-1}{n}} with the identity map as
\begin{myeq}\label{eqconvent}
\ip{0}~:~\ooW{0}{n}~\xra{=}~\PoW{-1}{n}.
\end{myeq}

Thus \w{\bD\sp{k}=\PoW{n-k-2}{n}} for each \w[,]{-1\leq k\leq n-1} and the
differential \w{\dif{k}\sb{\bD}:\bD\sp{k}\to\bD\sp{k+1}} is the composite
\begin{myeq}\label{eqtau}
\PoW{n-k-2}{n}~\xra{\pp{n-k-2}}~\ooW{n-k-2}{n}~\xra{\ip{n-k-2}}~\PoW{n-k-3}{n}~.
\end{myeq}

For later use we may (and shall) assume (see Step \textbf{VIII} in proof of Theorem
\ref{tresg}) that for each
\w[,]{0\leq j\leq n-1} \wref{eqmodpathloop} fits into a commutative diagram
\mydiagram[\label{eqmodpathloopve}]{
\ooW{j+1}{n}~\quad \ar@{->>}[d]^<<<<{\sigma\sp{j+1}}_<<<<{\simeq}
\ar@{^{(}->}[r]\sp{\ip{j+1}} &
~\PoW{j}{n}~\ar@{->>}[d]^<<<<{P\tau\sp{j}}_<<<<{\simeq} \ar@{->>}[r]\sp{\pp{j}} &
~\ooW{j}{n} \ar@{->>}[d]^<<<<{=}\\
\Omega\ooW{j}{n}~\quad \ar@{^{(}->}[r]\sp{\iota} &~ P\ooW{j}{n}~
\ar@{->>}[r]\sp{p} & ~\ooW{j}{n}~.
}
\noindent in which the bottom sequence is the usual path-loop fibration
for \w[,]{\ooW{j}{n}} and the vertical maps are all trivial fibrations\vsm.

\noindent\textbf{(ii)~~The map \w{F:\cMs\W{n-1}\to\bDs} and its fiber\vsn:}

As in \S \ref{dscr}(c), the $n$-th attaching map for \w{\Vd}
is to be realized by a cochain map \w{F:\cMs\W{n-1}\to\bDs} in the category of
coaugmented cochain complexes given by maps
\w{\Fk{k}:\cM{k}\W{n-1}\to\bD\sp{k}=\PoW{n-k-2}{n}} for each
\w[,]{-1\leq k\leq n-1} so by Definition \ref{dchaincof} the restricted cosimplicial
object  \w{\tWu{n}=\Fib(\widetilde{F})} is given by
\begin{myeq}\label{eqdopb}
\tWn{k}{n}~:=~\Wn{k}{n-1}\times\PoW{n-k-1}{n}
\end{myeq}
\noindent in dimension \w[,]{0\leq k\leq n} while by \wref{eqconvent} we have
\begin{myeq}\label{eqtoptw}
\tWn{n}{n}~:=~\Wn{n}{n-1}\times\oW{n}~.
\end{myeq}
\noindent We denote the two structure maps for the product \wref{eqdopb} by
\begin{myeq}
\psn{k}{n}:\tWn{k}{n}\to\Wn{k}{n-1},~~
\qk{k}{n}:\tWn{k}{n}\to\PoW{n-k-1}{n},
\end{myeq}
 \noindent respectively.
By Definition \ref{dchaincof} we see that the coface maps
\w{\td\sp{i}\sb{k}:\tWn{k}{n}\to\tWn{k+1}{n}} are determined by:
\begin{myeq}\label{eqzerocoface}
\Fk{k}\circ v\sp{k}\circ\psn{k}{n}~=~\qk{k+1}{n}\circ\td\sp{0}\sb{k}~.
\end{myeq}
\noindent where \w{v\sp{k}:\Wn{k}{n-1}\to\cM{k}\W{n-1}} is the structure map
for \wref[.]{eqmoorecc}  For \w{i=1} we have:
\begin{myeq}\label{eqfirstcoface}
\dif{k-1}\sb{\bD}\circ\qk{k}{n}~=~
\ip{n-k-1}\circ\pp{n-k-1}\circ\qk{k}{n}~=~\qk{k+1}{n}\circ\td\sp{1}\sb{k}~,
\end{myeq}
\noindent while for \w{i\geq 2} we have
\begin{myeq}\label{eqjcoface}
\qk{k+1}{n}\circ\td\sp{i}\sb{k}~=~0~.
\end{myeq}
\noindent When \w{k=n-1} we have
\w[,]{\pp{0}\circ\qk{n-1}{n}~=~\qk{n}{n}\circ\td\sp{1}\sb{n-1}} in accordance
with \wref[.]{eqconvent}

We note also that if the given \w{\W{n-1}} is equipped with a coaugmentation
\w[,]{\bve{n-1}:\bY\to\W{n-1}} we also have a coaugmented version
\w{\tve{n}:\bY\to\tWu{n}} for the restricted cosimplicial object \w{\tWu{n}}, which is defined in the same way by setting \w[\vsm.]{\tWn{-1}{n-1}:=\bY}

\noindent\textbf{(iii)~~Making \w{\tWu{n}} into a full cosimplicial object
\w[\vsn:]{\W{n}}}

Given the restricted cosimplicial object \w{\tWu{n}} obtained
in step \textbf{(ii)}, we first define a full cosimplicial object
\w{\vWu{n}} by setting
\begin{myeq}\label{eqaddon}
\uG{k}~:=~\begin{cases}
\PoW{n-k-1}{n}& \ \text{if}\ k\leq n\\
\ast & \ \text{if} \ k>n
\end{cases}
\end{myeq}
\noindent (using \wref[).]{eqconvent} We then let:
\begin{myeq}\label{eqsmatch}
\vWn{r}{n}:=\tWn{r}{n}\times
\prod\sb{0< k<r}\prod\sb{0\leq i\sb{1}<\dotsc<i\sb{k}\leq r-1}\uG{r-k}~
=~\Wn{r}{n-1}~\times~
\prod\sb{0\leq k\leq r}\prod\sb{0\leq i\sb{1}<\dotsc<i\sb{k}\leq r}\uG{r-k}
\end{myeq}
\noindent be the construction denoted by \w{\csk{n}\Gu} in \wref[.]{eqncoskeleton}
See \S \ref{slmo} and \wref{eqdopb} (and compare \wref{eqslatch} in the
dual case).

The codegeneracy map \w{s\sp{t}:\vWn{r+1}{n}\to\vWn{r}{n}} is defined into
the factor \w{\uG{r-k}} of \w{\vWn{r}{n}} indexed by the $k$-tuple
\w{I=(i\sb{1},\dotsc,i\sb{k})} by projecting \w{\vWn{r+1}{n}} onto
the factor \w{\uG{r-k}} indexed by the unique
\wwb{k+1}tuple \w{J=(j\sb{1},\dotsc,j\sb{k+1})} satisfying the cosimplicial
identity \w[.]{s\sp{I}\circ s\sp{t}=s\sp{J}}
The coface maps of \w{\vWu{n}} are determined by those of \w{\tWu{n}}
and the cosimplicial identities, and we have a natural map of
restricted cosimplicial objects \w{g:\U\vWu{n}\to\tWu{n}}
which is a dimensionwise trivial fibration.
\end{mysubsect}

\begin{remark}\label{rcoffib}
Note that \w{\vWu{n}} is obviously $n$-coskeletal. Moreover, it is Reedy fibrant,
since the natural map
\w{\widehat{\zeta}\sp{r}:\vWn{r}{n}\to M\sp{r}\vWu{n}} (see \S \ref{slmo}) is just
the product of \w{\zeta\sp{r}:\Wn{r}{n-1}\to M\sp{r}\W{n-1}} (which is a fibration,
since \w{\W{n-1}} is Reedy fibrant) with the
projection onto the appropriate factors in \wref{eqsmatch} (which is a fibration
since all objects \w{\uG{k}} are fibrant). Moreover, the composites of
\w{\vWn{k}{n}~\xepic{g\sp{k}}~\tWn{k}{n}~\xepic{\psn{k}{n}}~\Wn{k}{n-1}} for each
$k$ fit together to define a map of cosimplicial objects
\w[,]{\ppp\prnk{n}{}:\vWu{n}\epic\W{n-1}} which is a Reedy fibration (for the
same reason).

Finally, we let \w{h\bp{n}:\W{n}\to\vWu{n}} be a (functorial) Reedy cofibrant
replacement (see \cite[\S 15]{PHirM}, and compare \cite[X,\S 4.2]{BKanH}), so
\w{h\bp{n}} is a trivial Reedy fibration, and we set
\w{\prnk{n}{}:\W{n}\to\W{n-1}} to be the composite \w{\ppp\prnk{n}{}\circ h\bp{n}}
\wwh again a Reedy fibration. Thus the full cosimplicial object \w{\W{n}} is
Reedy fibrant and cofibrant, and dimensionwise weakly equivalent to \w[.]{\vWu{n}}
Since the latter is $n$-coskeletal, we see that \w{\W{n}} is $n$-coskeletal
``up to homotopy''.
\end{remark}

\begin{lemma}\label{linducea}
Let \w{(\bCs,\dif{}\sb{\bC})} be a cochain complex and \w{(\bDs,\,\dif{}\sb{\bD})}
a left Reedy fibrant \wwb{n-1}truncated cochain complex as in
\S \ref{senso}\textbf{(i)} in a pointed model category $\C$, and let
\w{\cZ{j}\bCs:=\Coker(\dif{j-1}\sb{\bC})} with structure map
\w{w\sp{j}:\bC\sp{j}\to\cZ{j}\bCs} as in \S \ref{dmco}. Any cochain map
\w{\Fk{j}:\bC\sp{j}\to\bD\sp{j}} defined for \w{k\leq j<n} induces a unique map
\w{\ak{k-1}:\cZ{k-1}\bCs\to\ooW{n-k-1}{n}} with
\begin{myeq}\label{eqikak}
\ip{n-k-1}\circ\ak{k-1}\circ w\sp{k-1}~=~\Fk{k}\circ\delta\sp{k-1}\sb{\bC}~,
\end{myeq}
\noindent in the notation of \wref{eqmodpathloop} and \wref[.]{eqtau}
\end{lemma}

\begin{proof}
By assumption, we have the following solid commuting diagram:
\mydiagram[\label{eqcochainmap}]{
\bC\sp{k+1}\ar[rrr]\sp{\Fk{k+1}} &&& \PoW{n-k-3}{n} & =~~\bD\sp{k+1} \\
& && \ooW{n-k-2}{n} \ar@{^{(}->}[u]\sp{\ip{n-k-2}} & \\
\bC\sp{k} \ar[uu]\sp{\dif{k}\sb{\bC}} \ar[rrr]\sp{\Fk{k}} &&&
\PoW{n-k-2}{n} \ar@{->>}[u]\sp{\pp{n-k-2}} &
=~~\bD\sp{k} \ar[uu]\sb{\dif{k}\sb{\bD}} \\
& \cZ{k-1}\bCs \ar[ul]\sb{\ud\sp{0}} \ar@{.>}[rr]\sp{\ak{k-1}}
&& \ooW{n-k-1}{n} \ar@{^{(}->}[u]\sp{\ip{n-k-1}} \ar@/_{3.9pc}/[uu]\sb{0} & \\
\bC\sp{k-1} \ar[ur]\sp{w\sp{k-1}} \ar@{.>}[rrru]\sb{a} \ar[uu]\sp{\dif{k-1}\sb{\bC}}
\ar@{.>}[rrr]\sb{\Fk{k-1}} &&& \PoW{n-k-1}{n} \ar@{->>}[u]\sp{\pp{n-k-1}} &
=~~\bD\sp{k-1} \ar[uu]\sb{\dif{k-1}\sb{\bD}}
}
\noindent Then
$$
\ip{n-k-2}\circ\pp{n-k-2}\circ\Fk{k}\circ\dif{k-1}\sb{\bC}~=~
\dif{k}\sb{\bD}\circ\Fk{k}\circ\dif{k-1}\sb{\bC}~=~
\Fk{k+1}\circ\dif{k}\sb{\bC}\circ\dif{k-1}\sb{\bC}~=~0~.
$$
\noindent Since \w{\ip{n-k-2}} is a monomorphism, in fact
\w[.]{\pp{n-k-2}\circ\Fk{k}\circ\dif{k-1}\sb{\bC}=0}
Therefore, since \wref{eqmodpathloop} is a fibration sequence,
\w{\Fk{k}\circ\dif{k-1}\sb{\bC}} factors through \w{a:\bC\sp{k-1}\to\ooW{n-k-1}{n}}
as indicated in \wref[.]{eqcochainmap} Moreover, since
\w{\ip{n-k-1}\circ a\circ\dif{k-2}\sb{\bC}=
\Fk{k}\circ\dif{k-1}\sb{\bC}\circ\dif{k-2}\sb{\bC}=0}
and \w{\ip{n-k-1}} is a monomorphism, too, the map $a$ factors through
\w{\ak{k-1}:\cZ{k-1}\bCs\to\ooW{n-k-1}{n}} as in \wref[,]{eqcochainmap}
satisfying \wref[.]{eqikak}
\end{proof}

We next note the following technical fact about Moore chain objects:

\begin{lemma}\label{lmoore}
Let \w{\Wu\in\C\sp{\Del}} be a Reedy cofibrant cosimplicial object over a
model category $\C$ as in \S \ref{amodel}, and $\bB$ a homotopy group object
in $\C$. Then for any  Moore chain \w{\alpha\in C\sb{n}[\Wu,\bB]} for the
simplicial group \w[:]{[\Wu,\bB]}
\begin{enumerate}
\renewcommand{\labelenumi}{(\alph{enumi})~}
\item $\alpha$ can be realized by a map \w{a:\bW\sp{n}\to\bB} with
\w{a\circ d\sp{i}\sb{n-1}=0} for all \w[,]{1\leq i\leq n} and thus
induces a map \w{\bar{a}:C\sp{n}\Wu\to\bB} with
\w[.]{\bar{a}\circ v\sp{n}=a}
\item If $\alpha$ is a Moore \emph{cycle}, we can choose a nullhomotopy
\w{H:\bW\sp{n-1}\to P\bB\subseteq \bB\sp{[0,1]}} for \w{a\circ \dz{n-1}} such
that \w{H\circ d\sp{j}\sb{n-2}=0} for \w[,]{1\leq j\leq n-1} and thus
induces a map \w{\bar{H}:C\sp{n}\Wu\to P\bB} with
\w[.]{\bar{H}\circ v\sp{n}=H}
\end{enumerate}
\end{lemma}

\begin{proof}
Since \w{\Wu} is Reedy cofibrant, the simplicial space
\w{\Ud=\mapa(\Wu,\bB)\in s\Sa} is Reedy fibrant, so we have an isomorphism
\begin{myeq}\label{eqcommmoor}
\iota\sb{\star}~:~\pi\sb{i}C\sb{n}\Ud~\to~C\sb{n}\pi\sb{i}\Ud\hsp
\text{for all}\ i\geq 0
\end{myeq}
\noindent (see \cite[X, Proposition 6.3]{BKanH}). Thus we can represent
\w{\alpha\in C\sb{n}\pi\sb{0}\Ud} by a map \w[,]{a\in C\sb{n}\Ud} which implies
(i)\vsm.

If $\alpha$ is a cycle, then \w{\partial\sb{n}(\alpha)=[a\circ \dz{n-1}]}
vanishes in \w[,]{\pi\sb{0}C\sb{n-1}\Ud} so we have a nullhomotopy $H$ for
\w{a\circ \dz{n-1}} in
$$
PC\sb{n-1}\mapa(\Wu,\bB)=C\sb{n-1}\mapa(\Wu,P\bB)~
\subseteq~\mapa(\bW\sp{n-1},P\bB)~,
$$
\noindent which implies (ii).
\end{proof}

From the description in \S \ref{senso} we can actually deduce:

\begin{prop}\label{pnstep}
Any \w{\W{n}} obtained from \w{\W{n-1}} as in Definition \ref{dscr}(c)
will satisfy \S \ref{dscr}((a)-(b), and the limit \w{\Wu} of \wref{eqtower}
is a Reedy fibrant cosimplicial object over $\C$ which realizes the given
algebraic resolution \w[.]{\Vd\to\HiT{\bY}}
\end{prop}

\begin{proof}
In the setting of \S \ref{dscr}, with \w{\bCs=\cMs\W{n-1}} for Reedy cofibrant
\w[,]{\W{n-1}} we can use Lemma \ref{lmoore}(a) to represent the attaching map
\w{\odz{n}:\oV{n}\to\cZl{n-1}\Vd} for the CW resolution \w{\Vd} by a map
\w{\Fk{n-1}:\cM{n-1}\W{n-1}\to\oW{n}} (see Step \textbf{IV} in the proof of
Theorem \ref{tresg} below).
From \wref{eqaddon} and \wref[,]{eqsmatch} we then see that \w{\W{n}} realizes
\w{\Vd} through simplicial dimension $n$.

Moreover, \w{\Wu} is as stated because the maps \w{\prn{n}} restrict to trivial
fibrations \w{\prnk{n}{k}:\Wn{k}{n}\to\Wn{k}{n-1}} for each \w[,]{0\leq k<n}
so \w{\bW\sp{k}=\holim\sb{n}\,\Wn{k}{n}} in $\C$.
\end{proof}

\begin{mysubsection}{Higher cohomology operations}
\label{shcos}
When \w{\Fk{k-1}} also exists making \wref{eqcochainmap}
commute, we have
\begin{myeq}\label{eqak}
\ak{k-1}\circ w\sp{k-1}~=~\pp{n-k-1}\circ\Fk{k-1}~,
\end{myeq}
\noindent so \w{\Fk{k-1}} is a nullhomotopy for
\w{\ak{k-1}\circ w\sp{k-1}} in the sense of \cite[I, \S 2]{QuiH}.
Note that Lemma \ref{linducea} also makes sense for \w[,]{k=n-1} where
\w[.]{\Fk{n}=0}

If we choose \w{\Fk{n-1}} as in the proof of Proposition \ref{pnstep},
the map \w[,]{\Fk{n-1}\circ\dif{n-2}} which induces
\w[,]{\ak{n-2}\circ w\sp{n-2}} is nullhomotopic, with nullhomotopy
\w[.]{\Fk{n-2}} We can then think of
\w{\ak{n-3}\circ w\sp{n-3}:\cM{n-2}\to\ooW{1}{n}}
as the value of the secondary cohomology operation corresponding to the diagram
\mydiagram[\label{eqtodacoh}]{
\bC\sp{n-3} \ar@/^{2.0pc}/[rr]\sp{0} \ar[r]\sp{\dif{n-3}} &
\bC\sp{n-2}  \ar@/_{1.2pc}/[rr]\sb{0} \ar[r]\sp{\dif{n-2}} &
\bC\sp{n-1}  \ar[r]\sp{\Fk{n-1}} & \oW{n}
}
\noindent as in \wref[.]{eqtoda} Only \w{\oW{n}} is an $R$-GEM, but this suffices
to let us think of each value \w{\ak{n-3}\circ w\sp{n-3}} of this Toda bracket
in \w{[\bC\sp{n-3},\,\Omega\oW{n}]}
as a collection of cohomology classes for \w[.]{\bC\sp{n-3}}
This nevertheless qualifies as a higher cohomology operation as described in
the Introduction, if we use a truncation of \w{\Del\sb{+}} as our indexing
category $I$.

By what we say above, if \w{\ak{n-3}\circ w\sp{n-3}\sim\ast} \wwh that is, the
secondary operation corresponding to \wref{eqtodacoh} vanishes \wh then the
choice of a nullhomotopy \w{\Fk{n-3}} yields a value
\w{\ak{n-4}\circ w\sp{n-4}:\bC\sp{n-4}\to\ooW{2}{n}} for the corresponding third
order operation, and so on. This observation is the key to what we are doing in
this paper.
\end{mysubsection}

\begin{defn}\label{dallowsk}
Let $\Theta$ be an  algebraic sketch as in \S \ref{rskhtpy},
associated to \w[,]{\eA\subseteq\Obj\C} so that by definition
any \w{\bB\in\Theta} is of the form
\w{\bB:=\prod\sb{\bA\in\eA}\ \prod\sb{i\in I\sb{\bA}}\ \bA}
with each \w{I\sb{\bA}} a finite indexing set (see \wref[).]{eqfintype}

We then say  that $\Theta$ is  \emph{allowable} if the natural map
\begin{myeq}\label{eqprodcoprod}
\coprod\sb{\bA\in\eA}\ \coprod\sb{i\in I\sb{\bA}}\ \HiT{\bA}~\to~\HiT{\bB}~
\end{myeq}
\noindent is an isomorphism for any such \w[.]{\bB\in\Theta}
\end{defn}

\begin{remark}\label{rallow}
Note that if we write \w{I:=\coprod\sb{A\in\eA}\ I\sb{\bA}} and
denote the copy of $\bA$ indexed by \w{i\in I\sb{\bA}} by \w[,]{\bB\sb{i}}
we have \w[.]{\bB=\prod\sb{i\in I}\,\bB\sb{i}}
For any \Tal $\Gamma$ we then have by
Lemma \ref{lfreeta} and \wref[:]{eqprodcoprod}
\begin{myeq}\label{eqextyoneda}
\begin{split}
\prod\sb{i\in I}\,\Gamma\lin{\bB\sb{i}}~=&~
\prod\sb{i\in I}\,\Hom\sb{\TAlg}(\HiT{\bB\sb{i}},\,\Gamma)\\
~=&~\Hom\sb{\TAlg}(\coprod\sb{i\in I}\,\HiT{\bB\sb{i}},~\Gamma)~=~
\Hom\sb{\TAlg}(\HiT{\bB},~\Gamma)~.
\end{split}
\end{myeq}
\end{remark}

\begin{lemma}\label{lallow}
If \w{R=\Fp} or a field of characteristic $0$ and $\lambda$ is any limit cardinal,
the FP-sketch \w{\TR\sp{\lambda}} of \S \ref{egrtal} is allowable.
\end{lemma}

\begin{proof}
Every \w{\bB\in\TR\sp{\lambda}} has the form
\w{\bB=\prod\sb{n=1}\sp{\infty}\,\prod\sb{I\sb{n}}\,\KP{V\sb{i}}{n}}
for \w{V\sb{i}\in\RM\sp{\lambda}} and finite indexing sets \w[.]{I\sb{n}}
Then
\w{\HiR{\bB}=\coprod\sb{n=1}\sp{\infty}\,\coprod\sb{I\sb{n}}\,\HiR{\KP{V\sb{i}}{n}}}
(see \cite[Lemma 4.17]{BSenH}).
\end{proof}

We are now in a position to state our first important technical result:

\begin{thm}\label{tres}
Let $\Theta$ be an allowable algebraic sketch in $\C$ and let \w{\Vd} be any CW
resolution of the realizable
\Tal \w[.]{\HiT{\bY}} Then there is a cofibrant sequential realization
\w{\cW=\lra{\W{n},\,\tWu{n}}\sb{n\in\NN}} of \w{\Vd} for $\bY$, with each
\w{\Wn{k}{n}} in $\Theta$.
\end{thm}

We defer the proof to Appendix \ref{apfthm}, where we actually prove a
more general result (which is needed elsewhere).

\begin{remark}\label{rcardinal}
If we want to use the allowability of \w{\TR\sp{\lambda}} in Lemma \ref{lallow} for
Theorem \ref{tres}, the choice of the cardinal $\lambda$ may depend on the
size of graded $R$-vector space \w{\HiR{\bY}} (see \cite[\S 3]{BSenH}).
However, in the most commonly
encountered case, \w{\HiR{\bY}} will be of finite type, and we may choose the
CW resolution \w{\Vd} to be finite type in each simplicial dimension, too.
In this case we can make do with the original \w{\TR=\TR\sp{\omega}} of
\S \ref{egrtal}.
\end{remark}

%
%
\sect{Comparing cosimplicial resolutions}
\label{ccsr}

Cosimplicial resolutions of the type constructed in Section \ref{crstr} play a
central role in our theory of higher cohomology operations, but they depend on
many particular choices.  In this section we shall show how any two such
cosimplicial objects are related by a zigzag of maps of a particularly simple
form.

 Although many of the results hold more generally, from now on we
restrict attention to the algebraic sketch \w{\TR} for \w{R=\Fp} or a field of
characteristic $0$ (see \S \ref{egrtal}), with \w[.]{\C=\Sa} This allows us
to assume for convenience that all the objects in each stage of our
sequential realizations are simplicial $R$-modules (though the
maps between them need not be strict simplicial homomorphisms).

We shall also assume from here on that all spaces are connected.
The modifications needed for the non-connected case should be clear\vsm.

We first note the following general facts about model categories:

\begin{lemma}\label{lcylin}
Let $X$ and $Y$ be two weakly equivalent fibrant and cofibrant objects in a
simplicial model category $\C$.
\begin{enumerate}
\renewcommand{\labelenumi}{(\alph{enumi})~}
\item There is a diagram of weak equivalences:
\mydiagram[\label{eqtwotriang}]{
&& Z \ar@/_1pc/[dll]_{s} \ar@/^1pc/[drr]^{t}  && \\
X \ar[rru]\sb{h} \ar[rrrr]^{f} &&&& Y \ar[llu]\sp{j}
\ar@/^1pc/[llll]^{g}
}
\noindent where \w[,]{h:=f'\circ i\sb{0}}  \w[,]{s\circ j=g} \w[,]{t\circ h=f}
\w[,]{s\circ h=\Id\sb{X}} \w[,]{t\circ j=\Id\sb{Y}} and the notation \w{f', i\sb{0}} is as in the proof.
\item There are maps \w[,]{X\amalg Y~\xra{F}~Z~\xra{G}~X\times Y} with $F$ a
cofibration which is a trivial cofibration on each summand,
$G$ a fibration which is a trivial fibration onto each factor, and the induced maps
\w{X\to X} and \w{Y\to Y} are identities.
\end {enumerate}
\end{lemma}

\begin{proof}
\ \ \textbf{(a)} \ By \cite[I, \S 1]{QuiH}) we have homotopy equivalences
\w{f:X\to Y} and \w{g:Y\to X} with a homotopy \w{H:g\circ f\sim\Id\sb{X}} fitting
into a commutative diagram:
\mydiagram[\label{eqhcobordism}]{
\ar @{} [drr]|<<<<<<<<<<<<<<<<<<<{\framebox{\scriptsize{PO}}}
X \ar@{^{(}->}[rr]^{i\sb{1}} \ar[d]_{f} \ar@/^5pc/[drrrr]^{=} &&
X\otimes\Delta[1] \ar[d]_{f'} \ar[rdd]^{H} \ar[rrd]^{p} &&
X \ar@{_{(}->}[ll]_{i\sb{0}}  \ar[d]^{=} \\
Y \ar@{^{(}->}[rr]^<<<<{j} \ar[rrrd]_{g} \ar@/_3pc/[drrrrr]_{=}&&
Z \ar@{.>}[rd]_<<<<<<{s}  \ar@{.>}[rrrd]^{t}  && X \ar[rd]^{f}\\
&&& X && Y
}
\noindent with all maps weak equivalences, where $Z$ is the pushout, the maps
\w{i\sb{0}} and \w{i\sb{1}} are induced by the inclusions
\w{\Delta[0]\hra \Delta[1]} and $p$ is induced by
\w[\vsm .]{\Delta[1]\epic\Delta[0]}

\noindent \textbf{(b)} \ Choose a weak equivalence \w[,]{f:X\to Y} and factor it
as \w[,]{X\xra{k}\hat{Z}\xra{\ell} Y} with $k$ a trivial cofibration and $\ell$ a
trivial fibration. By the LLP and fibrancy of $X$ we have a retraction
\w{r:\hat{Z}\to X} for $k$, and by RLP and cofibrancy of $Y$ we have a section
\w{u:Y\to\hat{Z}} for $\ell$, both weak equivalences. Set
\w[.]{\phi:=(\Id\sb{X}\top(r\circ u))\bot(f\top\Id\sb{Y}):X\amalg Y\to X\times Y}

Factor \w{k\bot u:X\amalg Y\to\hat{Z}} as
\w{X\amalg Y\xra{k'\bot u'} Z'\xra{p}\hat{Z}} (a cofibration followed by
a trivial fibration), and \w{r\top\ell:\hat{Z}\to X\times Y} as
\w{\hat{Z}\xra{i}Z''\xra{r'\top\ell'}X\times Y} (a trivial cofibration followed
by a fibration).   Finally, factor \w{i\circ p:Z'\xra{\simeq} Z''} as
\w{Z'\xra{e}Z\xra{q}Z'} (a trivial cofibration followed by a trivial fibration):

\mydiagram[\label{eqmodelcat}]{
X \ar[dd]^{\Id} \ar@{^{(}->}[rr]^{\inc}
\ar[rrrrd]_(0.3){k}_(0.7){\simeq} |!{[rr];[rrdd]}\hole &&
X\amalg Y \ar[dd]_{\phi} \ar[rrd]\sp{k\bot u} \ar@{^{(}->}[rrrr]\sp{k'\bot u'}
\ar@/^{4.0pc}/[rrrrrrd]_{F} &&&&
Z' \ar@{->>}[lld]^{p}_{\simeq}
\ar[dd]\sb{i\circ p}\sp{\simeq}\ar@{^{(}->}[rrd]\sb{e}\sp{\simeq}&&\\
&&&& \hat{Z} \ar@{^{(}->}[rrd]\sp{i}\sb{\simeq} \ar[lld]\sp{r\top\ell}
\ar[lllld]_(0.3){\simeq}_(0.7){r} |!{[llu];[lld]}\hole &&&&
Z \ar@{->>}[lld]\sb{q}\sp{\simeq} \ar@/^{4.0pc}/[lllllld]_{G} \\
X && X\times Y \ar@{->>}[ll]\sp{\proj} &&&&
Z'' \ar@{->>}[llll]\sp{r'\top\ell'}\sb{\simeq} \vsm &
}
\vsm\quad

\noindent Then \w{F:=e\circ(k'\bot u')} is a cofibration,
\w{G:=(r'\top\ell')\circ q} is a fibration, and the claim follows by tracking the
weak equivalences in \wref{eqmodelcat} (and similarly for $Y$).
\end{proof}

\begin{defn}\label{dacomp}
Given two CW resolutions \w{\vare:\Vd\to\Gamma} and \w{\varep:\Vdp\to\Gamma} of
a \TRal $\Gamma$, with CW bases \w{(\oV{n})\sb{n\in\NN}} and
\w[,]{(\oVp{n})\sb{n\in\NN}} an \emph{algebraic comparison map}
\w{\Psi:\Vd\to\Vdp} is a system
\begin{myeq}\label{eqacomp}
\Psi~=~\lra{\varphi,\,\rho,\,(\ophl{n},\,\orh{n})\sb{n\in\NN}}~,
\end{myeq}
\noindent where \w{\varphi:\Vd\to\Vdp} is a split monic weak equivalence of
simplicial \TRal[s] with retraction \w{\rho:\Vdp\to\Vd}
(with \w[),]{\varep\circ\ophl{0}=\vare} induced by inclusions of coproduct summands
\w{\ophl{n}:\oV{n}\hra\oVp{n}} with retractions \w{\orh{n}} for each \w[.]{n\geq 0}
\end{defn}

\begin{lemma}\label{lcylind}
Any two CW resolutions \w{\varu{0}:\Vud{0}\to\Gamma} and
\w{\varu{1}:\Vud{1}\to\Gamma} of the same \TRal $\Gamma$ have a common
``algebraic $h$-cobordism'' CW resolution \w[,]{\vare:\Vd\to\Gamma} with
algebraic comparison maps \w{\Psi\up{i}:\Vud{i}\to\Vd} \wb[.]{i=0,1}
\end{lemma}

\begin{proof}
Let \w{(\ouV{n}{i})\sb{n\in\NN}} be CW bases for \w{\Vud{i}} \wb[.]{i=0,1}
Since \w{X=\Vud{0}} and \w{Y=\Vud{1}} are fibrant and cofibrant in \w{s\TRA}
(see Proposition \ref{psimptal}),
they have homotopy equivalences \w{f:X\to Y} and \w{g:Y\to X} as in
Lemma \ref{lcylin}(a). We make explicit the construction of the Lemma
by producing a CW basis \w{(\oV{n})\sb{n\in\NN}} for \w{Z=\Vd} in
\wref[,]{eqhcobordism} together with inclusions of coproduct summands
\w{\ophin{n}{i}:\ouV{n}{i}\hra\oV{n}} for each \w{n\geq 0} as in Definition \ref{dacomp}.

If we write \w{e\sb{0},e\sb{1}\in\Delta[1]\sb{0}} and
\w{\sigma\in\Delta[1]\sb{1}} for the non-degenerate simplices, we have
\w[,]{V\sb{n}=\oVp{n}\amalg L\sb{n}\Vd} with:
\begin{myeq}\label{eqvn}
\begin{split}
\oVp{n}~:=&~~~\ouV{n}{0}\otimes(\Sn{n-1}e\sb{0})~~~~\amalg~~~~
\coprod\sb{k=0}\sp{n-1}~~
[\ouV{n}{0}\otimes(\Snk{n-1}{k}\sigma)]\\
&~\amalg~~\coprod\sb{k=0}\sp{n-1}~[s\sb{k}\ouV{n-1}{0}\otimes(\Snk{n-1}{k}\sigma)]\\
&~\amalg~~[\ouV{n}{1}\otimes(\Sn{n-1}e\sb{1})]~.
\end{split}
\end{myeq}
\noindent Here \w{\Sn{n}} is the iterated degeneracy map \w{s\sb{n}\dotsc s\sb{0}}
and
\w[.]{\Snk{n}{k}:=s\sb{n}\dotsc s\sb{k+1}\widehat{s\sb{k}}s\sb{k-1}\dotsc s\sb{0}}

The face maps are calculated as usual on each factor of
\w[,]{a\otimes b} except that
\begin{myeq}\label{eqfface}
d\sb{n}(s\sb{n-1}u\otimes(\Sn{n-2}\sigma))~=~
f\sb{n-1}u\otimes(\Sn{n-2}e\sb{1})~\in~[\ouV{n-1}{1}\otimes(\Sn{n-2}e\sb{1})],
\end{myeq}
\noindent for \w{k=n-1} in the second line of \wref[,]{eqvn} by
\wref[,]{eqhcobordism} where \w{f:\Vud{0}\to\Vud{1}} is the chosen
homotopy equivalence.

Note that \w{\oVp{n}} is not a CW basis object for \w[,]{\Vd}
since the summands in the second line of \wref{eqvn} always have at least two
non-vanishing face maps.

However, for any \w{v\in\oVp{n}} we can define
\w{v\bup{0}:=v} and \w{v\bup{k+1}:=v\bup{k}-s\sb{n-k-1}d\sb{n-k}v\bup{k}}
by induction on $k$, and find that \w{d\sb{i}v\bup{k}=0} for \w[,]{n-k<i\leq n}
so \w{v\bup{n}} is a Moore chain. Note that for \w{v\in U} in summands $U$ on
the first and third lines of \wref{eqvn} we have simply \w[.]{v=v\bup{n}}

Explicitly, we replace each generator \w{v=s\sb{k}u\otimes(\Snk{n-1}{k}\sigma)}
of a summand in the second line of \wref{eqvn} for \w{\oVp{n}} by
\begin{myeq}\label{eqlastsummand}
v\bup{n}=\begin{cases}
\sum\sb{i=0}\sp{k}(-1)\sp{i}\left[s\sb{k-i}u\right]\otimes
\left[(\Snk{n-1}{k}\sigma)-(\Snk{n-1}{k+1}\sigma)\right]&\text{if}~k<n-1\\
\sum\sb{i=1}\sp{n}\,(-1)\sp{i}
\left[\,s\sb{n-i}f\sb{n-1}u\otimes(\Sn{n-2}e\sb{1})
-s\sb{n-i}u\otimes(\Sn{n-2}\sigma)\,\right]& \text{if}~k=n-1.
\end{cases}
\end{myeq}
\noindent Thus \w{v\bup{n}} always has the form \w[,]{v+\sum\sb{i=1}\sp{k}\,u\sb{i}}
with the elements \w{u\sb{i}} all degenerate.

The \TRal retraction \w{\rho':V\sb{n}\to\oVp{n}} onto the summand \w{\oVp{n}}
therefore takes \w{v\bup{n}} to $v$. This implies that if
\w{\{v\sb{i}\}\sb{i\in I}} are
generators of \w[,]{\oVp{n}} the new elements \w{\{v\bup{n}\sb{i}\}\sb{i\in I}}
still generate a free sub-\TRal of \w[,]{V\sb{n}} because any relation of the form
\w{\psi(v\bup{n}\sb{1},\dotsc,v\bup{n}\sb{k})=0} in \w[,]{V\sb{n}}
where $\psi$ is some primary $R$-cohomology operation, implies that also
$$
0~=~\rho'(\psi(v\bup{n}\sb{1},\dotsc,v\bup{n}\sb{k}))~=~
\psi(\rho'(v\bup{n}\sb{1}),\dotsc,\rho'(v\bup{n}\sb{k}))~=~
\psi(v\sb{1},\dotsc,v\sb{k})
$$
\noindent which can hold only if \w{\psi\equiv 0} since the elements \w{v\sb{i}} are
generators of a free \TRal[.]

Therefore,  if we write \w{\oV{n}} for the sub-\TRal of \w{V\sb{n}} generated
by all elements \w[,]{v\bup{n}} as $v$ varies over a set of generators for (each
summand of) \w[,]{\oVp{n}} we still have a coproduct of free \TRal[s]
\w[,]{V\sb{n}=\oV{n}\amalg L\sb{n}\Vd} where \w{\oV{n}} now serves as an $n$-th
CW basis object for \w[.]{\Vd}

Moreover, we still have inclusions of coproduct summands
\w{\ophin{n}{i}:\ouV{n}{i}\hra\oV{n}} inducing split trivial cofibrations of
simplicial \TRal[s] \w{\varphi\up{i}:\Vud{i}\to\Vd} \wb[.]{i=0,1}
We can use these to further write
\w[.]{\oV{n}:=\oU{n}~\amalg~\ouV{n}{0}~\amalg~\ouV{n}{1}}
\end{proof}

\begin{defn}\label{dcorresp}
Given an algebraic comparison map
\w[,]{\Psi=\lra{\varphi,\,\rho,\,(\ophl{n},\,\orh{n})\sb{n\in\NN}}}
between two CW resolutions \w{\vare:\Vd\to\Gamma} and \w{\varep:\Vdp\to\Gamma}
of a realizable \TRal $\Gamma$ (see \S \ref{dacomp}), and sequential
realizations $\cW$ and \w{\ccWp} of \w{\Vd} and \w[,]{\Vdp} respectively,
a \emph{comparison map \w{\Phi:\cW\to\,\ccWp} over $\Psi$} is a system
\begin{myeq}\label{eqcorresp}
\Phi~=~\lra{\en{n},\,\rn{n},~(\Pon{k}{n})\sb{k=0}\sp{n-1},~
(\oon{k}{n})\sb{k=0}\sp{n-1},~(\Por{k}{n})\sb{k=0}\sp{n-1},~
(\oor{k}{n})\sb{k=0}\sp{n-1}\,}\sb{n\in\NN}
\end{myeq}
\noindent consisting of:
\begin{enumerate}
\renewcommand{\labelenumi}{(\roman{enumi})~}
\item Split fibrations of the modified path-loop fibrations of \wref{eqmodpathloop}
fitting into a diagram:
\mydiagram[\label{eqmapmodpathloop}]{
\ooWp{k+1}{n} \ar@{->>}[d]\sp{\oon{k+1}{n}}  \ar@{^{(}->}[rr]\sp{\ppp\ip{k+1}} &&~
\PoWp{k}{n} \ar@{->>}[d]\sp{\Pon{k}{n}} \ar@{->>}[rr]\sp{\ppp\pp{k}} &&
\ooWp{k}{n} \ar@{->>}[d]\sp{\oon{k}{n}}\\
\ooW{k+1}{n}  \ar@/^{1.1pc}/[u]\sp{\oor{k+1}{n}} \ar@{^{(}->}[rr]\sp{\ip{k+1}} &&
\PoW{k}{n} \ar@/^{1.1pc}/[u]\sp{\Por{k}{n}} \ar@{->>}[rr]\sp{\pp{k}} &&
\ooW{k}{n} \ar@/^{1.1pc}/[u]\sp{\oor{k}{n}}
}
\noindent for each \w[,]{0\leq k<n} in which both upward and downward squares
commute, as well as:
\begin{myeq}\label{eqsplitfib}
\Pon{k}{n}\circ\Por{k}{n}=\Id\hsp\text{and}\hsp
\oon{k}{n}\circ \oor{k}{n}=\Id\hs \text{for all}\hs 0\leq k < n~.
\end{myeq}
\noindent We require that for all \w[,]{0\leq k< n} the maps \w{\oon{k}{n}}
realize \w{\Omega\sp{k}\ovp\sb{n}} and the maps \w{\oor{k}{n}} realize
\w[.]{\Omega\sp{k}\orh{n}}
\item A cosimplicial map \w{\en{n}:\Wp{n}\to\W{n}}
realizing \w{\varphi:\Vd\to\Vdp} through simplicial dimension $n$, with section
\w{\rn{n}:\W{n}\to\Wp{n}} realizing $\rho$, such that
for each \w[,]{0\leq k<n} both squares in the following diagram commute:
\mytdiag[\label{eqcompf}]{
C\sp{k}(\Wp{n-1}) \ar[d]\sb{\Fkp{k}} \ar@{->>}[rr]\sp{C\sp{k}(\en{n-1}\sp{k})} &&
C\sp{k}(\W{n-1})\ar[d]\sp{\Fk{k}} & C\sp{k}(\Wp{n-1}) \ar[d]\sb{\Fkp{k}} &&
C\sp{k}(\W{n-1})\ar[d]\sp{\Fk{k}} \ar@{_{(}->}[ll]\sb{C\sp{k}(\rnk{n-1}{k})} \\
\PoWp{n-k-2}{n} \ar@{->>}[rr]\sb{\Pon{n-k-2}{n}} && \PoW{n-k-2}{n} &
\PoWp{n-k-2}{n} && \PoW{n-k-2}{n}. \ar@{_{(}->}[ll]\sp{\Por{n-k-2}{n}}
}
\end{enumerate}

When each map \w{\oon{k}{n}:\ooWp{k}{n}\epic\ooW{k}{n}}
and \w{\en{n}\sp{k}:\Wpn{k}{n}\epic\Wn{k}{n}} is a trivial fibration (and
thus each map  \w{\oor{k}{n}:\ooW{k}{n}\hra\ooWp{k}{n}} and
\w{\rnk{n}{k}:\Wpn{k}{n}\hra\Wn{k}{n}} is a trivial cofibration in \w[),]{\Sa}
we say that $\Phi$ is a \emph{trivial} comparison map.

If we only have
\begin{myeq}\label{eqncorresp}
\Phi~=~\lra{\en{n},\,\rn{n},~(\Pon{k}{n})\sb{k=0}\sp{n-1},~
(\oon{k}{n})\sb{k=0}\sp{n-1},~(\Por{k}{n})\sb{k=0}\sp{n-1},~
(\oor{k}{n})\sb{k=0}\sp{n-1}\,}\sb{n=0}\sp{N}
\end{myeq}
\noindent as above, we say that \w{\Phi:\cW\to\,\ccWp} is an \emph{$N$-stage
comparison map} over $\Psi$.
\end{defn}

\begin{remark}\label{rcomp}
If we let \w{\jnk{n}{k}:\ooX{k}{n}\hra\ooWp{k}{n}} denote the inclusion of
the fiber of \w[,]{\oon{k}{n}:\ooWp{k}{n}\epic\ooW{k}{n}} we see that the splitting
\w{\oor{k}{n}:\ooW{k}{n}\hra\ooWp{k}{n}} induces a retraction
\w{\snk{n}{k}:\ooWp{k}{n}\epic\ooX{k}{n}} for \w[,]{\jnk{n}{k}}
defined \w[,]{x\mapsto x-\oor{k}{n}\oon{k}{n}(x)} and thus a map
\mydiagram[\label{eqwkprod}]{
\ooWp{k}{n} \ar[rr]_<<<<<<<<<{\simeq}^<<<<<<<<<{\snk{n}{k}\top\oon{k}{n}} &&
\quad~~\ooX{k}{n}\times\ooW{k}{n}
}
\noindent which is a weak equivalence for each \w{0\leq k< n}
(using the abelian group structure on all spaces). As we shall see,
in many cases we can assume \wref{eqwkprod} is actually an equality.
\end{remark}

\begin{defn}\label{dzigzag}
A \emph{zigzag of comparison maps} between two sequential
realizations \w{\cuW{0}} and \w{\cuW{1}} of a realizable \TRal $\Gamma$
is a (possibly infinite) sequence of cospans of comparison maps
starting from \w{\cuW{0}} and ending at \w[,]{\cuW{1}} which is
\emph{locally finite} in the sense that for each
\w[,]{n\geq 0} only finitely many of the comparison maps in the zigzag between the
$n$-th stages \w{(\W{n})\up{0}} and \w{(\W{n})\up{1}} are not the identity map.

We say that two abstract sequential realizations \w{\cuW{0}} and
\w{\cuW{1}} (of arbitrary spaces \w{\bY\up{0}} and \w[)]{\bY\up{1}}
are \emph{weakly equivalent} if they are related by a zigzag
of comparison maps.

Similarly, if only the $n$-th stages \w{(\W{n})\up{0}} and \w{(\W{n})\up{1}} of
two such sequential realizations \w{\cuW{0}} and \w{\cuW{1}} are
related by a zigzag of comparison maps, we say that  \w{\cuW{0}} and \w{\cuW{1}}
are \emph{$n$-equivalent}, or that \w{(\W{n})\up{0}} and \w{(\W{n})\up{1}} are
\emph{weakly equivalent}\vsm.
\end{defn}

\begin{thm}\label{tcomp}
Any two cofibrant sequential realizations \w{\cuW{0}} and \w{\cuW{1}} of $\bY$
are weakly equivalent.
\end{thm}

\begin{proof}
We prove the Theorem in two main steps\vsm:

\noindent\textbf{(i)~~Different algebraic resolutions\vsn:}

We first show that, given an algebraic comparison map \w{\Psi:\Vd\to\Vdp} for $\bY$
and a cofibrant sequential realization $\cW$ of \w[,]{\Vd} there is a
cofibrant sequential realization
\w{\ccWp} of \w{\Vdp} with a comparison map \w{\Phi:\cW\to\,\ccWp} over $\Psi$,
constructed (with the maps \w{\en{n}:\Wp{n}\to\W{n}}
and sections \w[)]{\rn{n}:\W{n}\to\Wp{n}} by induction on \w[:]{n\geq 0}

At the $n$-th stage, we may assume by Lemma \ref{lcylin} that \w{\en{n-1}}
is a fibration and \w{\rn{n-1}} is a cofibration in the resolution model
category \w[,]{\Chn{\C}{n}} so in particular \w{\enk{n-1}{j}} is
a fibration and \w{\rnk{n-1}{j}} a cofibration for \w{0\leq j\leq n-1}
(see \cite[\S 3.2]{BousC}).

Since \w{\oV{n}} is a coproduct summand in \w[,]{\oVp{n}=\oV{n}\amalg\oU{n}}
the map \w{\ophl{n}:\oV{n}\hra\oVp{n}} is simply the inclusion, while
\w{\orh{n}:\oVp{n}\hra\oV{n}} has the form \w[.]{\Id\bot\zeta}
If we realize \w{\oV{n}} by \w{\oW{n}} and \w{\oU{n}} by \w{\oX{n}}
then \w{\oVp{n}} is realized by \w[.]{\hWp{n}:=\oX{n}\times\oW{n}}
By Definition \ref{dscr}, the $n$-th stage of $\cW$ is determined
by the choice of left Reedy fibrant replacement \w{\bDs}
of \w[,]{\oW{n}\ouS{n-1}} equipped with a left Reedy cofibration
\w{F:\cMs\W{n-1}\to\bDs} realizing the given attaching map
\w[.]{\odz{n}:\oV{n}\to C\sb{n-1}\Vd}

If \w{\bKs} is similarly a left Reedy fibrant replacement for
\w[,]{\oX{n}\ouS{n-1}} the attaching map
\w{\ppp\odz{n}:\oVp{n}\to C\sb{n-1}\Vdp} has the form \w[,]{\odz{n}\bot\tau}
and we may realize \w{\tau:\oU{n}\to C\sb{n-1}\Vdp} by
\w{T:\cMs\Wp{n-1}\to\bKs} (not a cofibration) and \w{\zeta:\oU{n}\to\oV{n}} by
\w[.]{Z:\bDs\to\bKs}

Consider the following diagram in the left Reedy model category of $n$-truncated
cochain complexes over $\C$, in which \w{\bPs} is the pushout of the upper
left square, and the map $p$ with section  $r$ is induced by \w[:]{\cMs\rn{n-1}}

\mydiagram[\label{eqpoatt}]{
\ar @{} [drr]|(0.71){\framebox{\scriptsize{PO}}}
\cMs\W{n-1} \ar@{^{(}->}[d]\sp{F} \ar@{^{(}->}[rr]\sb{\cMs\rn{n-1}}
&& \cMs\Wp{n-1} \ar@{^{(}->}[d]\sp{j} \ar@/_1.1pc/[ll]\sb{\cMs\en{n-1}}
\ar@/^1.0pc/[rrd]\sp{T} && \\
\bDs  \ar@/_2.5pc/[rrrr]\sp{Z} \ar@{^{(}->}[rr]\sp{r} &&
\bPs \ar@/^0.9pc/[ll]\sp{p} \ar@{.>}[rr]\sb{S} && \bKs~,
}

Since by Definition \ref{dacomp}
\w[,]{C\sb{n-1}\rho\circ\ppp\odz{n}=\odz{n}\circ\orh{n}} also
\w[,]{C\sb{n-1}\rho\circ\tau=\odz{n}\circ\orh{n}\rest{\oU{n}}=\odz{n}\circ\zeta}
so the outer square in \wref{eqpoatt} commutes up to homotopy.  Since
$F$ is a cofibration, we may change $Z$ up to homotopy to make it commute on the
nose by \cite[Lemma 5.11]{BJTurnR}. The maps $Z$ and $T$ then induce $S$ as
indicated. This allows us to extend \wref{eqpoatt} to a commuting diagram
\mydiagram[\label{eqpoat}]{
\cMs\W{n-1} \ar@{^{(}->}[d]\sp{F} \ar@{^{(}->}[rr]\sb{\cMs\rn{n-1}}
&& \cMs\Wp{n-1} \ar@{^{(}->}[d]\sp{j} \ar@/_1.1pc/[ll]\sb{\cMs\en{n-1}}
\ar@/^1.0pc/[rrd]\sp{T\top(p\circ j)} && \\
\bDs  \ar@{^{(}->}[rr]\sp{r} &&
\bPs \ar@/^0.9pc/[ll]\sp{p} \ar[rr]\sp{S\top p} &&
\bKs\times\bDs \ar@/^2.5pc/[llll]\sb{\proj}~,
}
\noindent We now factor \w{S\top p} as a cofibration \w{G':\bPs\hra\bEs}
followed by a trivial fibration \w{t:\bEs\epic\bKs\times\bDs}
(in the left Reedy model structure on truncated cochain complexes).
If we set \w{G:\cMs\Wp{n-1}\to\bEs} equal to
\w[,]{G'\circ j} \w{\overline{r}:\bEs\to\bDs} equal to \w[,]{\proj\circ t}
and \w{\overline{e}:\bDs\to\bEs} equal to the cofibration \w[,]{G'\circ e}
we see that \w{\bEs} is a left Reedy fibrant replacement for
\w{\hWp{n}\ouS{n-1}} (since \w{\bKs\times\bDs} is a product of
fibrant objects), $G$ is a left Reedy cofibration, and they fit into a diagram
\mydiagram[\label{eqpba}]{
\cMs\W{n-1} \ar@{^{(}->}[d]\sp{F} \ar@{^{(}->}[rr]\sb{\cMs\rn{n-1}} &&
\cMs\Wp{n-1} \ar@{^{(}->}[d]\sp{G} \ar@/_1.1pc/[ll]\sb{\cMs\en{n-1}} \\
\bDs \ar@{^{(}->}[rr]\sp{\overline{e}} && \bEs \ar@/^1.7pc/[ll]\sb{\overline{r}}
}
\noindent in which both the left and right squares commute, and
\w[.]{\overline{r}\circ\overline{e}=\Id}

Applying the functorial procedure of \S \ref{dscr}(c) to the two vertical
arrows in \wref{eqpba} again we obtain $n$-stage comparison map
\w{\Phi:\cW\to\cWp} extending the given \wwb{n-1}stage comparison map\vsm.

\noindent\textbf{(ii)~~Reducing to the case of one  algebraic resolution\vsn:}

Assume \w{\cuW{0}} and \w{\cuW{1}} are associated respectively to the two
CW resolutions \w{\Vud{0}} and  \w{\Vud{1}} of the \TRal
\w[,]{\Gamma=\HiR{\bY}} with CW bases
\w{(\ouV{n}{i})\sb{n\in\NN}} for \w[\vsm.]{i=0,1}

By Lemma \ref{lcylind}, there is a third CW resolution
\w[,]{\ppp\vare:\Vdp\to\Gamma} with CW basis \w[,]{(\oVp{n})\sb{n\in\NN}}
equipped with algebraic comparison maps \w{\Psi\up{i}:\Vud{i}\to\Vdp} \wb[.]{i=0,1}
By Step \textbf{(i)}, there are then two sequential
realizations \w{\cWpi{i}} of \w[,]{\Vdp\to\Gamma}  for \w[,]{i=0,1}
each equipped with a comparison map \w{\Psi\up{i}:\cuW{i}\to\cWpi{i}}
over \w[.]{\Psi\up{i}}
Thus we are reduced to dealing with the case where the two (cofibrant) sequential
realizations \w{\cuW{0}} and \w{\cuW{1}} (i.e., the \w{\cWpi{i}}
just constructed) are of the same CW resolution
\w[,]{\Vd\to\Gamma} with CW basis \w[.]{(\oV{n})\sb{n\in\NN}}
We construct a zigzag of comparison maps between them, by induction on
\w[:]{n\geq 0}

We assume by induction the existence of a cospan of \wwb{n-1}stage trivial
comparison maps \w{\Phi\up{i}:\cuW{i}\to\cW} \wb{i=0,1} over \w[.]{\Id\sb{\Vd}}
For \w{n=0} this is simply \w[.]{\Wi{-1}{0}=\cu{\bY}=\Wi{-1}{1}}

By Definition \ref{dscr}(c), the $n$-th stage for \w{\cuW{i}} is determined
by the choice of left Reedy fibrant replacements \w{\uDs{i}}
of \w{\oW{n}\ouS{n-1}} (where \w{\oW{n}} is some realization of the $n$-th
algebraic CW basis object \w[),]{\oV{n}} together with left Reedy cofibrations
\w{F\up{i}:\cMs\Wi{n-1}{i}\to\uDs{i}} \wb{i=0,1} realizing the given
attaching map \w[.]{\odz{n}:\oV{n}\to C\sb{n-1}\Vd}

For \w[,]{i=0,1} consider the following diagram in the left Reedy model category
\w[,]{\Chn{\C}{n}} in which \w{\bPs} is again the pushout of the upper left square
\mydiagram[\label{eqpoattach}]{
\ar @{} [drr]|(0.74){\framebox{\scriptsize{PO}}}
\cMs\Wi{n-1}{i} \ar@{^{(}->}[d]\sp{F\lo{i}}
\ar@{^{(}->}[rr]\sb{\cMs\rn{n-1}}\sp{\simeq}
&& \cMs\Wp{n-1} \ar@{^{(}->}[d]\sp{j\lo{i}}
\ar@/_1.1pc/[ll]\sb{\cMs\en{n-1}}  \ar@{.>}[rrd]\sp{G\lo{i}}  &&&&\\
\uDs{i}  \ar@{^{(}->}[rr]\sp{r\lo{i}}\sb{\simeq} &&
\uPs{i} \ar@/^1.5pc/[ll]\sb{p\lo{i}} \ar@{^{(}->}[rr]\sp{S\lo{i}}\sb{\simeq} &&
\uEs{i} \ar@/^3.0pc/[llll]\sb{\pi\lo{i}}\sp{\simeq} \ar@{->>}[r] & \ast
}
\noindent The retraction \w{p\lo{i}} for the trivial cofibration \w{r\lo{i}}
is induced by the retraction \w{\cMs\eni{n-1}{i}} for the trivial cofibration
\w[,]{\cMs\rni{n-1}{i}} so \w{p\lo{i}} is a weak equivalence.
Factor \w{p\lo{i}} as a trivial cofibration \w{k\lo{i}:\uPs{i}\to\uEs{i}}
followed by a fibration \w{\pi\lo{i}:\uEs{i}\to\uDs{i}} (also a weak
equivalence), so \w{\uEs{i}} is in particular a fibrant replacement for
\w{\uPs{i}} since \w{\uEs{i}} is fibrant.

Now set \w{G\lo{i}:=S\lo{i}\circ j\lo{i}:\cMs\Wp{n-1}\hra\uEs{i}} (a cofibration).
Because the maps \w{\pi\sb{i}\circ G\lo{i}=F\lo{i}\circ\cMs\en{n-1}} realize
the same algebraic attaching map \w{\phi:\oV{n}\olS{n-1}\to\cMs\Vd} for
\w[,]{i=0,1} they are weakly equivalent in the arrow category of
\w[.]{\Chn{\C}{n-1}} Thus \w{G\up{0}} and \w{G\up{1}} are weakly equivalent
fibrant and cofibrant objects in the under category
\w{\cMs\Wp{n-1}\backslash\Chn{\C}{n-1}} with its standard model category
structure (see \cite[Theorem 7.6.5(a)]{PHirM}). We can therefore apply
Lemma \ref{lcylin}(b) to obtain an intermediate object $G$ fitting into the
following diagram, in which all four triangles commute, and
\w{s\lo{i}\circ f\lo{i}=\Id} \wb[:]{i=0,1}

\mydiagram[\label{eqmaparrows}]{
&& \cMs\Wp{n-1} \ar@{^{(}->}[d]\sp{G} \ar@{^{(}->}[lld]\sb{G\lo{0}}
\ar@{^{(}->}[rrd]\sp{G\lo{1}} && \\
\uEs{0} \ar@{_{(}->}[rr]\sp{f\lo{0}}\sb{\simeq}  &&
\ppp\bEs \ar@/_1.5pc/[rr]\sb{s\lo{1}}
\ar@/^1.5pc/[ll]\sp{s\lo{0}} &&
\uEs{1} \ar@{^{(}->}[ll]\sb{f\lo{1}}\sp{\simeq}\\
}

Again applying the functors of \S \ref{dscr}(c) to all three downward arrows
of \wref{eqmaparrows} yields a new $n$-stage sequential
realization \w{\cWp} (corresponding to \w[),]{G:\cMs\W{n-1}\hra \ppp\bEs}
with two new $n$-stage trivial comparison maps
\w{\Phip{i}:\cWpi{i}\to\cWp} \wb[.]{i=0,1}

The two composites:
\mydiagram[\label{eqnewcospan}]{
\cuW{0} \ar[r]\sp{\Phi\up{0}} & \cWpi{0} \ar[r]\sp{\Phip{0}} & \cWp &
\cWpi{1} \ar[l]\sb{\Phip{1}} & \cuW{1} \ar[l]\sb{\Phi\up{1}}
}
\noindent then yield the required cospan of $n$-stage comparison maps.
\end{proof}
%
%
\sect{Higher cohomology operations}
\label{chhc}

The notion of secondary and higher cohomology operations has a long history in
homotopy theory, going back to  the 1950's, but there is no completely
satisfactory general theory of such operations.  Here we follow the point of
view taken in \cite{BMarkH,BJTurnHH}, where they are subsumed under the notion
of general pointed higher homotopy operations.

We want to think of a cofibrant sequential realization $\cW$ for a
space $\bY$ as providing a template for an infinite sequence of operations of order
$n$ \wb[,]{n=2,3,\dotsc} potentially acting on any space $\bZ$ with
\w{\HiR{\bZ}} (abstractly) isomorphic to \w[.]{\HiR{\bY}} The operation of
order $n$ is defined only when specific choices have been made inductively for all
lower order operations in such a way that they all vanish.

The sequence of such choices is called a ``strand'' of the higher cohomology
operation associated to the given sequential realization $\cW$,
and the corresponding ``system of higher cohomology operations'' will be an
equivalence class of strands under comparison maps.

\begin{defn}\label{dshho}
Let  \w{\Gamma=\HiR\bY} be the \TRal associated to a space $\bY$, and
\w{\vare:\Vd\to\Gamma} a CW resolution with CW basis \w[.]{(\oV{n})\sb{n\in\NN}}
Assume we are also given \emph{initial data} consisting of a sequential
realization \w{\cW=\lra{\W{n},\,\tWu{n}}\sb{n\in\NN}}
of \w[,]{\Vd} and an isomorphism of  \TRal[s] \w{\vth:\Gamma\to\HiR{\bZ}}
for some space $\bZ$.

An \emph{$n$-strand} \wb{0\leq n\leq\infty}
\w{\Stl{n}=(\bve{0},\bve{1},\dotsc,\bve{n})} for \w{(\cW,\bZ,\vth)} consists of
a compatible collection of coaugmentations \w{\bve{k}:\bZ\to\W{k}}
\wb{k=0,\dotsc,n} realizing \w{\vth\circ\vare:\Vd\to\HiR{\bZ}} through simplicial
dimension $n$.
Compatibility means that \w{\bve{k-1}=\prn{k}\circ\bve{k}} (see \S \ref{dscr}(c)).
In particular, an \emph{$\infty$-strand} for \w{(\cW,\bZ,\vth)} is an infinite
sequence \w{\Stl{\infty}:=(\bve{0},\dotsc,\bve{n},\dotsc)} of such compatible
coaugmentations.

Given an \wwb{n-1}strand \w{\Stl{n-1}} for \w[,]{(\cW,\bZ,\vth)}
consider the composite $\xi$ of
$$
\bZ~\xra{\bve{n-1}}~\Wn{0}{n-1}~\xra{\Flk{0}{n-1}}~\PoW{n-2}{n}~
\xra{\pp{n-2}}~\ooW{n-2}{n}~\xra{\ip{n-2}}~\PoW{n-3}{n}~,
$$
\noindent which by \wref{eqfirstcoface} represents the component
\w{\qk{2}{n-1}\circ d\sp{1}\circ d\sp{0}\circ \bve{n-1}} of the iterated coface map
from $\bZ$ into \w[.]{\PoW{n-3}{n}} As in Lemma \ref{linducea}, since
\w{d\sp{1}\circ d\sp{0}\circ \bve{n-1}=d\sp{2}\circ d\sp{1}\circ \bve{n-1}} and
\w{\qk{2}{n-1}\circ d\sp{2}=0} by \wref[,]{eqjcoface} we see that $\xi$ is the
zero map. Since \w{\ip{n-2}} is monic, this means that the composite
\w{\pp{n-2}\circ\Flk{0}{n-1}\circ\bve{n-1}} is already zero, so
\w{\Flk{0}{n-1}\circ\bve{n-1}} factors through the fiber
\w{\ooW{n-1}{n}} of \w[.]{\pp{n-2}} We denote the resulting map by
\w[,]{\alk{-1}{n-1}:\bZ~\to~\ooW{n-1}{n}} with
\begin{myeq}\label{eqikakm}
\ip{n-1}\circ\alk{-1}{n-1}~=~\Flk{0}{n-1}\circ \bve{n-1}~,
\end{myeq}
\noindent as in \wref[.]{eqikak}

Note that $\vth$ induces an isomorphism
\w{[\bZ,~\ooW{n-1}{n}]\cong\Gamma\lin{\ooW{n-1}{n}}}
so the homotopy class \w{[\alk{-1}{n-1}]} may be identified with an element
\begin{myeq}\label{eqvalue}
\Val{\Stl{n-1}}\in\Gamma\lin{\ooW{n-1}{n}}~\cong~\Gamma\lin{\Omega\sp{n-1}\oW{n}}~,
\end{myeq}
\noindent called the \emph{value} of the \wwb{n-1}strand \w[.]{\Stl{n-1}}
\end{defn}

\begin{remark}\label{rshho}
We do not need the full sequential realization $\cW$ to define
\w[,]{\Val{\Stl{n}}} but only the restricted cosimplicial set \w{\tWu{n}}
of its \wwb{n-1}st stage \w[.]{\W{n-1}}
Thus a \emph{$0$-strand} is completely determined by a choice of a realization
\w{\bv:\bZ\to\Wn{0}{0}} of \w[.]{\vth\circ\vare:V\sb{0}\to\HiR{\bZ}} Such a $\bv$
always exists, and is unique up to homotopy.

Note also that Definition \ref{dshho} can be stated purely in the language
of \ww{\TsR}-\ma[s] \wh see \cite{BBlaC} and Appendix \ref{apfthm} below.
\end{remark}

\begin{lemma}\label{lvanish}
Given an \wwb{n-1}strand \w{\Stl{n-1}=(\bve{0},\bve{1},\dotsc,\bve{n-1})}
for \w[,]{(\cW,\bZ,\vth)} the coaugmentation \w{\bve{n-1}:\bZ\to\W{n-1}}
extends to a coaugmentation \w{\bve{n}:\bZ\to\W{n}} if and only if
\w{\Val{\Stl{n-1}}=0} in \w[.]{\Gamma\lin{\Omega\sp{n-1}\oW{n}}}
\end{lemma}

\begin{proof}
If the value is zero, we can choose a nullhomotopy
\w{\Flk{-1}{n}:\bZ\to\PoW{n-1}{n}} for \w[,]{\alk{-1}{n-1}} (in the sense of
\cite[I, \S 2]{QuiH}), with
\w[.]{\pp{n-1}\circ\Flk{-1}{n}=\alk{-1}{n-1}} This extends \w{\bve{n-1}} to a
coaugmentation \w{\tbve{n}:\bZ\to\tWn{0}{n}} as in the proof of Theorem
\ref{tresg}, which also shows how to extend \w{\tbve{n}} to a coaugmentation
for \w[,]{\tWu{n}} and thus (after making it Reedy cofibrant) for \w[.]{\W{n}}

Conversely, since \w{\Wn{0}{n}=\tWn{0}{n}=\Wn{0}{n-1}\times\PoW{n-1}{n}}
by \wref{eqdopb} and step \textbf{(iii)} of \S \ref{senso}, given an
extension \w{\bve{n}:\bZ\to\Wn{0}{n}} of \w[,]{\bve{n-1}} we can compose it
with the projection \w{\qk{0}{n}:\tWn{0}{n}\to\PoW{n-1}{n}} to obtain a
nullhomotopy \w{\Flk{-1}{n}} for \w[.]{\alk{-1}{n-1}}
\end{proof}

\begin{mysubsection}{Correspondence of strands}
\label{scorrespstr}
Note that if \w{\Psi:\Vd\to\Vdp} is an algebraic comparison map between two
CW resolutions for a \TRal $\Gamma$, as in \wref[,]{eqacomp} the mutually inverse
weak equivalences \w{\varphi:\Vd\to\Vdp}  and \w{\rho:\Vdp\to\Vd} induce mutually
inverse isomorphisms of \TRal[s]
\w[.]{\varphi\sb{\#}:\pi\sb{0}\Vd\cong\Gamma\cong\pi\sb{0}\Vdp:\rho\sb{\#}}

On the other hand, if \w{\Phi:\cW\to\,\ccWp} is an
$n$-stage comparison map between two sequential realizations over
$\Psi$, as in \wref[,]{eqncorresp} then
\w{(\oar{n-1})\sb{\ast}:\Gamma\lin{\Omega\sp{n-1}\oW{n}}\hra
\Gamma\lin{\Omega\sp{n-1}\oWp{n}}} is just a split inclusion, with retraction
\w[.]{(\oen{n-1})\sb{\ast}:\Gamma\lin{\Omega\sp{n-1}\oWp{n}}\epic
\Gamma\lin{\Omega\sp{n-1}\oW{n}}}

Let \w{\vth:\Gamma\cong\HiR{\bZ}} be an isomorphism of \TRal[s,] and let
\w{\Stl{n}} and \w{\Stlp{n}} be two $n$-strands for the initial
data \w{(\cW,\bZ,\vth)} and \w[,]{(\ccWp,\bZ,\vth)} respectively.
If \w{\Phi=\lra{\en{n},\,\rn{n},~(\Pon{k}{n})\sb{k=0}\sp{n-1},~
(\oon{k}{n})\sb{k=0}\sp{n-1},~(\Por{k}{n})\sb{k=0}\sp{n-1},~
(\oor{k}{n})\sb{k=0}\sp{n-1}\,}\sb{n=0}\sp{N}} is an $n$-stage comparison map
as above, we write
\w{\Stlp{n}=\rs(\Stl{n})} if \w{\bvep{k}=\rnk{k}{0}\circ\bve{k}} for
each \w[,]{0\leq k\leq n} and \w{\Stl{n}=\es(\Stlp{n})} if
\w{\bve{k}=\enk{k}{0}\circ\bvep{k}} for each \w[.]{0\leq k\leq n}

From \wref[,]{eqmapmodpathloop} \wref[,]{eqcompf} and \wref{eqikakm} we see that
\begin{myeq}\label{eqcorrvals}
\Val{\rs(\Stl{n})}=\oar{n}\sb{\ast}(\Val{\Stl{n}})\hsp \text{and}\hsp
\Val{\es(\Stlp{n})}=\oen{n}\sb{\ast}(\Val{\Stlp{n}})~.
\end{myeq}
\noindent Therefore, given an $n$-stage comparison map
\w{\Phi:\cW\to\,\ccWp} as above, an $n$-strand \w{\Stl{n}} for $\cW$ and an
$n$-strand \w{\Stlp{n}} for \w[,]{\cWp} we see that:
\begin{myeq}\label{eqvancorr}
\begin{array}{l}
(a)\hsn\Val{\Stl{n}}=0\hsn \text{if and only if}\hsn \Val{\rs(\Stl{n})}=0~.\\
(b)\hsm\text{If}\hsn \Val{\Stlp{n}}=0\hsn\text{then}\hsn \Val{\es(\Stlp{n})}=0,
\text{but not necessarily conversely}.
\end{array}
\end{myeq}

This explains the need for the following:
\end{mysubsection}

\begin{defn}\label{dequivstr}
Given spaces $\bY$ and $\bZ$ with \w[,]{\vth:\HiR{\bY}\xra{\cong}\HiR{\bZ}}
we define two equivalence relations $\sim$ and $\approx$ on $n$-strands for $\bZ$
(with respect to various sequential realizations):

The \emph{weak equivalence} relation of strands $\sim$ is generated by
\w{\es(\Stlp{n})\sim\Stlp{n}} for any $n$-strand \w{\Stlp{n}} for
\w{(\cWp,\bZ,\vth)} and any comparison map \w[.]{\Phi:\cW\to\,\ccWp} We
denote the corresponding equivalence class by \w[.]{[\Stl{n}]}

The \emph{strong equivalence} relation of strands $\approx$ is generated by the
relation \w{\es(\Stlp{n})\approx\Stlp{n}} for any $n$-strand \w{\Stlp{n}}
for \w{(\cWp,\bZ,\vth)} and any comparison map \w{\Phi:\cW\to\,\ccWp}
satisfying:
\begin{myeq}\label{eqpoval}
(\snk{n+1}{n})\sb{\#}(\Val{\Stlp{n}})~=~0~,
\end{myeq}
\noindent in the notation of \S \ref{rcomp}. We denote strong equivalence
classes by \w[.]{[[\Stl{n}]]}
\end{defn}

\begin{remark}\label{requivstr}
Clearly \w{\Stl{n}\approx\Stlp{n}} implies that \w[,]{\Stl{n}\sim\Stlp{n}}
and both notions coincide if the comparison map \w{\Phi:\cW\to\,\ccWp} is
trivial \wh that is, if in the underlying algebraic comparison map $\Psi$,
\w{\Vd} and \w{\Vdp} have isomorphic CW bases.

Note also that \w{\Stl{n}\approx\rs(\Stl{n})} (and thus
\w[)]{\Stl{n}\sim\rs(\Stl{n})} for \emph{any} comparison map
\w[,]{\Phi:\cW\to\,\ccWp} since \w{\es\rs(\Stl{n})=\Stl{n}} and
\w{(\snk{n+1}{n})\sb{\#}(\Val{\rs(\Stl{n})})=
(\snk{n+1}{n})\sb{\#}(\oar{n})\sb{\ast}(\Val{\Stl{n}}=0}
by \wref{eqcorrvals} and \wref[.]{eqwkprod}
\end{remark}

\begin{lemma}\label{lmutvanish}
When \w[,]{\Stl{n}\approx\Stlp{n}} \w{\Val{\Stl{n}}=0} if and only if
\w[.]{\Val{\Stlp{n}}=0}
\end{lemma}

\begin{proof}
This follows from \wref[,]{eqcorrvals} since by \wref{eqpoval} we see that
\w{\Val{\Stlp{n}}\in\Gamma\lin{\Omega\sp{n}\oWp{n+1}}} is uniquely determined by
its image \w{\oen{n}\sb{\ast}(\Val{\Stlp{n}})} in
\w{\Gamma\lin{\Omega\sp{n}\oW{n+1}}} under the projection \w[.]{\oen{n}\sb{\ast}}
\end{proof}

\begin{defn}\label{dnoco}
Given a space $\bY$ with \w[,]{\Gamma:=\HiR{\bY}} we want to think of each
sequential realization $\cW$ for $\bY$ as a template for a countable sequence
\w{\llrra{\bY}=(\llrr{\bY}{n})\sb{n=2}\sp{\infty}} of higher operations, where
for each \w[,]{n\geq 2} we define the \emph{universal $n$-th order cohomology
operation} \w{\llrr{\bY}{n}} associated to the space $\bY$ to be the function
that assigns to every \wwb{n-1}strand \w{\Stl{n-1}} for \w{(\cW,\bZ,\vth)}
the class
$$
\llrr{\bY}{n}(\Stl{n-1})~:=~\Val{\Stl{n-1}}\in \Gamma\lin{\Omega\sp{n-1}\oW{n}}~.
$$

We say that the operation \w{\llrr{\bY}{n}} \emph{vanishes for $\bZ$} if
there is a cofibrant sequential realization $\cW$ and an \wwb{n-1} strand
\w{\Stl{n-1}} for \w{(\cW,\bZ,\vth)} such that \w[.]{\llrr{\bY}{n}(\Stl{n-1})=0}
By Lemma \ref{lmutvanish}, this notion of the vanishing depends only the strong
equivalence classes \w{[[\Stl{n-1}]]} of the strand.
\end{defn}

\begin{defn}\label{dcohervan}
Given spaces $\bY$ and $\bZ$ with \w{\vth:\HiR{\bY}\xra{\cong}\HiR{\bZ}} and a
sequential realization $\cW$ of $\bY$, we say
that \w{\llrra{\bY}} \emph{vanishes coherently for} \w{(\cW,\bZ,\vth)}
if there is an $\infty$-strand \w{\Stl{\infty}} for $\cW$ as in \S \ref{dshho}:
that is, for each \w[,]{n\geq 2} we have a given \wwb{n-1}strand \w{\Stl{n-1}}
for \w{(\cW,\bZ,\vth)} such that \w[,]{\Val{\Stl{n-1}}=0}
which extends to the next $n$-strand \w{\Stl{n}} using Lemma \ref{lvanish}.
\end{defn}

\begin{example}\label{eqcohervan}
For any sequential realization $\cW$ of a space
$\bY$, the sequence \w{\llrra{\bY}} vanishes coherently for \w{(\cW,\bY,\Id\sb{\Gamma})}
since then we have a given coaugmentation \w[,]{\bv:\bY\to\Wu}
which we can then project to each \w{\W{n}} (see Proposition \ref{pnstep})
to yield \w{\bve{n}} showing that \w{\Val{\Stl{n-1}}=0} for the corresponding
\wwb{n-1}strand \w[,]{\Stl{n-1}} by Lemma \ref{lvanish}.
\end{example}

\begin{remark}\label{rvanish}
Note that \emph{a priori}, each individual strand \w{\Stl{n-1}} has a different
template for \w{\llrra{\bY}} \wwh namely, the restricted cosimplicial set
\w{\tWu{n}} of the \wwb{n-1}st stage \w{\W{n-1}} of $\cW$.
However, the following result, which follows from Theorem \ref{tcomp}, shows that
we can in fact use any one cofibrant sequential realization to calculate
\w[:]{\llrr{\bY}{n}}
\end{remark}

\begin{keylemma}\label{lkey}
Given $\bY$ and \w{\vth:\HiR{\bY}\xra{\cong}\HiR{\bZ}} as above,
\w{\llrr{\bY}{n}} vanishes for $\bZ$ if and only if for \emph{every} $n$-stage
cofibrant sequential realization $\cW$ of $\bY$, there is an \wwb{n-1}strand
\w{\Stl{n-1}} for \w{(\cW,\bZ,\vth)} such that \w[.]{\Val{\Stl{n-1}}=0}
\end{keylemma}

\begin{proof}
By definition, \w{\llrr{\bY}{n}} vanishes for $\bZ$ if there is \emph{some}
cofibrant $n$-stage sequential realization \w{\cWp} of $\bY$
and an \wwb{n-1}strand \w{\Stl{n-1}'} for \w{(\cWp,\bZ,\vth)} such that
\w[.]{\Val{\Stl{n-1}'}=0} By Theorem \ref{tcomp} we know that there is a finite
zigzag of cospans of comparison maps connecting \w{\cWp} to $\cW$, say
$$
\Phip{1}:\cuW{0}=\cWp\to\cuW{1},\hsp \Phip{2}:\cuW{2}\to\cuW{1},\hsp
\Phip{3}:\cuW{2}\to\cuW{3}~,
$$
\noindent and so on until \w[.]{\Phip{N}:\cuW{N-1}\to\cuW{N}=\cW} If
\w{\Phip{1}=\lra{\en{k},\,\rn{k},\dotsc\,}\sb{k=0}\sp{n}} as in \wref[,]{eqncorresp}
we set \w{\Stl{n-1}\up{1}:=\rs(\Stl{n-1}')} (an \wwb{n-1}strand for \w[),]{\cuW{1}}
and see from \wref{eqcorrvals} that \w[.]{\Val{\Stl{n-1}\up{1}}=0}
Similarly, if \w{\Phip{2}=\lra{\ppp\en{k},\,\ppp\rn{k},\dotsc\,}\sb{k=0}\sp{n}}
we set \w{\Stl{n-1}\up{2}:=\ppp\es(\Stl{n-1}\up{1})} (an \wwb{n-1}strand for
\w[),]{\cuW{2}} and again see from \wref{eqcorrvals} that
\w[.]{\Val{\Stl{n-1}\up{2}}=0} Continuing in this way we finally obtain
an \wwb{n-1}strand \w{\Stl{n-1}=\Stl{n-1}\up{N}} for \w{\cuW{N}=\cW}
with \w[,]{\Val{\Stl{n-1}}=0} as required.

The converse follows from the fact that \w{\HiR{\bY}} has at least one
CW resolution by \S \ref{rcwres}, and thus there is at least one cofibrant sequential
realization for $\bY$ by Theorem \ref{tres}.
\end{proof}

Clearly, the \wwb{n-1}strands  \w{\Stl{n-1}'} and \w{\Stl{n-1}} are weakly
equivalent.  However, they are not necessarily strongly equivalent, since
there is no reason for \wref{eqpoval} to hold for the even-numbered comparison
maps above \w[,]{\Phip{2}} \w[,]{\Phip{4}} and so on.

\begin{thm}\label{tvanish}
For \w{R=\Fp} or a field of characteristic $0$, let $\bY$ and $\bZ$ be $R$-good
spaces with isomorphic \TRal[s.]  Then the following are equivalent:
\begin{enumerate}
\renewcommand{\labelenumi}{(\alph{enumi})~}
\item The system of higher cohomology operations \w{\llrra{\bY}} vanishes coherently
for \w{(\cW,\bZ,\vth)} for \emph{some} cofibrant sequential realization $\cW$
of $\bY$ and some $\vth$
\item \w{\llrra{\bY}} vanishes coherently for \emph{every} cofibrant sequential
realization of $\bY$
\item $\bY$ and $\bZ$ are $R$-equivalent.
\end{enumerate}
\end{thm}

\begin{proof}
(a)$\Leftrightarrow$(b) by  Key Lemma \ref{lkey}\vsn.

\noindent (a)$\Rightarrow$(c):\hsp Assume that \w{\llrra{\bY}} vanishes coherently
\wh that is,  there is an
$\infty$-strand \w{\Stl{\infty}} for some \w{(\cW,\bZ,\vth)} (where $\cW$ need not
be cofibrant), and thus coaugmentations \w{\bve{n}:\bZ\to\W{n}} for
all \w[.]{n\geq 0} These fit together to define a coaugmentation \w{\bv:\bZ\to\Wu}
for \w[,]{\Wu:=\holim\W{n}} which induces an isomorphism
\w[.]{\HiR{\Tot\Wu}\to\HiR{\bZ}}
Since $\bY$ is $R$-good, it is $R$-equivalent to the total space
\w{\Tot\Wu\simeq\LG\bY} (see \S \ref{sgcge}), and thus the map
\w{f:\bZ\to\Tot\Wu} induced by the coaugmentation $\bv$ realizes $\vth$, so
$\bY$ and $\bZ$ are related by a cospan of $R$-equivalences\vsn.

\noindent (c)$\Rightarrow$(a):\hsp Conversely, if $\bY$ and $\bZ$ are
$R$-equivalent, we have a
zigzag of $R$-equivalences from $\bY$ to $\bZ$ inducing an isomorphism of \TRal[s]
\w[,]{\vth:\HiR{\bY}\to\HiR{\bZ}} so it suffices to consider the
following two special cases:

\begin{enumerate}
\renewcommand{\labelenumi}{(\roman{enumi})~}
\item Given a $R$-equivalence \w[,]{f:\bZ\to\bY} and some (not necessarily
cofibrant) sequential realization $\cW$ for $\bY$, by precomposing the
coaugmentations \w{\bve{n}:\bY\to\W{n}} with $f$ we obtain coaugmentations
\w[,]{\bve{n}\circ f:\bZ\to\W{n}} still realizing
\w[,]{\Vd\to\Gamma} since \w{f\sp{\#}:\Gamma\to\HiR{\bZ}} is an
isomorphism. This yields an $\infty$-strand for \w{(\cW,\bZ,\vth)}
by Lemma \ref{lvanish}.
\item On the other hand, given a $R$-equivalence \w{g:\bY\to\bZ} and any
cofibrant sequential realization \w{\cWp} for $\bZ$,
by precomposing the coaugmentations \w{\bve{n}:\bZ\to\Wp{n}}
with $g$ as in (a) we obtain coaugmentations \w{\bve{n}\circ g:\bY\to\Wp{n}}
realizing \w[,]{\Vd\to\Gamma} and thus making \w{\cWp} itself with the
new coaugmentations into a cofibrant sequential realization \w{\cWpp}
for $\bY$. The coaugmentations \w{\bve{n}:\bZ\to\Wp{n}} form an
$\infty$-strand \w{\Stl{\infty}} for \w[,]{ \cWpp} showing that
\w{\llrra{\bY}} vanishes coherently for \w[.]{(\cWpp,\bZ,\vth)}
\end{enumerate}
\noindent This completes the proof.
\end{proof}

\begin{corollary}\label{csecoper}
If $\bY$ and \w{\bY'} are $R$-equivalent $R$-good spaces, any
cofibrant sequential realization \w{\cW=\lra{\W{n},\,\tWu{n}}\sb{n\in\NN}}
for $\bY$ is also a (cofibrant) sequential realization for \w[.]{\bY'}
\end{corollary}

\begin{proof}
The $R$-equivalence implies there is an isomorphism
\w[,]{\vth:\HiR{\bY}\cong\HiR{\bY'}} so by the Theorem there is an
$\infty$-strand \w{\Stl{\infty}} for \w[,]{(\cW,\bY',\vth)} and thus a
coaugmentation \w[.]{\bv':\bY'\to\Wu}
\end{proof}

\begin{mysubsection}{Low dimensional cases}
\label{rstrand}
As noted in Remark \wref[,]{rshho} given a simplicial set $\bZ$ equipped with an isomorphism of \TRal[s]
\w{\vth:\Gamma\to\HiR{\bZ}} and a sequential realization $\cW$
for $\Gamma$ as above,  we can always define a coaugmentation
\w{\bve{0}:\bZ\to\Wn{0}{0}} realizing \w[,]{\phi= \vth\circ\vare:V\sb{0}\to\HiR{\bZ}} which is unique up to homotopy,
as in the proof of Lemma \wref[.]{lvanish}
Moreover, if  \w{\udz{0}:\oW{0}\to\oW{1}} realizes the first attaching map
\w[,]{\odz{1}:\oV{1}\to V\sb{0}=\oV{0}} then
\w{\alk{-1}{0}:=\udz{0}\circ\bve{0}} is nullhomotopic, since it
realizes \w{\vare\circ\odz{1}} (see \wref[).]{eqwone}
Thus we can always choose a nullhomotopy \w{\Flk{-1}{0}} for \w[,]{\alk{-1}{0}}
and use it do define \w[,]{\bve{1}:\bZ\to\W{1}} as in the proof of the Lemma.
Note however that while the map \w{\bve{0}} is unique up to homotopy, the map
\w{\bve{1}} depends on our choice of  \w[.]{\Flk{-1}{0}}

This explains why our definition of $n$-th order cohomology operations only makes
sense for \w[.]{n\geq 2}

Thus the first case of interest is \w[.]{n=2} Since \w{\TsR} consists of
abelian group objects, we can replace our $2$-truncated
restricted cosimplicial diagram \w{\W{1}\to\oW{2}} by:
\begin{myeq}\label{eqtwotoda}
\bZ~\xra{\bve{1}}~\Wn{0}{1}~\xra{d\sp{0}-d\sp{1}}~\Wn{1}{1}~
\xra{\udz{1}}~\oW{2}
\end{myeq}
 \noindent where the fact that \w{d\sp{0}\circ\bve{1}=d\sp{1}\circ\bve{1}}
means that the first composite is nullhomotopic,  while the fact that
\w{\udz{1}\circ d\sp{0}} is nullhomotopic and \w{\udz{1}\circ d\sp{1}=0} means that
 the second composite is also nullhomotopic.

 In particular, the first map in \wref{eqtwotoda} represents (a collection of)
$R$-cohomology classes $\alpha$, while the remaining  two represent
$R$-cohomology operations $\xi$ and $\zeta$, with  \w{\xi(\alpha)=0} and
\w[.]{\zeta\circ\xi=0} Thus our universal secondary operation is just
a (collection of) secondary cohomology operations in the sense of Adams
(see \cite{AdamsN}), taking value in
\w[.]{[\bZ,\,\Omega\oW{2}]=\Hu{\ast-1}{\bZ}{R}}
\end{mysubsection}

\begin{remark}\label{rhighercoh}
Our higher cohomology operations are modeled on Adams' (stable) secondary
cohomology operations (see also \cite{HarpSC}). However, the more delicate
questions involving the prerequisites for an \wwb{n+1}st order operation to be
defined, and the dependence on various choices made, are hidden here in the two
components of the $n$-strand \w[,]{\Stl{n}} consisting of:
\begin{enumerate}
\renewcommand{\labelenumi}{(\alph{enumi})}
\item The $n$-th stage \w{\W{n}} in the sequential approximation $\cW$
encodes a preliminary choice of nullhomotopies for that part of the diagram
consisting only of spaces in \w{\TsR} (the representing spaces for cohomology).
\item The data associated to the specific simplicial set $\bZ$ consists of the
coaugmentation \w[,]{\bve{n}:\bZ\to\W{n}} which itself is determined by:
\begin{enumerate}
\renewcommand{\labelenumii}{\roman{enumii}.~}
\item  The coherent system of earlier choices made, encoded in the \wwb{n-1}strand
\w{\Stl{n-1}=(\bve{0},\bve{1},\dotsc,\bve{n-1})} (essentially, the single map
\w[);]{\bve{n-1}}
\item The  single choice of the nullhomotopy \w{\Flk{-1}{n}} for
\w[,]{\alk{-1}{n}} (i.e.,  the value of the previously defined $n$-th order
operation, which must necessarily vanish in order to proceed to the \wwb{n+1}st
step).
\end{enumerate}
\end{enumerate}
\end{remark}

\begin{mysubsection}{Models for rational homotopy theory}
\label{smrht}
When working with \w{R=\QQ} it is convenient to use some of the known
models for rational homotopy theory (see \cite{QuiR}). In particular,
finite type simply connected rational spaces \w{\bY\in\SQ} can be
modelled in the category \w{\DGA} of differential graded commutative
$\QQ$-algebras (CDGAs), using a suitable Sullivan model \w{(A\sp{\ast},d)\in\DGA}
for $\bY$ (see \cite[\S 12]{FHThR}).

The equivalence of homotopy categories \w{\ho\SQ\to\ho\DGA}
is contravariant and takes products to coproducts and path or loop spaces to
cone or suspension objects.  Thus if we try to apply Theorem \ref{tres} to
the model \w{(A\sp{\ast},d)\in\DGA} directly, rather than to \w[,]{\bY\in\SQ}
we will end up with a \emph{simplicial} CDGA \w[,]{\Wd} obtained as the
homotopy colimit of sequential simplicial realization (see \cite{BJTurnHI} for
full details of the simplicial version). We could in fact replace this
simplicial CDGA by a bi-graded CDGA (see \cite{FelM}, and compare \cite{BlaHR}).

In particular, we have an adjunction
\w{\Lambda:\Ch{\ast}\sb{\QQ}\rightleftharpoons\DGA:U} between cochain complexes
and CDGAs, where $U$ is the forgetful functor and \w{\Lambda(V\sp{\ast},d)}
the free graded commutative algebra on the graded vector space \w[,]{V\sp{\ast}}
with \w{d\sp{i}:V\sp{i}\to V\sp{i+1}} extended as a derivation.
Note that each \w{\Lambda(V\sp{\ast},d)} is a Sullivan algebra, and thus a
cofibrant CDGA (see \cite[\S 1]{HessR}).

The functor $\Lambda$ yields formal minimal models for each $\QQ$-GEM, as well
as their cylinders, cones, and suspensions. For example, if \w{V\sp{\ast}} is a
graded $\QQ$-vector space which is degree-wise finite-dimensional, then
\w{\Lambda(V\sp{\ast},0)} is a minimal model for
\w[.]{\prod\sb{i=0}\sp{\infty}\,\KP{V\sp{i}}{i}} Similarly, if
\w{A\sp{\ast}:=\Lambda(V\sp{\ast},d)} is a Sullivan model for some space $\bY$,
and \w{i:(V\sp{\ast},d)\hra C(V\sp{\ast},d)} is the inclusion into the cone
(see \cite[\S 1.5]{WeibHA}), then
\w[,]{i\sb{\ast}:\Lambda(V\sp{\ast},d)\to\Lambda(C(V\sp{\ast},d))} the
corresponding cone inclusion in \w[,]{\DGA} is a CDGA model for the path
fibration \w{p:P\bY\to\bY} (see \cite[\S 14]{FHThR}).
\end{mysubsection}

\begin{mysubsect}{A rational example}
\label{sratex}

Even though the above discussion was stated in terms of the sketch \w{\TR}
in \w{\ho\C} (for \w[),]{\C=\Sa} when \w[,]{R=\QQ} as we just pointed out, we can
also apply it (\emph{mutatis mutandis}) to the corresponding CDGA models.

For example, let \w{\bY} be the simply-connected $\QQ$-local finite type space
represented by the free CDGA \w{(A\sp{\ast},d)} with
\begin{enumerate}
\renewcommand{\labelenumi}{\roman{enumi}.~}
\item \w{A\sp{n}=\QQ\lra{x,y,z}} with \w[.]{dx=dy=dz=0}
\item \w{A\sp{2n-1}=\QQ\lra{u,v}} with \w{du=xy} and \w[.]{dv=xz}
\item \w{A\sp{3n-2}=\QQ\lra{q,r,s,t}} with \w[,]{dq=xu} \w[,]{dr=xv} \w[,]{ds=yu}
and \w[.]{dt=zv}
\item For \w{i>3n} \w{A\sp{i}} is then chosen so that \w{H\sp{i}(A\sp{\ast})=0} for
\w[.]{i\geq 3n}
\end{enumerate}
Here \w{n>1} is odd.

Thus \w{(A\sp{\ast},d)} has rational cohomology \w{\Gamma=\HiQ{\bY}} (as a \TQal
\wh that is, a graded $\QQ$-algebra) with
\begin{enumerate}
\renewcommand{\labelenumi}{\roman{enumi}.~}
\item \w[.]{\Gamma\sp{n}=\QQ\lra{[x],[y],[z]}}
\item \w[.]{\Gamma\sp{2n}=\QQ\lra{[y]\cdot[z]}}
\item \w{\Gamma\sp{3n-1}=\QQ\lra{\omega}} where $\omega$ is represented in
\w{A\sp{\ast}} by \w[.]{zu+yv}
\item \w{\Gamma\sp{i}=0} for \w[.]{i\neq n,2n,3n-1}
\end{enumerate}

Note that the (formal) rational space
\w{\bZ:=(\bS{n}\vee(\bS{n}\times\bS{n})\vee\bS{3n-1})\sb{\QQ}}
also has \w{\HiQ{\bZ}\cong\Gamma} (as \TQal[s).]
It is represented by the Sullivan model \w{(B\sp{\ast},d)} with
\begin{enumerate}
\renewcommand{\labelenumi}{\roman{enumi}.~}
\item \w{B\sp{n}=\QQ\lra{x,y,z}} with \w[.]{dx=dy=dz=0}
\item \w{B\sp{2n-1}=\QQ\lra{u,v}} with \w{du=xy} and \w[.]{dv=xz}
\item \w{B\sp{3n-2}=\QQ\lra{p,q,r,s,t}} with \w[,]{dp=zu+yv} \w[,]{dq=xu}
\w[,]{dr=xv} \w[,]{ds=yu} and \w[.]{dt=zv}
\item \w{B\sp{3n-1}=\QQ\lra{w}} with \w[,]{dw=0} where $\omega$ is
represented in \w{B\sp{\ast}} by $w$.
\item Again, for \w{i>3n} \w{B\sp{i}} is then chosen so that
\w{H\sp{i}(B\sp{\ast})=0} for \w[.]{i\geq 3n}
\end{enumerate}

Let us denote by
\w{\FQ{x\sb{n\sb{1}},\dotsc,x\sb{n\sb{k}}}} the free
\TQal generated by elements \w{x\sb{n\sb{i}}} in degree \w{n\sb{i}}
\wb{i=1,\dotsc,k} \wwh so that
\w{\FQ{x\sb{n\sb{1}},\dotsc,x\sb{n\sb{k}}}~\cong~
\Hu{\ast}{\prod\sb{i=1}\sp{k}~\KQ{n\sb{i}}}{\QQ}} (see \S \ref{smrht})\vsm .

We may choose a (minimal) CW resolution of \TQal[s] \w{\Vd\to\Gamma} with
CW basis \w{(\oV{n})\sb{n\in\NN}} as follows:
\begin{enumerate}
\renewcommand{\labelenumi}{(\alph{enumi})~}
\item \w[,]{\oV{0}=\FQ{x\sb{n},y\sb{n},z\sb{n},w\sb{3n-1}}} with the
obvious augmentation \w[.]{\vare:\oV{0}\to\Gamma}
\item \w[,]{\oV{1}=\FQ{u\sb{2n},v\sb{2n}}\amalg\oU{1}} where \w{\oU{1}} is a
free \TQal with generators in degrees $>3n$. \ The attaching map
\w{\odz{}:\oV{1}\to V\sb{0}=\oV{0}} is defined by
\w{u\sb{2n}\mapsto x\sb{n}y\sb{n}} and \w[.]{v\sb{2n}\mapsto x\sb{n}z\sb{n}}
\item \w[,]{\oV{2}=\FQ{p\sb{3n},q\sb{3n},r\sb{3n},s\sb{3n},t\sb{3n}}\amalg\oU{2}}
where \w{\oU{2}} is again a free \TQal with generators in degrees $>3n$. \

The attaching map
\w{\odz{}:\oV{2}\to V\sb{1}=\oV{1}\amalg s\sb{0}\oV{0}} is defined by
\begin{myeq}\label{eqattachone}
\begin{split}
p\sb{3n}~\mapsto&~
(s\sb{0}z\sb{n}) u\sb{2n}+(s\sb{0}y\sb{n})v\sb{2n},\hsp
q\sb{3n}~\mapsto~(s\sb{0}x\sb{n})u\sb{2n},\\
r\sb{3n}~\mapsto&~(s\sb{0}x\sb{n})v\sb{2n},\hsm
s\sb{3n}~~\mapsto~(s\sb{0}y\sb{n})u\sb{2n},\hsm
t\sb{3n}~~\mapsto~(s\sb{0}z\sb{n})v\sb{2n}~.
\end{split}
\end{myeq}
\item For \w{k\geq 3} the basis \TQal \w{\oV{k}} has generators
in degrees $>3n$.
\end{enumerate}

Denote by \w{\KsQ{\bx\sb{n}}=\Lambda(M\sp{\ast},0)} the formal free CDGA model
for \w{\KQ{n}} (where \w{M\sp{\ast}} is the graded vector space
concentrated in degree $n$ with basis \w{\{\bx\sb{n}\}} (see \S \ref{smrht}).

We can realize \w{\Vd\to\Gamma} through (co)simplicial dimension $2$ and
degree \w{3n} by an augmented simplicial CDGA \w{\Wl{2}\to B\sp{\ast}} with CW
basis \w{(\oWl{n})\sb{n\in\NN}} constructed as follows\vsm:

\noindent\textbf{Step A.} \
First, we construct the simplicial analogue of \w{\Wl{1}} through simplicial
dimension $1$:

\begin{enumerate}
\renewcommand{\labelenumi}{(\alph{enumi})~}
%
\item We let
$$
\oWl{0}~=~
\KsQ{\bx\sb{n},\by\sb{n},\bz\sb{n},\bw\sb{3n-1}}~:=~
\KsQ{\bx\sb{n}}~\amalg~\KsQ{\by\sb{n}}~
\amalg~\KsQ{\bz\sb{n}}~\amalg~\KsQ{\bw\sb{3n-1}}
$$
\noindent with augmentation \w{\bv:\oWl{0}\to A\sp{\ast}} defined by
$$
\bx\sb{n}\mapsto x,\hsp
\by\sb{n}\mapsto y,\hsp\bz\sb{n}\mapsto z,\hsp
\bw\sb{3n-1}\mapsto zu+yv~.
$$
%
\item We let
$$
\tWln{1}{1}~=~\oWl{1}~:=~\KsQ{\bu\sb{2n},\bfv\sb{2n}}~,
$$
\noindent where the untruncated version of \w{\oWl{1}} has an additional free
CDGA coproduct summand \w{D\sb{2}} with generators in degrees $>3n$. \
The attaching map \w{\odz{}:\oWl{1}\to\Wln{0}{2}} is defined
\w[,]{\bu\sb{2n}\mapsto\bx\sb{n}\by\sb{n}}
\w[,]{\bfv\sb{2n}\mapsto\bx\sb{n}\bz\sb{n}} and all other generators sent to $0$.
%
\item Dually to Step \textbf{III} in the proof of Theorem \ref{tres}, we must
add a coproduct summand \w{C\oWl{1}} to obtain
\w[,]{\tWln{0}{1}:=\oWl{0}\amalg C\oWl{1}}
where the cone on the formal CDGA
\w[,]{\oWl{1}=\Lambda[\bu\sb{2n},\bfv\sb{2n}]} which models the
path space \w[,]{P\oW{1}} is the CDGA
\w{C\oWl{1}=\Lambda(\QQ(\bu'\sb{2n},\bfv'\sb{2n},\obu\sb{2n-1},\obv\sb{2n-1}),d)}
with differential \w{d(\obu\sb{2n-1})=-\bu'\sb{2n}} and
\w{d(\obv\sb{2n-1})=-\bfv'\sb{2n}} (see \S \ref{smrht}). We will denote this
simply by \w[.]{\Lambda(\bu'\sb{2n},\bfv'\sb{2n},\obu\sb{2n-1},\obv\sb{2n-1})}

The augmentation \w{\bv:C\oWl{1}\to A\sp{\ast}} is given by
$$
\bu'\sb{2n}\mapsto xy,\hsm \bfv'\sb{2n}\mapsto xz,\hsm
\obu\sb{2n-1}\mapsto u,\hsm \obv\sb{2n-1}\mapsto v~.
$$
%
\item Finally, we must add the degeneracies to \w{\tWln{\bullet}{1}} (as in
Step \textbf{iii} in \S \ref{senso}) \wh that is, we add
two coproduct summands
\w{s\sb{0}\oWl{0}=
\KsQ{s\sb{0}\bx\sb{n},s\sb{0}\by\sb{n},s\sb{0}\bz\sb{n},s\sb{0}\bw\sb{3n-1}}} and
\w{s\sb{0}C\oWl{1}=\Lambda(s\sb{0}\bu'\sb{2n},s\sb{0}\bfv'\sb{2n},
s\sb{0}\obu\sb{2n-1},s\sb{0}\obv\sb{2n-1})} \wwh to obtain the $1$-truncation
of the augmented simplicial CDGA \w{\Wl{1}} (dual to \wref[)]{eqwone} in
degrees $\leq 3n$, depicted in Figure \ref{fig2}.
%
%
%
\begin{figure}[htbp]
\begin{center}
\xymatrix@C=10pt{
\Wln{1}{1}~~= \ar@/_{0.2pc}/[d]\sb{d\sb{0}} \ar@/^{0.2pc}/[d]\sp{d\sb{1}} &
s\sb{0}\oWl{0} \ar@/_{0.2pc}/[d]\sb{\Id}\ar@/^{0.2pc}/[d]\sp{\Id} &
\amalg & \oWl{1}=\Lambda[\bu,\bfv]
\ar[dll]_(0.6){\odz{}} \ar[d]\sp{d\sb{1}=\iota} & \amalg &
s\sb{0}C\oWl{1} \ar@/_{0.2pc}/[dll]\sb{\Id}\ar@/^{0.2pc}/[dll]\sp{\Id}\\
\Wln{0}{1}~~=\ar[d]\sb{\bv} & \oWl{0}=\Lambda[\bx,\by,\bz,\bw] &\amalg &
C\oWl{1}= \Lambda(\bu',\bfv',\obu,\obv)\\
A\sp{\ast}~~= & \Lambda[x,y,z,u,v,q,r,s,t]
}
\end{center}
\caption[fig2]{$\Wl{1}$ \ in degrees $\leq 3n$}
\label{fig2}
\end{figure}
\end{enumerate}

\noindent\textbf{Step B.} \ \
To extend \w{\Wl{1}} to \w[,]{\Wl{2}} we proceed as follows:

\begin{enumerate}
\renewcommand{\labelenumi}{(\alph{enumi})~}
%
\item First, we set
$$
\tWln{2}{2}~=~\oWl{2}~:=~
\KsQ{\bfp\sb{3n},\bq\sb{3n},\br\sb{3n},\bs\sb{3n},\bt\sb{3n}}~,
$$
\noindent (where again we omit generators in degrees $>3n$). \

As a first approximation, we would like to use \wref{eqattachone} to determine
\w[.]{\odz{\oWl{2}}:\oWl{2}\to\Wln{1}{2}} However, although
this will guarantee that \w[,]{d\sb{0}\circ\odz{\oWl{2}}=0} we would then not
have \w{d\sb{1}\circ\odz{\oWl{2}}=0} (see Step \textbf{VII} in the proof of
Theorem \ref{tres}).

We therefore define \w{\odz{\oWl{2}}} by
\begin{myeq}\label{eqattachonec}
\begin{split}
\bfp\sb{3n}~\mapsto&~
(s\sb{0}\bz\sb{n})\bu\sb{2n}+(s\sb{0}\by\sb{n})\bfv\sb{2n}
-s\sb{0}(\bz\sb{n}\bu'\sb{2n})-s\sb{0}(\by\sb{n}\bfv'\sb{2n}),\\
\bq\sb{3n}~\mapsto&~(s\sb{0}\bx\sb{n})\bu\sb{2n}-s\sb{0}(\bx\sb{n}\bu'\sb{2n}),\hsm
\br\sb{3n}~\mapsto~(s\sb{0}\bx\sb{n})\bfv\sb{2n}-s\sb{0}(\bx\sb{n}\bfv'\sb{2n}),\\
\bs\sb{3n}~\mapsto&~(s\sb{0}\by\sb{n})\bu\sb{2n}-s\sb{0}(\by\sb{n}\bu'\sb{2n}),\hsm
\bt\sb{3n}~\mapsto~(s\sb{0}\bz\sb{n})\bfv\sb{2n}-s\sb{0}(\bz\sb{n}\bfv'\sb{2n})~.
\end{split}
\end{myeq}
\noindent with \w{d\sb{1}:\oWl{2}\to\tWln{1}{2}} the inclusion into the new cone
summand
$$
C\oWl{2}~=~\Lambda(\bfp'\sb{3n},\bq'\sb{3n},\br'\sb{3n},\bs'\sb{3n},\bt'\sb{3n},
\obp\sb{3n-1},\obq\sb{3n-1},\obr\sb{3n-1},\obs\sb{3n-1},\obt\sb{3n-1})
$$
\noindent in \w{\tWln{1}{2}=\Wln{1}{1}\amalg C\oWl{2}} (with
\w[,]{d(\obp\sb{3n-1})=\bfp'\sb{3n}} and so on).

The map \w{d\sb{0}=F\sb{1}:C\oWl{2}\to\Wln{0}{1}} is given by
\begin{myeq}\label{eqattachonef}
\begin{split}
\bfp'\sb{3n}~\mapsto&~-\bz\sb{n}\bu'\sb{2n}-\by\sb{n}\bfv'\sb{2n},\hsp
\bq'\sb{3n}~\mapsto~-\bx\sb{n}\bu'\sb{2n},\\
\br'\sb{3n}~\mapsto&~-\bx\sb{n}\bfv'\sb{2n},\hsm
\bs'\sb{3n}~\mapsto~-\by\sb{n}\bu'\sb{2n},\hsm
\bt'\sb{3n}~\mapsto~-\bz\sb{n}\bfv'\sb{2n}\\
\obp\sb{3n-1}~\mapsto&~
\bz\sb{n}\obu\sb{2n-1}+\by\sb{n}\obv\sb{2n-1}-\bw\sb{3n-1},\hsp
\obq\sb{3n-1}~\mapsto~\bx\sb{n}\obu\sb{2n-1},\\
\obr\sb{3n-1}~\mapsto&~\bx\sb{n}\obv\sb{2n-1},\hsm
\obs\sb{3n-1}~\mapsto~\by\sb{n}\obu\sb{2n-1},\hsm
\obt\sb{3n-1}~\mapsto~\bz\sb{n}\obv\sb{2n-1}~.
\end{split}
\end{myeq}
\noindent as in \wref[.]{eqattachone}
%
\item In dimension $0$ we have \w[,]{\tWln{0}{2}=\Wln{0}{1}\amalg C\Sigma\oWl{2}}
where
$$
C\Sigma\oWl{2}~=~
\Lambda(\obp'\sb{3n-1},\obq'\sb{3n-1},\obr'\sb{3n-1},\obs'\sb{3n-1},\obt'\sb{3n-1},
\oobp\sb{3n-2},\oobq\sb{3n-2},\oobr\sb{3n-2},\oobs\sb{3n-2},\oobt\sb{3n-2})
$$
\noindent (with \w[,]{d(\oobp\sb{3n-2})=\obp'\sb{3n-1}} and so on).

The face map \w{d\sb{1}:\tWln{1}{2}\to\tWln{0}{2}} is defined on the new summand
\w{C\oWl{2}} to be the quotient \w{C\oWl{2}\epic\Sigma\oWl{2}} followed by
the inclusion \w[,]{\Sigma\oWl{2}\hra C\Sigma\oWl{2}} which is given by
\w[,]{\bfp\sb{3n}\mapsto 0} \w[,]{\obp\sb{3n-1}\mapsto-\obp'\sb{3n-1}} and so on.

The augmentation \w{\bv:C\Sigma\oWl{2}\to A\sp{\ast}} is given by
\begin{myeq}\label{eqattachoneg}
\begin{split}
\obp'\sb{3n-1}~\mapsto&~0,\hsm \obq'\sb{3n-1}~\mapsto~x\sb{n}u\sb{2n-1},\hsm
\obr'\sb{3n-1}~\mapsto~x\sb{n}v\sb{2n-1},\\
\obs'\sb{3n-1}~\mapsto&~y\sb{n}u\sb{2n-1},\hsm
\obt'\sb{3n-1}~\mapsto~z\sb{n}v\sb{2n-1},\\
\oobp\sb{3n-2}~\mapsto&~0,\hsm
\oobq\sb{3n-2}~\mapsto~-q\sb{3n-2},\hsm
\oobr\sb{3n-2}~\mapsto~-r\sb{3n-2},\\
\oobs\sb{3n-2}~\mapsto&~-s\sb{3n-2},\hsm
\oobt\sb{3n-2}~\mapsto~-t\sb{3n-2}~.
\end{split}
\end{myeq}
\end{enumerate}

The $2$-truncation of \w{\tWln{\bullet}{2}} in degrees $\leq 3n$ \ is depicted in
Figure \ref{fig3}.
%
%
%
\begin{figure}[htbp]
\begin{center}
\xymatrix@C=5pt{
\oWl{2}=\Lambda[\bfp,\bq,\br,\bs,\bt]
\ar[d]\sb{\odz{}} \ar[drr]\sp{\odz{}} \ar@/^{0.8pc}/[drrrr]\sp{\odz{}}
\ar@/^{2.0pc}/[drrrrrr]\sp{d\sb{1}}\\
s\sb{0}\oWl{0} \ar@/_{0.2pc}/[d]\sb{=}\ar@/^{0.2pc}/[d]\sp{=} &
\amalg & \oWl{1}=\Lambda[\bu,\bfv]
\ar[dll]_>>>>>>>>>>{\odz{}} \ar[d]\sp{d\sb{1}} & \amalg &
s\sb{0}C\oWl{1} \ar@/_{0.2pc}/[dll]\sb{=}\ar@/^{0.2pc}/[dll]^>>>>>>>>>>>>{=} &
\amalg & C\oWl{2}\ar@/_{0.2pc}/[dllll]\sp{d\sb{0}}\ar@/^{0.2pc}/[dll]\sp{d\sb{1}} \\
\oWl{0}=\Lambda[\bx,\by,\bz,\bw]\ar[d]^{\bv} &\amalg &
C\oWl{1}= \Lambda(\bu',\bfv',\obu,\obv) \ar[lld]_{\bv} & \amalg & C\Sigma\oWl{2}\\
A\sp{\ast}= \Lambda[x,y,z,u,v,q,r,s,t]
}
\end{center}
\caption[fig3]{$\Wl{2}$ \ in degrees $\leq 3n$}
\label{fig3}
\end{figure}

\noindent\textbf{Step C.} \ \
When we try to map \w{\Wl{2}} to the CDGA \w[,]{B\sp{\ast}} we must modify the
augmentation \w{\bv:\tWln{0}{2}\to A\sp{\ast}} as follows:

In order to realize \w{\vare:V\sb{0}\to\Gamma} we must have
\w{\bv(\bw\sb{3n-1})=w} and otherwise $\bv$ is the same as in Step \textbf{A}.
Therefore,
$$
\bv(\obp'\sb{3n-1})~=~\bv(d\sb{1}(\obp\sb{3n-1}))~=~
\bv(\bz\sb{n}\obu\sb{2n-1}+\by\sb{n}\obv\sb{2n-1}-\bw\sb{3n-1})~=~
zu+yv-w~,
$$
\noindent by \wref[.]{eqattachonef} Therefore, we must map \w{\oobp\sb{3n-2}} to
an element $a$ in \w{B\sp{3n-2}} with \w[.]{d(a)=zu+yv-w} Since \w{d(p)=zu+yv} but
$w$ represents a non-zero element in \w[,]{\Gamma=\Hu{3n-1}{\bZ}{\QQ}} no such $a$
exists. Thus the two rational spaces $\bY$ and $\bZ$ are
not weakly equivalent by Corollary \ref{csecoper}.

Intuitively, the element \w{\omega\in\Gamma\sp{3n-1}} is represented
by the Massey product \w{[zu+yv]=\lra{[y],[x],[z]}} in \w[,]{\Hu{3n-1}{\bY}{\QQ}}
while $\omega$ is not a Massey product in \w{\Hu{3n-1}{\bZ}{\QQ}} since $\bZ$ is
formal. In our language these facts are represented by the nullhomotopies
\w{\obp\sb{3n-1}\mapsto\bz\sb{n}\obu\sb{2n-1}+\by\sb{n}\obv\sb{2n-1}-\bw\sb{3n-1}}
in \w[,]{\bY\to\Wu} where in the analogous construction for $\bZ$ we would have had
\w[.]{\obp\sb{3n-1}\mapsto\bz\sb{n}\obu\sb{2n-1}+\by\sb{n}\obv\sb{2n-1}}
\end{mysubsect}

%
%
\sect{Higher cohomology invariants for maps}
\label{chhim}

The system of higher cohomology operations associated to a \TRal
\w{\HiR{\bY}} described in the previous section may be thought
of as a sequence of obstructions to realizing an algebraic isomorphism
\w{\vth:\HiR{\bY}\xra{\cong}\HiR{\bZ}} by a map \w{f:\bZ\to\bY} (necessarily an
$R$-equivalence) \wh as well as constituting a complete collection
of higher invariants for the  weak $R$-homotopy type of spaces.

In this section we address the analogous problem for arbitrary maps: given two
maps \w{\fu{0},\fu{1}:\bZ\to\bY} which induce the same morphism of \TRal[s]
\w[,]{\psi:\HiR{\bY}\to\HiR{\bZ}} we define a sequence of higher cohomology
operations which vanish coherently if and only if \w{\fu{0}} and \w{\fu{1}} are
$R$-equivalent.

\begin{mysubsection}{Obstructions for lifting homotopies}
\label{solhot}
We start with the \emph{initial data} \w[,]{(\cW,\fu{0},\fu{1})} consisting of
a sequential realization $\cW$ for $\bY$, and two maps \w{\fu{0}} and \w{\fu{1}} as above.
Since \w{\fu{0}} and \w{\fu{1}} induce the same map of \TRal[s,]  there is
a homotopy \w{\Huk{0}{0}:\bZ\otimes I\to\oW{0}=\Wn{0}{0}} between
\w{\bve{0}\circ\fu{0}} and \w[,]{\bve{0}\circ\fu{1}} so
\begin{myeq}\label{eqhukone}
\bve{0}\circ(\fu{0}\bot\fu{1})~=~\Huk{0}{0}\circ(i\sb{0}\bot i\sb{1})
\end{myeq}
\noindent We call \w{\Htl{0}=(\Huk{0}{0})} a $0$-\emph{strand} for
\w[.]{(\cW,\fu{0},\fu{1})}

Recall that the standard cosimplicial space \w{\Du} is given by
the diagram of $n$-simplices with the standard maps between them,
%
%
\noindent where \w{\eta\sp{i}:\Deln{k-1}\hra\Deln{k}} is the inclusion of the $i$-th
face, and \w{\sigma\sp{j}:\Deln{k}\epic\Deln{k-1}} is the $j$-th collapse map
(see \cite[X, 2.2]{BKanH}). Applying the simplicial structure operation
\w{-\otimes\Du} to a fixed space \w{\bZ}
yields a cosimplicial space \w[.]{\bZ\otimes\Du}
\end{mysubsection}

\begin{defn}\label{dstrandmap}
An $n$-\emph{strand}
\w{\Htl{n}=(\Huk{}{0},\dotsc,\Huk{}{m},\dotsc,\Huk{}{n})}
for \w{(\cW,\fu{0},\fu{1})} is a compatible sequence of maps of cosimplicial
spaces \w{\Huk{}{m}:\bZ\otimes I\otimes\Du\to\W{m}} \wb[,]{m=0,\dotsc,n}
each determined by the collection of $k$-homotopies
\w{\Huk{k}{m}:\bZ\otimes(I\times\Del[k])\to\Wn{k}{m}} for \w{0\leq k\leq m}
such that all the downward and upward squares in the following diagram
\mydiagram[\label{eqhukn}]{
\bZ\amalg\bZ \ar[rrr]\sp{\fu{0}\bot\fu{1}} \ar[d]\sp{i\sb{0}\bot i\sb{1}}
&&& \bY \ar[d]\sp{\bve{m}} \\
\bZ\otimes I   ~\ar[rrr]\sp{\Huk{0}{m}}\ar@{..}[d] &&& \Wn{0}{m}\ar@{..}[d] \\
\bZ\otimes (I\times\Deln{k-1})~
\ar[d]\sb{\eta\sp{i}\sb{\ast}}\ar[rrr]\sp{\Huk{k-1}{m}} &&&
\Wn{k-1}{m} \ar[d]\sb{d\sp{i}}\\
\bZ\otimes (I\times\Deln{k})~\ar@/_{1.5pc}/[u]\sb{\sigma\sp{j}\sb{\ast}}
\ar[rrr]\sp{\Huk{k}{m}} &&& \Wn{k}{m} \ar@/_{1.5pc}/[u]\sb{s\sp{j}}
}
\noindent commute for all choices of \w{0\leq i\leq k\leq m} and
\w[.]{0\leq j\leq k-1}

More precisely, the choices of \w{\Huk{k}{m}} for \w{0\leq k\leq m}
uniquely determine a map of cosimplicial spaces
\w[,]{\vHn{m}:\bZ\otimes I\otimes\Du\to\vWu{m}} since the target is
$m$-coskeletal. We then use the left lifting property for the solid
commuting square of cosimplicial spaces:
\mydiagram[\label{eqllpcs}]{
\ast \ar[rr] \ar@{^{(}->}[d] && \W{m} \ar@{->>}[d]\sb{\simeq}\sp{h\bp{m}} \\
\bZ\otimes I\otimes\Du \ar[rr]\sb{\vHn{m}} \ar@{.>}[rru]\sp{\Huk{}{m}} && \vWu{m}
}
\noindent to obtain the required map \w[,]{\Huk{}{m}} unique up to weak equivalence,
using the fact that \w{\bZ\otimes I\otimes\Du} is cofibrant and
the map \w{h\bp{m}:\W{m}\to\vWu{m}} is a trivial fibration by Remark \ref{rcoffib}.

We say that an $n$-strand \w{\Htl{n}} \emph{extends} a given
\wwb{n-1}strand \w{\Htl{n-1}} if
\begin{myeq}\label{eqexthstrand}
\Huk{}{n-1}~=~\prn{n}\circ\Huk{}{n}~:~\bZ\otimes(I\times\Du)\to\W{n-1}
\end{myeq}
\noindent (see \wref[).]{eqtower}

An $\infty$-\emph{strand} is a sequence
\w{\Htl{\infty}:=(\Huk{}{n})\sb{n=1}\sp{\infty}} satisfying \wref{eqexthstrand}
for each \w[.]{n>0}
\end{defn}

\begin{remark}\label{rfoldpoly}
In order to extend a given \wwb{n-1}strand \w{\Htl{n-1}} for
a sequential realization $\cW$ of $\bY$ to an $n$-strand, we need to
produce maps \w{\hHuk{k}{n}:\bZ\otimes(I\times\Deln{k})\to\PoW{n-k-1}{n}}
for \w{0\leq k\leq n} satisfying
\begin{myeq}\label{eqnewfaces}
\begin{cases}
~\hHuk{k}{n}\circ \eta\sp{0}\sb{\ast}~=&~
\Fk{k-1}\circ v\sp{k-1}\circ\Huk{k-1}{n-1}~,\\
~\hHuk{k}{n}\circ \eta\sp{1}\sb{\ast}~=&~\dif{k-2}\circ\hHuk{k-1}{n}~,\\
~\hHuk{k}{n}\circ \eta\sp{i}\sb{\ast}~=&~0\hsm\text{for}\hsm i\geq 2~,
\end{cases}
\end{myeq}
\noindent where \w{\dif{j}=\dif{j}\sb{\bD}} is the differential of \w{\bDs} as in
\wref[.]{eqtau}

Thus, in the following diagram we are given the solid \wwb{n-1}strand for
\w[,]{\W{n-1}} which we  wish to extend by the dashed maps to the (given)
restricted cosimplicial space \w[:]{\tWu{n}}
\mydiagram[\label{eqexttothree}]{
\bZ\amalg\bZ \ar[d]\sp{(i\sb{0}\bot i\sb{1})} \ar[rrrr]\sp{(\fu{0}\bot\fu{1})}
 &&&& \bY \ar[ld]\sb{\Fk{-1}} \ar[d]\sb{\bv} \\
\bZ\otimes(I\times\Deln{0}) \ar@{-->}[rrr]\sp(0.6){\hHuk{0}{n}}
 \ar@/^{2.0pc}/[rrrr]\sp{\Huk{0}{n-1}}
 \ar@<-0.5ex>[d]\sb{\eta\sp{0}\sb{\ast}} \ar@<0.5ex>[d]\sp{\eta\sp{1}\sb{\ast}}
&&&
\PoW{n-1}{n}\hspace*{4mm}\times \ar[d]\sp(0.3){\dif{-1}}
& \Wn{0}{n-1}   \ar[ld]\sb{\Fk{0}\circ v\sp{0}}
\ar@<-0.5ex>[d]\sb{d\sp{0}} \ar@<0.5ex>[d]\sp{d\sp{1}}&
\hspace*{-8mm}= \tWn{0}{n} \\
\bZ\otimes(I\times\Deln{1}) \ar@{-->}[rrr]\sp(0.65){\hHuk{1}{n}}
 \ar@/^{2.0pc}/[rrrr]\sp{\Huk{1}{n-1}}
 \ar@<-2.5ex>[d]\sb{\eta\sp{0}\sb{\ast}}
\sp{\eta\sp{1}\sb{\ast}} \ar[d] \ar@<1ex>[d]\sp{\eta\sp{2}\sb{\ast}}
&&&
\PoW{n-2}{n}\hspace*{4mm}\times \ar[d]\sp(0.3){\dif{0}}
& \Wn{1}{n-1}  \ar[ld]\sb{\Fk{1}\circ v\sp{1}}
\ar@<-2.5ex>[d]\sb{d\sp{0}}\sp{d\sp{1}} \ar[d] \ar@<2ex>[d]\sp{d\sp{2}}&
\hspace*{-8mm}= \tWn{1}{n} \\
\bZ\otimes(I\times\Deln{2}) \ar@{-->}[rrr]\sp(0.65){\hHuk{2}{n}}
 \ar@ /^{2.0pc}/[rrrr]\sp{\Huk{2}{n-1}} \ar@{.}[d]
&& &
\PoW{n-3}{n}\hspace*{4mm}\times \ar@{.}[d]&
\Wn{2}{n-1} \ar@{.}[d]&
\hspace*{-8mm}= \tWn{2}{n}\\
\bZ\otimes (I\times\Deln{n-1}) \ar@{-->}[rrr]\sp(0.65){\hHuk{n-1}{n}}
 \ar@/^{2.0pc}/[rrrr]\sp{\Huk{n-1}{n-1}}
 \ar@<-2.5ex>[d]\sb{\eta\sp{0}\sb{\ast}}\sp{\dotsc}
 \ar@<2.5ex>[d]\sp{\eta\sp{n}\sb{\ast}}
&&&
\PoW{0}{n}\hspace*{4mm}\times \ar[d]\sb(0.4){p}\sp(0.4){\dif{n-2}}
& \Wn{n-1}{n-1}  \ar[ld]\sp{\udz{n}}  &
\hspace*{-8mm}= \tWn{n-1}{n}\\
\bZ\otimes(I\times\Deln{n}) \ar@{-->}[rrr]\sp{\hHuk{n}{n}}
&&& \oW{n} &
}
\end{remark}

\begin{mysubsection}{Folding polytopes}
\label{sfoldp}
Consider the iterated trivial fibration of \wref[:]{eqmodpathloopve}
\begin{myeq}\label{eqxij}
\xi\sp{j}:=
P\Omega\sp{j-1}\sigma\sp{1}\circ\cdots\circ P\Omega\sigma\sp{j-1}\circ P\sigma\sp{j}\circ P\tau\sp{j}:
\PoW{j}{n}\epic P\Omega\sp{j}\oW{n}~.
\end{myeq}
\noindent If we identify the $k$-simplex \w{\Deln{k}} with (a quotient of)
the $k$-cube \w[,]{I\sp{k}} each map
\w[,]{\hHuk{k}{n}:\bZ\otimes(I\times\Deln{k})\to\PoW{n-k-1}{n}} after
post-composing with
\w[,]{\xi\sp{n-k-1}:\PoW{n-k-1}{n}\to P\Omega\sp{n-k-1}\oW{n}}
can be identified by adjunction with a pointed map
\w{\tHuk{k}:(\bZ\otimes I)\otimes I\sp{n}\to\oW{n}} taking certain facets of
\w{I\sp{n}} to the base point.

Moreover, the compatibility conditions of \wref{eqnewfaces} translate into
requirements that the restrictions of the maps \w{\tHuk{k}} to certain facets of
\w{I\sp{n}} match up in an appropriate way. This information can be encoded
by gluing together \w{n+1} $n$-cubes (corresponding to cosimplicial dimensions
\w[)]{0,1,\dotsc,n} along their facets to obtain a single $n$-dimensional cubical
complex, as follows:
\end{mysubsection}

\begin{defn}\label{dfoldp}
The barycentric subdivision, as a triangulation of the standard $n$-simplex
\w[,]{\Deln{n}} exhibits it as a PL-cone on its boundary
\w[.]{\partial \Deln{n}} We may similarly define by induction a triangulation
of the standard $n$-cube \w{I\sp{n}=[0,1]\sp{n}} obtained as the cone on its
boundary \w{\partial I\sp{n}} (more precisely, the join of the
barycenter of \w{I\sp{n}} with the inductively-defined triangulation of
\w[).]{\partial I\sp{n}}

This allows us to define PL-homeomorphisms
\w[,]{\zeta\sp{n}:I\sp{n}\to\Deln{n}} starting with the obvious isomorphism for
\w[,]{n=1} taking boundary to boundary, and extending
to the interior by applying the join with the respective barycenters.

For each \w{n\geq 2}, we consider the
corner $C$ of \w{\partial I\sp{n}} consisting of all \wwb{n-1}facets
\w{(E\sb{k})\sb{k=0}\sp{n-1}} incident with the fixed vertex \w{v=(0,\dotsc,0)}
where \w[.]{E\sb{k}=\{(t\sb{1},\dots,t\sb{n})\in I\sp{n}~:\  t\sb{k+1}=0\}}
 We use the previously defined \w{\zeta\sp{n-1}} to identify
 \w{E\sb{k}\cong I\sp{n-1}} with the $k$-th face \w{\Delnk{n}{k}}
 of \w[.]{\Deln{n}} For the complementary corner \w{C'} (incident with the
 vertex  \w{v'=(1,\dotsc,1)} opposite $v$), we use the orthogonal projection
from the last  vertex of \w{\Deln{n}} onto the face \w{\Delnk{n}{n}}
 opposite it to obtain a subdivision of \w{\Delnk{n}{n}} into
 $n$ \wwb{n-1}dimensional simplices, which we identify with the
 $n$ \wwb{n-1}dimensional facets of \w{\C'} using \w[.]{(\zeta\sp{n-1})\sp{-1}}

 Now for each \w{n\geq 1} consider \w{n+1} standard $n$-cubes
 \w[,]{\Ink{n}{0},\dotsc,\Ink{n}{n}} where we have a PL-isomorphism
 \w{\zeta\sp{n}\sb{k}} as follows:
 $$
 \Ink{n}{k}~\cong~I\sp{k}\times I\sp{n-k}~\xra{\zeta\sp{k}\times\Id}~
 \Deln{k}\times I\sp{n-k}~.
 $$
 \noindent For \w{k<n} we think of the first direction
 of \w{I\sp{n-k}} as the \emph{path} direction, and the remaining \w{n-k-1}
 directions as \emph{loop} directions. This allows us to represent a pointed map
 \w{h:\hbZ\otimes \Deln{k}\to P\Omega\sp{n-k-1}\oW{}} by a map
 \w[.]{\widehat{h}:\hbZ\otimes \Ink{n}{k}\to\oW{}} in a canonical way (sending
 certain facets of \w{\Ink{n}{k}} to the base-point).

  For any \w{1\leq k\leq n} we have two \wwb{k-1}faces \w{\Delnk{k}{0}}
 and  \w{\Delnk{k}{1}} of \w[,]{\Deln{k}} and the \wwb{n-1}dimensional prisms
 \w{\Delnk{k}{0}\times I\sp{n-k}}
 and  \w{\Delnk{k}{1}\times I\sp{n-k}} are identified under the map
 \w{(\zeta\sp{n}\sb{k})\sp{-1}} with two \wwb{n-1}dimensional facets of
 \w[,]{\Ink{n}{k}} which we denote by \w{\Bk{k}{0}} and
 \w[,]{\Bk{k}{1}} respectively. By our convention, if \w{0\leq k<n} we have
 another special \wwb{n-1}dimensional facet of \w[,]{\Ink{n}{k}} denoted by
 $$
 \Ck{k}:=\{(t\sb{1},\dotsc,t\sb{n})\in I\sp{k}\times I\sp{n-k}~|\ t\sb{k+1}=0\}
 $$
 \noindent (the zero-face in the "path direction").

 We now define the $n$-th \emph{folding polytope} \w[,]{\Pn{n}} for each
\w[,]{n\geq 2} by taking the disjoint union of the
\w{n+1} $n$-cubes \w[,]{\Ink{n}{0},\dotsc,\Ink{n}{n}} and identifying
 \w{\Bk{k}{1}} with \w{\Ck{k-1}} for each \w[.]{1\leq k\leq n}
\end{defn}

\begin{lemma}\label{lfoldp}
For each \w[,]{n\geq 2} the folding polytope \w{\Pn{n}} is homeomorphic to
an $n$-ball, with boundary \w{\partial\Pn{n}} homeomorphic to an \wwb{n-1}-sphere.
\end{lemma}

\begin{remark}\label{rfoldp}
 Note that all the faces \w{(\Bk{k}{1})\sb{k=1}\sp{n}} and
 \w{(\Ck{k})\sb{k=0}\sp{n-1}} are now interior to \w[,]{\Pn{n}} while
 the remaining facets  of the cubes \w[,]{\Ink{n}{k}} including
 \w[,]{(\Bk{k}{0})\sb{k=1}\sp{n}} constitute the boundary
 \w[.]{\partial\Pn{n}}
\end{remark}

\begin{example}\label{egfoldp}
The four constituent $3$-cubes of \w{\Pn{3}} are illustrated in
Figure \ref{fig4}, with the dotted arrows indicating glued faces.
Note that the two faces \w{\Bk{k}{0}} and \w{\Bk{k}{1}} are adjacent
for \w[,]{2\leq k\leq n} while \w{\Bk{1}{0}} and \w{\Bk{1}{1}} are
opposite each other (since the same is true of \w{\Delnk{k}{0}} and
\w{\Delnk{k}{1}} in \w[).]{\Deln{k}} On the other hand, \w{\Ck{k}} is always
adjacent to both \w{\Bk{k}{0}} and \w[.]{\Bk{k}{1}}

%
%
\begin{figure}[htbp]
\begin{center}
\begin{picture}(405,130)(0,-15)
%
%
\put(47,110){$\Ink{3}{0}$}
%
%
\put(0,0){\circle*{3}}
\put(0,0){\line(1,0){90}}
\put(0,0){\line(1,1){30}}
\put(0,0){\line(0,1){90}}
\put(90,0){\circle*{3}}
\put(90,0){\line(-1,1){30}}
\put(90,0){\line(0,1){90}}
\put(0,90){\circle*{3}}
\put(0,90){\line(1,-1){30}}
\put(0,90){\line(1,0){90}}
\put(90,90){\circle*{3}}
\put(90,90){\line(-1,-1){30}}
%
%
\put(30,30){\circle*{3}}
\put(30,30){\line(1,0){30}}
\put(30,30){\line(0,1){30}}
\put(60,30){\circle*{3}}
\put(60,30){\line(0,1){30}}
\put(30,60){\circle*{3}}
\put(30,60){\line(1,0){30}}
\put(60,60){\circle*{3}}
\put(68,43){\scriptsize $\Ck{0}$}
\bezier{30}(80,29)(105,19)(130,9)
\put(84,28){\vector(-4,1){10}}
\put(128,9){\vector(3,-1){10}}
%
%
\put(152,110){$\Ink{3}{1}$}
%
%
\put(105,0){\circle*{3}}
\put(105,0){\line(1,0){90}}
\put(105,0){\line(1,1){30}}
\put(105,0){\line(0,1){90}}
\put(195,0){\circle*{3}}
\put(195,0){\line(-1,1){30}}
\put(195,0){\line(0,1){90}}
\put(105,90){\circle*{3}}
\put(105,90){\line(1,-1){30}}
\put(105,90){\line(1,0){90}}
\put(195,90){\circle*{3}}
\put(195,90){\line(-1,-1){30}}
\put(144,73){\scriptsize $\Bk{1}{0}$}
\put(144,13){\scriptsize $\Bk{1}{1}$}
%
%
\put(135,30){\circle*{3}}
\put(135,30){\line(1,0){30}}
\put(135,30){\line(0,1){30}}
\put(165,30){\circle*{3}}
\put(165,30){\line(0,1){30}}
\put(135,60){\circle*{3}}
\put(135,60){\line(1,0){30}}
\put(165,60){\circle*{3}}
\put(173,43){\scriptsize $\Ck{1}$}
\bezier{6}(195,45)(200,45)(205,45)
\put(195,45){\vector(-1,0){8}}
\put(207,45){\vector(1,0){8}}
%
%
\put(257,110){$\Ink{3}{2}$}
%
%
\put(210,0){\circle*{3}}
\put(210,0){\line(1,0){90}}
\put(210,0){\line(1,1){30}}
\put(210,0){\line(0,1){90}}
\put(300,0){\circle*{3}}
\put(300,0){\line(-1,1){30}}
\put(300,0){\line(0,1){90}}
\put(210,90){\circle*{3}}
\put(210,90){\line(1,-1){30}}
\put(210,90){\line(1,0){90}}
\put(300,90){\circle*{3}}
\put(300,90){\line(-1,-1){30}}
\put(248,73){\scriptsize $\Bk{2}{0}$}
\put(252,43){\scriptsize $\Ck{2}$}
%
%
\put(240,30){\circle*{3}}
\put(240,30){\line(1,0){30}}
\put(240,30){\line(0,1){30}}
\put(270,30){\circle*{3}}
\put(270,30){\line(0,1){30}}
\put(240,60){\circle*{3}}
\put(240,60){\line(1,0){30}}
\put(270,60){\circle*{3}}
\put(221,43){\scriptsize $\Bk{2}{1}$}
\bezier{25}(270,45)(293,45)(316,45)
\put(273,45){\vector(-1,0){8}}
\put(312,45){\vector(1,0){8}}
%
%
\put(362,110){$\Ink{3}{3}$}
%
%
\put(315,0){\circle*{3}}
\put(315,0){\line(1,0){90}}
\put(315,0){\line(1,1){30}}
\put(315,0){\line(0,1){90}}
\put(405,0){\circle*{3}}
\put(405,0){\line(-1,1){30}}
\put(405,0){\line(0,1){90}}
\put(315,90){\circle*{3}}
\put(315,90){\line(1,-1){30}}
\put(315,90){\line(1,0){90}}
\put(405,90){\circle*{3}}
\put(405,90){\line(-1,-1){30}}
\put(357,73){\scriptsize $\Bk{3}{0}$}
%
%
\put(345,30){\circle*{3}}
\put(345,30){\line(1,0){30}}
\put(345,30){\line(0,1){30}}
\put(375,30){\circle*{3}}
\put(375,30){\line(0,1){30}}
\put(345,60){\circle*{3}}
\put(345,60){\line(1,0){30}}
\put(375,60){\circle*{3}}
\put(325,43){\scriptsize $\Bk{3}{1}$}
\end{picture}
\end{center}

\caption[fig4]{The four $3$-cubes of $\Pn{3}$}
\label{fig4}
\end{figure}
\end{example}

\begin{lemma}\label{lmapfoldp}
Given two maps \w{\fu{0},\fu{1}:\bZ\to\bY}
which induce the same algebraic homomorphism of \TRal[s]  \w{\HiR{\bY}\to\HiR{\bZ}}
and a sequential realization $\cW$ for $\bY$, let
\w{\Huk{}{n-1}:\bZ\otimes I\otimes\Du\to\W{n-1}} be an \wwb{n-1}strand as
in \S \ref{solhot}. Then there is a one-to-one correspondence between
collections of  maps
$$
\hHuk{k}{n}:\bZ\otimes(I\times\Deln{k})\to\PoW{n-k-1}{n}
\hsp \text{for}\hsm 0\leq k\leq n
$$
\noindent as in \S \ref{rfoldpoly} (satisfying \wref[),]{eqnewfaces} and maps
\w{h:(\bZ\otimes I)\otimes\Pn{n}\to\oW{n}} such that
\begin{myeq}\label{eqmapfoldp}
h\rest{(\bZ\otimes I)\otimes \Bk{k}{0}}=
\xi\sp{n-k-1}\circ\Fk{k-1}\circ v\sp{k-1}\circ\Huk{k-1}{n-1}\hsm
\text{for}\hsm 1\leq k\leq n \hs\text{and}\hsm
h\rest{(\bZ\otimes I)\otimes E}=\ast,
\end{myeq}
\noindent where
\w[.]{E:=\partial\Pn{n}\setminus\bigcup\sb{k=1}\sp{n}\,\Bk{k}{0}}
\end{lemma}

\begin{proof}
Given a map
\w[,]{\hHuk{k}{n}:\bZ\otimes I\otimes\Deln{k}\to\PoW{n-k-1}{n}}
we obtain a map \w[,]{\tHuk{k}\lo{n}:(\bZ\otimes I)\otimes I\sp{n}\to\oW{n}}
where we identify \w{I\sp{k}} with \w{\Deln{k}} using
\w[,]{\zeta\sp{k}} and taking the \wwb{k+1}st coordinate for the path direction and
the remaining \w{n-k-1} coordinates for the loop directions, as in \S \ref{dfoldp}.

The first condition in \wref{eqnewfaces} says that on \w{\Bk{k}{0}} (corresponding
to the $0$-face of \w[),]{\Deln{k}} \w{\hHuk{k}{n}} equals
\w[.]{\Fk{k-1}\circ v\sp{k-1}\circ\Huk{k-1}{n-1}} The second condition there says
that on \w{\Bk{k}{1}} (corresponding to the $1$-face of \w[),]{\Deln{k}}
\w{\hHuk{k}{n}} equals  \w{\dif{n-k}\circ\hHuk{k-1}{n}}
(where \w{\delta\sp{n-k}} is defined in \wref[),]{eqtau} which stated in terms
of cubes means that it coincides with \w{\hHuk{k-1}{n-1}}
restricted to \w[.]{\Ck{k-1}} Since the coface maps \w{d\sp{i}} into
\w{P\Omega\sp{n-k-1}\oW{n}} vanish for \w[,]{i\geq 2} and
\w{\hHuk{k-1}{n}} also vanishes at the other end of the path direction, and at both
ends of the loop directions, we obtain the description above.

Conversely, given such a map $\tHuk{}$, we use its restrictions to the $n$-cubes
\w{\Ink{0}{n},\dotsc,\Ink{n}{n}} to define the maps \w[,]{\tHuk{k}} and thus maps
\w{\ppp\hHuk{k}{n}:\bZ\otimes(I\times\Deln{k})\to P\Omega\sp{n-k-1}\oW{n}}
satisfying
\begin{myeq}\label{eqnewfacesm}
\begin{cases}
~\ppp\hHuk{k}{n}\circ \eta\sp{0}\sb{\ast}~=&~
\xi\sp{n-k-1}\circ\Fk{k-1}\circ v\sp{k-1}\circ\Huk{k-1}{n-1}~,\\
~\ppp\hHuk{k}{n}\circ \eta\sp{1}\sb{\ast}~=&~
\xi\sp{n-k-1}\circ\dif{n-k}\circ\ppp\hHuk{k-1}{n}~,\\
~\ppp\hHuk{k}{n}\circ \eta\sp{i}\sb{\ast}~=&~0\hsm\text{for}\hsm i\geq 2
\end{cases}
\end{myeq}
\noindent for \w{\xi\sp{j}} as in \wref[.]{eqxij} We now show by induction
on \w{0\leq k} that these lift to maps
\w{\hHuk{k}{n}:\bZ\otimes(I\times\Deln{k})\to\PoW{n-k-1}{n}} satisfying
\wref[,]{eqnewfaces} and
\begin{myeq}\label{eqnewold}
\ppp\hHuk{k}{n}~=~\xi\sp{n-k-1}\circ\hHuk{k}{n}
\end{myeq}

Indeed, the inductively-defined lift \w{\hHuk{k-1}{n}} induces a map
\w{L:\bZ\otimes(I\times\partial\Deln{k})\to\PoW{n-k-1}{n}} fitting into
the following solid commutative diagram:
\mytdiag[\label{eqliftingh}]{
\bZ\otimes(I\times\Deln{k-1}) \ar@{^{(}->}[dd]\sb{\eta\sp{0}\sb{\ast}}
\ar@/^{2.5pc}/[rrr]\sp{\Huk{k-1}{n-1}} &
\bZ\otimes(I\times\Deln{k-1}) \ar@{^{(}->}[ldd]\sb{\eta\sp{1}\sb{\ast}}
\ar[rd]\sp{\hHuk{k-1}{n}} && \Wn{k-1}{n-1} \ar[d]\sp{v\sp{k-1}} \\
&\bZ\otimes(I\times\Deln{k-1})
\ar@{^{(}->}[ld]\sb{\eta\sp{i}\sb{\ast}}\sp{(i\geq 2)}
\ar[rrd]\sb{0} & \PoW{n-k}{n} \ar[rd]\sp{\dif{n-k}} &
C\sp{k-1}\W{n-1} \ar[d]\sp{\Fk{k-1}} \\
\bZ\otimes(I\times\partial\Deln{k}) \ar[rrr]\sp{L} \ar@{^{(}->}[d]\sb{\inc\sb{\ast}}
&&& \PoW{n-k-1}{n} \ar@{->>}[d]\sb{\simeq}\sp{\xi\sp{n-k-1}} \\
\bZ\otimes(I\times\Deln{k}) \ar[rrr]\sb{\ppp\hHuk{k}{n}}
\ar@{.>}[rrru]\sb{\hHuk{k}{n}} &&& P\Omega\sp{n-k-1}\oW{n}
}
\noindent where the upper squares fit together to define $L$ by induction,
using \wref[,]{eqnewfaces} and the bottom solid square then commutes
by \wref{eqnewfacesm} and \wref[.]{eqnewold}

Since \w{\bZ} is cofibrant in \w{\C} by Assumption \wref[,]{amodel}  the map
\w{\inc\sb{\ast}} is a cofibration   cf. \cite[ II, \S 2]{QuiH}. Moreover
\w{\xi\sp{n-k-1}} is a trivial fibration, so we have the lifting \w{\hHuk{k}{n}}
by the LLP. The fact that \wref{eqliftingh} commutes implies that
\wref{eqnewfaces} holds for $k$, too. To start the induction for \w[,]{k=0}
we just need the fact that \w{\xi\sp{n-1}} is a trivial fibration and \w{Z\otimes I}
is cofibrant with \w[,]{L=0} since \wref{eqnewfaces} is then vacuous.
\end{proof}

\begin{defn}\label{dvalstrand}
Assume given initial data \w{(\cW,\fu{0},\fu{1}:\bZ\to\bY)} with a
corresponding \wwb{n-1}strand \w[,]{\Htl{n-1}=(\Huk{}{m})\sb{m=0}\sp{n-1}}
as in \S \ref{solhot}. We associate to this  a map
\w{g:(\bZ\otimes I)\otimes\partial\Pn{n}\to\oW{n}} which sends
\w{(\bZ\otimes I)\otimes\Bk{k}{0}} to \w{\oW{n}} by
\w{\Fk{k}\circ\Huk{k}{n}} for each \w[,]{1\leq k\leq n} and all other
\wwb{n-1}cubes of \w{\partial\Pn{n}} to the base-point.
Here we use the convention of \wref[,]{eqconvent} so \w[.]{\Fk{n}=\udz{n}}

Since at most two additional \wwb{n-1}facets of \w{I\sp{n}\lo{k}} are
identified with \wwb{n-1}facets of \w[,]{I\sp{n}\lo{k\pm 1}} we may think of
\w{P\Omega\sp{n-k}\oW{n}} as contained in
\w[,]{\mapa(d\sb{1}\sp{0}\Ink{n}{k},\oW{n})} so the map induced by
\w{\Huk{k-1}{n}\circ\Fk{k}} is well-defined. Moreover, these maps are
compatible for adjacent values of $k$ by \wref[.]{eqnewfaces}

By Lemma \ref{lfoldp} we can think of  $g$ as a map
\w[,]{(\bZ\otimes I)\otimes\bS{n-1}\to\oW{n}} and because all maps are pointed,
this actually factors through the half-smash
$$
(\bZ\otimes I)\ltimes\bS{n-1}~:=~
((\bZ\otimes I)\times\bS{n-1})/(\ast\times\bS{n-1})~,
$$
\noindent which can be canonically identified with
\w{\Sigma\sp{n-1}(\bZ\otimes I)\vee (\bZ\otimes I)} (see \cite{BJiblSL}).
Moreover, the map \w{(\bZ\otimes I)\to \oW{n}} in question is nullhomotopic
for \w[,]{n\geq 1} so we may restrict attention to the factor
\w[,]{g':\Sigma\sp{n-1}(\bZ\otimes I)\to \oW{n}}
and define the \emph{value} of the \wwb{n-1}strand \w{\Htl{n-1}}
to be the class
\begin{myeq}\label{eqvalmapstr}
\Val{\Htl{n-1}}~:=~[g']\in[\Sigma\sp{n-1}(\bZ\otimes I),\,\oW{n}]~\cong~
[\bZ,\,\Omega\sp{n-1}\oW{n}]~,
\end{myeq}
\noindent so it consists of a set of cohomology classes for $\bZ$.
\end{defn}

\begin{prop}\label{pvanish}
Under the assumptions of Lemma \ref{lmapfoldp}, \w{\Val{\Htl{n-1}}=0}
if and only \w{\Htl{n-1}} extends to an $n$-strand \w[.]{\Htl{n}}
\end{prop}

\begin{proof}
The \wwb{n-1}strand \w{\Htl{n-1}} extends to an $n$-strand \w{\Htl{n}} if
and only if we have a collection of maps
\w{\hHuk{k}{n}:\bZ\otimes(I\times\Deln{k})\to\PoW{n-k-1}{n}}
\wb{0\leq k\leq n} as in \S \ref{rfoldpoly}, satisfying \wref[,]{eqnewfaces}
and by Lemma \ref{lmapfoldp} this corresponds to a map
\w{(\bZ\otimes I)\otimes\Pn{n}\to\oW{n}} whose restriction to
\w{(\bZ\otimes I)\otimes\partial\Pn{n}} is the map $g$ determined by
\w{\Htl{n-1}} as in Definition \ref{dvalstrand}. The map $g$ extends to
\w{(\bZ\otimes I)\otimes\Pn{n}} if and only if \w{g'} is nullhomotopic.
\end{proof}

\begin{mysubsection}{Correspondence of strands for maps}
\label{scorrespstrmap}
Given \w{\fu{0},\fu{1}:\bZ\to\bY} with
\w[,]{\fu{0}\sp{\ast}=\fu{1}\sp{\ast}:\HiR{\bY}\to\HiR{\bZ}}
an $n$-stage comparison map \w{\Phi:\cW\to\,\ccWp} between two sequential
realizations for $\bY$ as in \wref[,]{eqncorresp} and two
$n$-strands \w{\Htl{n}} and \w{\Htlp{n}} for $\cW$ and \w[,]{\ccWp} respectively,
we write \w{\Htlp{n}=\rs(\Htl{n})} if
\w{\Hupk{k}{m}=\rnk{m}{k}\circ\Huk{k}{m}:\bZ\otimes(I\times\Deln{k})\to\Wpn{k}{m}}
and \w{\Htl{n}=\es(\Htlp{n})} if
\w{\Huk{k}{m}=\enk{m}{k}\circ\Hupk{k}{m}:\bZ\otimes(I\times\Deln{k})\to\Wn{k}{m}}
for each \w{0\leq k\leq m\leq n} (compare \S \ref{scorrespstr}).

By comparing \wref{eqmapmodpathloop} and \wref{eqcompf} with \wref{eqexttothree}
and Definition \ref{dvalstrand}, we see that
\begin{myeq}\label{eqcorrvalsmap}
\Val{\rs(\Htl{n})}=\oar{n}\sb{\ast}(\Val{\Htl{n}})\hsp \text{and}\hsp
\Val{\es(\Htlp{n})}=\oen{n}\sb{\ast}(\Val{\Htlp{n}})~,
\end{myeq}
\noindent as in \wref[,]{eqcorrvals} so
\begin{myeq}\label{eqvancorrmap}
\begin{array}{l}
(a)\hs \Val{\Htl{n}}=0\hsm \text{if and only if}\hsn \Val{\rs(\Htl{n})}=0\\
(b)\hs \text{If}\hsn\Val{\Htlp{n}}=0\hsm \text{then}\hsn \Val{\es(\Htlp{n})}=0~,
\end{array}
\end{myeq}
\noindent as in \wref[.]{eqvancorr}

We define weak and strong equivalences relations on strands as
in \S \ref{dequivstr}.
\end{mysubsection}

\begin{defn}\label{dvanish}
Given two maps \w{\fu{0},\fu{1}:\bZ\to\bY} inducing the same homomorphism of
\TRal[s]  \w[,]{\phi:\HiR{\bY}\to\HiR{\bZ}} the associated
\emph{universal $n$-th order cohomology operation} \w{\lrfn{n}} which assigns
to an \wwb{n-1}strand \w{\Htl{n-1}} for this data the value:
$$
\lrfn{n}(\Htl{n-1})~:=~\Val{\Htl{n-1}}~\in \Gamma'\lin{\Omega\sp{n-1}\oW{n}}~,
$$
\noindent where \w[.]{\Gamma':=\HiR{\bZ}}
We say that the operation \w{\lrfn{n}} \emph{vanishes} if there is
a cofibrant $\cW$ with an \wwb{n-1}strand \w{\Htl{n-1}} for \w{(\cW,\fu{0},\fu{1})}
such that \w[.]{\lrfn{n}(\Htl{n-1})=0} Note that this depends only on the strong
equivalence class of \w[.]{\Htl{n-1}}

As in Definition  \ref{dcohervan}, we then say that the system
\w{\lrf=(\lrfn{n})\sb{n=2}\sp{\infty}} of $n$-th
order cohomology operations for \w{(\fu{0},\fu{1})} \emph{vanishes coherently}
for \w{(\cW,\fu{0},\fu{1})} if there is an $\infty$-strand \w{\Htl{\infty}} for
this data \wh that is, for each \w[,]{n\geq 1} we have an $n$-strand
\w{\Htl{n}} for \w{(\cW,\fu{0},\fu{1})} such that \w[,]{\Val{\Htl{n}=0}}
which extends to the \wwb{n+1}strand \w{\Htl{n+1}} using
Proposition \ref{pvanish}.
\end{defn}

The proof of Key Lemma \ref{lkey} shows also:

\begin{lemma}\label{lkeymap}
Given \w{\fu{0},\fu{1}:\bZ\to\bY} as above, \w{\lrfn{n}}
vanishes if and only if for \emph{every} $n$-stage cofibrant sequential
realization $\cW$ of $\bY$, there is an \wwb{n-1}strand
\w{\Htl{n-1}} for \w{(\cW,\fu{0},\fu{1})} with \w[.]{\Val{\Htl{n-1}}=0}
\end{lemma}

Moreover, if \w{\lrfn{n}} vanishes at the \wwb{n-1}strand \w{\Htl{n-1}} for
\w[,]{(\cW,\fu{0},\fu{1})} then for any other $n$-stage cofibrant sequential
realization \w{\cWp} of $\bY$ we can choose the \wwb{n-1}strand
\w{\Htlp{n-1}} for \w{\cWp} to be weakly equivalent to \w[.]{\Htl{n-1}}

In analogy to Theorem \ref{tvanish} we therefore have:

\begin{thm}\label{tvanishmap}
For \w{R=\Fp} or a field of characteristic $0$, let \w{\fu{0},\fu{1}:\bZ\to\bY}
be two maps between  $R$-good spaces which induce the
same map of \TRal[s] \w[.]{\HiR{\bY}\to\HiR{\bZ}} Then the following are equivalent:
\begin{enumerate}
\renewcommand{\labelenumi}{(\alph{enumi})~}
\item Then the system of higher cohomology operations \w{\lrf}
vanishes coherently for \emph{some} cofibrant sequential realization $\cW$
of $\bY$.
\item \w{\lrf} vanishes coherently for \emph{every} cofibrant sequential
realization of $\bY$.
\item The maps \w{\fu{0}} and \w{\fu{1}} are $R$-equivalent (see \S \ref{sgcge}).
\end{enumerate}
\end{thm}

\begin{proof}
(a)$\Leftrightarrow$(b) by Lemma \ref{lkeymap}\vsn.

(a)$\Rightarrow$(c): \ Note that the projection
\w{p\sb{\bX}:\bX\otimes\Du\to\cu{\bX}} is a trivial Reedy fibration for any
\w[.]{\bX\in\C} Since
\w{\bZ\amalg\bZ\xra{i\sb{0}\bot i\sb{1}}\bZ\otimes I\xra{\sigma}\bZ} is a
cylinder object in $\C$ (see \cite[I, \S 1]{QuiH}, the same is true after
applying \w[.]{(-)\otimes\Du} An $\infty$-strand \w{\Htl{\infty}} for
a sequential realization $\cW$ (with associated \w[)]{\bv:\bY\to\Wu}
defines a map \w[,]{H:(\bZ\otimes I)\otimes\Du\to\Wu} fitting into a commutative
diagram of cosimplicial spaces:
\mydiagram[\label{eqpathobject}]{
\bZ\otimes\Du \ar@{->>}[rr]\sb{\simeq}\sp{p\sb{\bZ}}
\ar[d]\sp{i\sb{j}\otimes\Id} && \cu{\bZ}
\ar[rrrd]\sp{\cu{f\sb{j}}}&&& \\
(\bZ\amalg\bZ)\otimes\Du \ar[d]\sb{(i\sb{0}\bot i\sb{1})\otimes\Id}
\ar[rrr]\sb(0.6){(f\lo{0}\bot f\lo{1})\otimes\Id}
&&&
\bY\otimes\Du \ar@{->>}[rr]\sb{\simeq}\sp{p\sb{\bY}} && \cu{\bY}\ar[d]\sp{\bv}\\
(\bZ\otimes I)\otimes\Du \ar[rrrrr]\sp(0.6){H} &&&&& \Wu
}
\noindent for \w[.]{j=0,1} Applying \w{\Tot} yields a cylinder object
\w{\Tot((\bZ\otimes I)\otimes\Du)}
for \w[,]{\Tot((\bZ\amalg\bZ)\otimes\Du)} and a homotopy \w{\Tot H} between
\w{\bv\sb{\ast}\circ\fu{0}\circ\Tot(p\sb{\bZ})} and
\w[.]{\bv\sb{\ast}\circ\fu{1}\circ\Tot(p\sb{\bZ})}
Since \w{\Tot(p\sb{\bZ})} is a weak equivalence and
\w{\bv\sb{\ast}:\bY\to\Tot\Wu} is an $R$-equivalence, we see that \w{\fu{0}}
and \w{\fu{1}} are $R$-equivalent\vsn.

(c)$\Rightarrow$(b): \ Let $\cW$ be any cofibrant sequential
realization for $\bY$ with associated \w[.]{\bv:\bY\to\Wu} By Definition \ref{dscr},
\w{\Wu} is Reedy fibrant. Thus \w{\Tot\Wu} is an $R$-complete Kan complex, with
\w{\bv\sb{\ast}:\bY\to\Tot\Wu} the $R$-completion map, and so the $R$-equivalent
maps \w{\bv\sb{\ast}\circ\fu{0}} and \w{\bv\sb{\ast}\circ\fu{1}} are actually
homotopic (see \cite[I, Lemma 5.5]{BKanH}). We may therefore choose a homotopy
\w{F:\bZ\otimes I\to\Tot\Wu} between them, whose adjoint is the map of cosimplicial
spaces \w{\widetilde{F}:\bZ\otimes I\otimes\Du\to\Wu} (see \cite[I, 3.3]{BKanH}).
Composing $\widetilde{F}$ with the structure maps \w{\Wu\to\W{n}} for the limit of
\wref{eqtower} yields a compatible sequence of cosimplicial maps
\w[.]{\Huk{}{n}:\bZ\otimes I\otimes\Du\to\W{n}}

This defines compatible $n$-strands for $\cW$ and all \w[,]{n\geq 1} showing
that the system \w{\lrf} of higher order operations vanishes by
Proposition \ref{pvanish}.
\end{proof}

\begin{corollary}\label{ccimaps}
If \w{\fu{0},\fu{1}:\bZ\to\bY} are maps between $R$-complete Kan complexes
inducing the same map of \TRal[s]  \w[,]{\psi:\HiR{\bY}\to \HiR{\bZ}} the
system of higher operations \w{(\fu{0},\fu{1})} is a complete set of
invariants for the homotopy classes \w{[\fu{0}]} and \w[.]{[\fu{1}]}
\end{corollary}

\begin{mysubsect}{A rational example for a map}
\label{sratexm}

As in \S \ref{sratex}, we now consider an example of the obstruction to a map
\w{f:\bZ\to\bY} being rationally trivial when \w{f\sp{\ast}:\HiQ{\bY}\to\HiQ{\bZ}}
is the zero map:

Let \w{\bZ:=\bS{2n-1}\sb{\QQ}} and \w[,]{\bY:=(\bS{n}\vee\bS{n})\sb{\QQ}} for
\w{n>1} odd, with \w{f:=[\iota\sb{n},\iota'\sb{n}]:\bZ\to\bY} the Whitehead product
map. The free CDGA model for $\bY$ is \w{(A\sp{\ast},d)} given in degrees $\leq 2n$
by \w{A\sp{\ast}=\QQ\lra{x\sb{n},y\sb{n},u\sb{2n-1}}} with \w[,]{d(u)=xy} while
$\bZ$ has the formal CDGA model \w[.]{B\sp{\ast}=\KsQ{\bz\sb{2n-1}}} The CDGA model
for $f$ is \w{\varphi:A\sp{\ast}\to B\sp{\ast}} mapping $u$ to $z$.

Realizing the obvious minimal free algebraic resolution of \w[,]{\HiQ{\bY}}
we obtain the $1$-truncated augmented simplicial CDGA
\w{\tWln{\bullet}{1}\to A\sp{\ast}\to B\sp{\ast}} in degrees $\leq 2n$ \ depicted in
Figure \ref{fig5}.
%
%
\begin{figure}[htbp]
\begin{center}
\xymatrix@C=25pt{
\tWln{1}{1}~~= \ar@/_{0.2pc}/[d]\sb{d\sb{0}} \ar@/^{0.2pc}/[d]\sp{d\sb{1}} &&
\oWl{1}=\Lambda[\bu\sb{2n}] \ar[dl]\sb{d\sb{0}:\bu\mapsto\bx\by}
\ar[dr]\sp{d\sb{1}:\bu\mapsto \bu'} & \\
\Wln{0}{1}~~=\ar[d]\sb{\bv} &
\oWl{0}=\Lambda[\bx\sb{n},\by\sb{n}]\ar[d]\sp{\bv:\bx\mapsto x,\hsm\by\mapsto y}
&\amalg & C\oWl{1}= \Lambda(\bu'\sb{2n},\obu\sb{2n-1})
\ar[dll]\sp{\bv:\obu\,\mapsto -u,\hsm \bu'\mapsto xy}\\
A\sp{\ast}~~=\ar[d]\sb{\varphi}  & \Lambda[x,y,u]\ar[d]\sp{\varphi:\bu\mapsto z}\\
B\sp{\ast}~~=& \Lambda[z]\\
}
\end{center}
\caption[fig5]{$\tWln{\bullet}{1}$ \ in degrees $\leq 2n$}
\label{fig5}
\end{figure}

As noted above, the original map \w{f:\bZ\to\bY} is nullhomotopic
if and only if we can extend the composite \w{\tWln{\bullet}{1}\to B\sp{\ast}}
in the diagram in Figure \ref{fig5} to the cone \w[,]{C\tWln{\bullet}{1}\to B\sp{\ast}}
which by Theorem \ref{tvanishmap} is equivalent to the vanishing of the associated
system of higher operations.

The model for the cofiber of \w{\Wl{1}\hra C\oWl{1}}
(corresponding to the loop space \w[)]{\Omega\oW{1}} is the formal CDGA
\w[,]{\Sigma\oWl{1}:=\Lambda(\obu'\sb{2n-1})} and its cone
(corresponding to the path space \w[)]{P\Omega\oW{1}} is
\w[,]{C\Sigma\oWl{1}:=\Lambda(\obu''\sb{2n-1},\oobu\sb{2n-2})}
with \w[.]{d(\oobu\sb{2n-2})=\obu''\sb{2n-1}}

Thus in order to extend the given map
\w{\varphi\circ\bv:\Wln{0}{1}\to B\sp{\ast}} (sending \w{\obu\sb{2n-1}} to \w[)]{-z}
to its cone, we need a map \w{F:C\Sigma\oWl{1}\to B\sp{\ast}}
sending \w{\obu''\sb{2n-1}} to \w[.]{-z} However, since necessarily \w[,]{F(\oobu)=0}
this is impossible \wh that is, the secondary cohomology
operation does not vanish: its value is represented by the map
\w{\psi:\Sigma\oWl{1}\to B\sp{\ast}} defined \w[.]{\psi(\obu'\sb{2n-1})=z}
\end{mysubsect}

\appendix

%
%
\section{Proof of Theorem \protect{\ref{tres}}.}
\label{apfthm}

In this Appendix we state and prove Theorem \ref{tres} in a more general form
needed in \cite{BSenH}. For this purpose we recall the notion of a
\emph{mapping algebra}, which encodes the extra structure on the mapping spaces
\w{\mapa(\bY,\KR{n})} needed to recover the $R$-completion of $\bY$ from them
(see \cite{BSenH}).

\begin{defn}\label{deth}
An \emph{enriched sketch} \w{(\bT,\PP,\K)} is a small subcategory $\bT$
of a simplicial category $\C$ (see \S \ref{dsmcat}), with $\bT$ closed under a
given set of limits $\PP$ and under \w{(-)\sp{K}} for $K$ in a given subcategory
$\K$ of $\cS$.
We assume that all mapping spaces \w{\map\sb{\bT}(\bB,\bB')} are Kan complexes.

A \emph{\Tma} is a pointed simplicial functor \w{\fX:\bT\to\Sa} (written
\w[)]{\fX:\bB\mapsto\fX\lin{\bB}} which preserves the limits in $\PP$ and with
\w{\fX\lin{(\bB)\sp{K}}=(\fX\lin{\bB})\sp{K}} for any \w{\bB\in\bT} and
\w[.]{K\in\K} The category of \Tma[s] will be denoted by \w[.]{\MT}
See \cite[8]{BBlaC} and \cite[\S 1]{BSenH} for further details.

For any \w{\bY\in\C} we have a \emph{realizable} \Tma \w{\fMT\bY} defined for each
\w{\bB\in\bT} by \w[.]{\fMT\bY\lin{\bB}:=\map\sb{\C}(\bY,\bB)}
\end{defn}

\begin{remark}\label{rgalg}
Assume that in our enriched sketch \w[,]{(\bT,\PP,\K)} the category $\K$ includes
\w[,]{\Delta[0]\hra\Delta[1]} and $\PP$ includes all finite products and the
pullback squares
\mydiagram[\label{eqpathloop}]{
\ar @{} [drr] |(0.25){\framebox{\scriptsize{PB}}}
P\bB \ar@{^{(}->}[rr] \ar@{->>}[d]\sb{\simeq} &&
\bB\sp{\Delta[1]} \ar@{->>}[d]\sp{\simeq}\sb{\ev\sb{0}}&&
\ar @{} [drr] |(0.25){\framebox{\scriptsize{PB}}}
\Omega\bB \ar@{^{(}->}[rr]\sp{\iota\sb{\bB}} \ar@{->>}[d] &&
P\bB \ar@{->>}[d]\sb{\ev\sb{1}} \\
\ast \ar@{^{(}->}[rr] && \bB && \ast \ar@{^{(}->}[rr] && \bB
}
\noindent for any \w[.]{\bB\in\TsR} For each \Tma $\fX$ and objects $\bB$ and
\w{\{\bB\sb{i}\}\sb{i=1}\sp{n}} in $\bT$ we then have natural isomorphisms:
\begin{enumerate}
\renewcommand{\labelenumi}{(\alph{enumi})~}
\item \w[;]{i\sb{\Pi}:\fX\lin{\prod\sb{i=1}\sp{n}\,\bB\sb{i}}~\xra{\cong}~
\prod\sb{i=1}\sp{n}\,\fX\lin{\bB\sb{i}}}
\item \w[;]{i\sb{P}:\fX\lin{P\bB}~\xra{\cong}~P\fX\lin{\bB}}
\item \w[.]{i\sb{\Omega}:\fX\lin{\Omega\bB}~\xra{\cong}~\Omega\fX\lin{\bB}}
\end{enumerate}

We assume that the objects \w{B\in \bT} are fibrant in \w[.]{\C} In order to
simplify the proof of Theorem \ref{tresg} below, we make the following
ad hoc assumption: if \w{f:\Wu\to\Yu} is a map of cosimplicial object over $\C$
with each \w{\bW\sp{n}} and \w{\bY\sp{n}} in $\bT$ and \w{\Wu\to\Zu\to\Yu} is the
functorial factorization of $f$ as a (trivial) Reedy cofibration followed by a
(trivial) Reedy fibration, then each \w{\bZ\sp{n}} is in $\bT$, too.

This will hold, for instance, when \w{\C=\Sa} and $\bT$ consists of all simplicial
$R$-modules of cardinality $<\lambda$ (for some limit cardinal $\lambda$).
This is the example for which Case (1) of Theorem \ref{tresg} is needed in
\cite{BSenH}.
\end{remark}

We have the following enriched version of Lemma \ref{lfreeta}:

\begin{lemma}[see \protect{\cite[1.9]{BSenH}}]\label{lfreema}
For any \Tma $\fY$ and \w[,]{\bB\in\bT} there is a natural isomorphism
\w[.]{\Hom\sb{\MT}(\fMT\bB,\fY)~\xra{\cong}~\fY\lin{\bB}\sb{0}}
\end{lemma}

\begin{defn}\label{dgalg}
To any enriched sketch \w{(\bT,\PP,\K)} we associate a sketch
\w{\pi\sb{0}\bT} as in \S \ref{rskhtpy}, with the same objects and products
as $\bT$, where \w[.]{\Hom\sb{\Theta}(\bB,\bB'):=\pi\sb{0}\map\sb{\bT}(\bB,\bB')}
A map of \Tma[s] \w{f:\fX\to\fY} is a \emph{weak equivalence} if it induces a
weak equivalence \w{f\sb{B}:\fX\lin{\B}\to\fY\lin{\B}} for any \w[.]{\B\in\bT}
This means that $f$ induces an isomorphism
\w{f\sb{\#}:\pi\sb{0}\fX\to\pi\sb{0}\fY} of the corresponding \pTal[s.]
\end{defn}

\begin{example}\label{etsr}
For any commutative ring  $R$, let \w{\TsR\subseteq\Sa} denote the enriched sketch
whose objects are finite type $R$-GEMs of the form
\w{\prod\sb{i=1}\sp{\infty}\,\KP{V\sb{i}}{m\sb{i}}} for \w[,]{m\sb{i}\geq 0}
with \w{V\sb{i}} a finite dimensional free $R$-module
where $\K$ is as above and $\PP$ includes all such finite type products.
Since each \w{\bB\in\TsR} is an $R$-module object in \w[,]{\Sa} the same is true
of \w[,]{\fX\lin{\bB}} so \ww{\TsR}-\ma[s] actually take value in simplicial
$R$-modules. The enriched sketch \w{\TsR\sp{\lambda}} is defined analogously (see
\S \ref{egrtal}).

Note that \w{\pi\sb{0}\TsR} includes also $0$-dimensional Eilenberg-Mac~Lane spaces,
while the algebraic sketch \w{\TR} consists of
$R$-GEMs in dimensions $\geq 1$ (to avoid having to deal with the non-reduced
cohomology of non-connected spaces).  Thus \w[,]{\TR\subset\pi\sb{0}\TsR} which
motivates the following:
\end{example}

\begin{defn}\label{dscrm}
For $\bT$ an enriched sketch in $\C$ and $\fX$ a \Tma[,] let
\w{\Theta\subseteq\pi\sb{0}\bT} be a sub-algebraic sketch (still closed under
finite products, but not necessarily under loops), and let
\w{\Vd\to\pi\sb{0}\fX} be a CW resolution of the corresponding
\Tal[.] A \emph{sequential realization}
\w{\cW=\lra{\W{n},\,\tWu{n}}\sb{n\in\NN}} of \w{\Vd} for $\fX$
consists of a tower of Reedy fibrant and cofibrant cosimplicial objects as in
\wref[,]{eqtower} such that:
\begin{enumerate}
\renewcommand{\labelenumi}{(\alph{enumi})~}
\item We have an augmentation \w{\vn{n}:\fMT\Wn{0}{n}\to\fX} for the simplicial
\Tma \w[,]{\fMT\W{n}} realizing \w{\Vd\to\pi\sb{0}\fX} through simplicial
dimension $n$ \wh i.e., we have a natural isomorphism as in \wref[.]{eqnatisom}
\item The augmentation \w{\vn{n-1}:\fMT\Wn{0}{n-1}\to\fX} extends along the
\Tma map \w{\prn{n}\sp{\ast}:\fMT\Wn{0}{n-1}\to\fMT\Wn{0}{n}} to
\w[.]{\vn{n}:\fMT\Wn{0}{n}\to\fX}
\item Each \w{\W{n}}  is obtained from \w{\W{n-1}} as in \S \ref{dscr}(c).
\end{enumerate}
\end{defn}

\begin{remark}\label{rscrm}
We do not require the sequential realization
\w{\cW=\lra{\W{n},\,\tWu{n}}\sb{n\in\NN}} for $\fX$ to be cofibrant, as
in  \S \ref{dscr}(d). Therefore, in the explicit description
of \S \ref{senso} we may take the fibration sequences \wref{eqmodpathloop}
in step \textbf{(i)} to be the standard path-loop fibrations, so
\begin{myeq}\label{eqstloop}
\ooW{j}{n}~:=~\Omega\sp{j}\oW{n}
\end{myeq}
\noindent for all \w[.]{0\leq j\leq n-1} Thus \wref{eqdopb} becomes simply
\begin{myeq}\label{eqdopbg}
\tWn{k}{n}~=~\Wn{k}{n-1}\times P\Omega\sp{n-k-1}\oW{n}
\end{myeq}
\noindent for all \w{0\leq k\leq n} (with convention \wref{eqtoptw} still holding
for \w[).]{k=n}
\end{remark}

\begin{thm}\label{tresg}
Let $\bT$ be an enriched sketch $\bT$ in a model category $\C$ as in
\S \ref{rgalg}, and \w{\Theta\subseteq\pi\sb{0}\bT} an algebraic sketch:
\begin{enumerate}
\renewcommand{\labelenumi}{(\arabic{enumi})~}
\item If $\fX$ is a \Tma[,] and \w{\Vd} a CW resolution of \w{\Gamma:=\pi\sb{0}\fX}
with CW basis \w{(\oV{n})\sb{n\in\NN}} such that each \w{\oV{n}}
is realizable by an object \w[,]{\oW{n}\in\bT} then there is a sequential realization
\w{\cW=\lra{\W{n},\,\tWu{n}}\sb{n\in\NN}} of \w{\Vd} as in Definition \ref{dscrm}.
\item If $\Theta$ is allowable (\S \ref{dallowsk}), \w[,]{\bY\in\C} and \w{\Vd}
is any CW resolution of \w[,]{\Gamma:=\HiT{\bY}=\pi\sb{0}\fMT\bY} then \w{\Vd} has
a cofibrant sequential realization \w{\cW=\lra{\W{n},\,\tWu{n}}\sb{n\in\NN}}
of \w{\Vd} for $\bY$.
\end{enumerate}
Thus in both cases \w{\Vd} is realizable by \w[.]{\Wu=\holim\sb{n}\,\W{n}}
In case (2), if $\Theta$ is contained in some class of injective models $\G$
in $\C$, then \w{\bY\to\Wu} is a weak $\G$-resolution (cf. \S \ref{sgcge}).
\end{thm}

\begin{proof}
In case (2), for each \w{n\geq 0} we choose an object \w{\oW{n}\in\Theta}
realizing \w[.]{\oV{n}} In both cases, we then construct a sequential realization
$\cW$ by a double induction, where in the outer induction \w{\W{n}} is obtained
from \w{\W{n-1}} as in Definition \ref{dscr}(c), using an inner descending
induction on \w[\vsm:]{0\leq k\leq n}

\noindent\textbf{I.~Step $\mathbf{n=0}$ of the outer induction\vsn.}

We start the induction with \w{\W{0}:=\cu{\oW{0}}} (the constant cosimplicial
object), which is weakly $\G$-fibrant.
Because \w{\oV{0}} is a free \Tal[,] in case (1) the \Tal augmentation
\w{\vare:\oV{0}\to\Gamma} corresponds by Lemma \ref{lfreeta} to a unique element in
\w[,]{[\svn{0}]\in\Gamma\lin{\oW{0}}=\pi\sb{0}\fX\lin{\oW{0}}} for which
we may choose a representative \w[,]{\svn{0}\in\fX\lin{\oW{0}}\sb{0}}
corresponding to a map of \Tma[s] \w{\vn{0}:\fMT\oW{0}\to\fX} by Lemma
\ref{lfreema}.  In case (2), we may realize $\vare$ by a map
\w{\bv:\bY\to\oW{0}} by Lemma \ref{lfreeta}  and \wref[,]{eqextyoneda}
since $\Theta$ is allowable\vsm.

\noindent\textbf{II.~Step $\mathbf{n=1}$ of the outer induction\vsn.}

We choose a map \w{C\sp{0}\W{0}=\oW{0}\to\oW{1}} realizing the first attaching map
\w[,]{\odz{1}:\oV{1}\to V\sb{0}=\oV{0}} with \w{\W{1}} given in dimensions
$\leq 1$ by
\mydiagram[\label{eqwone}]{
\Wn{0}{1} \ar@/_{1.5pc}/[d]\sb{\dz{0}} \ar@/^{1.5pc}/[d]\sp{d\sp{1}\sb{0}} &=&
\oW{0} \ar@/_{0.5pc}/[d]\sb{\dz{0}=d\sp{1}\sb{0}=\Id}
\ar[drr]\sp{\udz{0}} && \times &&
P\oW{1} \ar[dll]\sb{d\sp{1}\sb{0}=p}
\ar@/^{0.5pc}/[d]\sp{\dz{0}=d\sp{1}\sb{0}=\Id} \\
\Wn{1}{1} \ar[u]\sp{s\sp{0}} &=& \oW{0}\ar@/_{0.5pc}/[u]\sb{=} &\times &
\oW{1} & \times & P\oW{1}.\ar@/^{0.5pc}/[u]\sp{=}
}

In case (1), to define the augmentation \w{\svn{1}} as a $0$-simplex in
\w{\fX\lin{\Wn{0}{1}}\sb{0}} extending
\w[,]{\svn{0}\in\fX\lin{\Wn{0}{0}}=\fX\lin{\oW{0}}} we use the
fact that \w[,]{\fX\lin{\Wn{0}{1}}=\fX\lin{\oW{0}}\times P\fX\lin{\oW{1}}}
by \S \ref{rgalg}(a)-(b), so we need only to find a $0$-simplex $H$
in \w{P\fX\lin{\oW{1}}} \wwh which, by \wref[,]{eqpathloop}
is a $1$-simplex in \w{\fX\lin{\oW{1}}} with \w[.]{d\sb{1}H=0}

In order to qualify as an augmentation \w{\fMT\Wn{0}{1}\to\fX} of simplicial
\Tma[s,]  \w{\vn{1}} must satisfy the simplicial identity
\begin{myeq}\label{eqaugid}
\vn{1}\circ d\sb{0}~=~\vn{1}\circ d\sb{1}~:~\fMT\Wn{1}{1}~\to~\fX
\end{myeq}
\noindent as maps of \Tma[s] \wwh or equivalently, these must correspond to
the same $0$-simplex in
\w[.]{\fX\lin{\Wn{1}{1}}=\fX\lin{\oW{0}}\times\fX\lin{\oW{1}}\times P\fX\lin{\oW{1}}}
In the first and third factors this obviously holds, so we need only
consider the two $0$-simplices in \w[:]{\fX\lin{\oW{1}}} in other words,
since the path fibration $p$ in \wref{eqwone} (induced by the inclusion
\w[)]{\Delta[0]\hra\Delta[1]} becomes \w{d\sb{0}} in \w[,]{\fX\lin{\oW{1}}}
we must choose $H$ so that \w{d\sb{0}H} is the $0$-simplex
\w[.]{(\udz{0})\sb{\#}\svn{0}}
By \wref[,]{eqaugmzero} \w{\vare\circ\odz{0}=0} in \w[,]{\TAlg}
which implies (by our choices of \w{\udz{0}} and \w{\svn{0}} representing
\w{\odz{0}} and $\vare$, respectively) that \w{(\udz{0})\sb{\#}\svn{0}}
is nullhomotopic, so the required $H$ exists.

In case (2), we choose a nullhomotopy for \w{\udz{0}\circ\bv} to extend \w{\bve{0}}
to the factor \w[,]{P\oW{1}} and thus define \w[\vsm.]{\bve{1}:\bY\to\Wn{0}{1}}

\noindent\textbf{III.~Step $\mathbf{n}$ of the outer induction
\wb[\vsn:]{n\geq 2}}

To construct \w{\W{n}} given \w[,]{\W{n-1}} by Proposition \ref{pnstep} it suffices
to produce a cochain map \w{F:\cMs\W{n-1}\to\bDs} as in \S \ref{senso}
(where the left Reedy fibrant
replacement \w{\bDs} for \w{\oW{n}\ouS{n-1}} is given by \wref{eqstloop} and
\wref[).]{eqtau} We do so by a descending induction on the cosimplicial dimension
\w[,]{0\leq k\leq n-1} starting with step \textbf{IV} below for \w[,]{k=n-1}
under the following induction hypotheses:

In stage $k$ we assume the existence of
\w{\Fk{j}:C\sp{j}\W{n-1}\to P\Omega\sp{n-j-2}\oW{n}} (constituting a cochain map for
\w[)]{k+1\leq j\leq n-1} and
\w{\ak{j}:Z\sp{j}\W{n-1}\to\Omega\sp{n-j-2}\oW{n}} for \w[,]{k\leq j\leq n-1}
with a nullhomotopy \w{\hFk{k}:C\sp{k}\W{n-1}\to P\Omega\sp{n-k-2}\oW{n}}
such that
\begin{myeq}\label{eqhhvanish}
\ak{k}\circ w\sp{k}~=~p\circ \hFk{k}:\cM{k}\W{n-1}\to\Omega\sp{n-k-2}\oW{n}~,
\end{myeq}
\noindent as in \wref[.]{eqak} By Lemma \ref{linducea}, \w{\hFk{k}} induces
\w{\hak{k-1}:Z\sp{k-1}\W{n-1}\to\Omega\sp{n-k-1}\oW{n}} with
\begin{myeq}\label{eqhikak}
\iota\circ\hak{k-1}\circ w\sp{k-1}~=~\hFk{k}\circ\delta\sp{k-1}
\end{myeq}
\noindent as in \wref[,]{eqikak} where
\w{\iota=\ip{n-k-1}:\Omega\sp{n-k-1}\oW{n}\hra P\Omega\sp{n-k-2}\oW{n}}
is the inclusion. However, we do not assume in our induction hypothesis that
\w{\hak{k-1}\circ w\sp{k-1}} is nullhomotopic\vsn.

\noindent\textbf{IV.~Step $\mathbf{k=n-1}$ of the inner induction\vsn:}

To define \w[,]{\Fk{n-1}:C\sp{n-1}\W{n-1}\to\oW{n}} note that the simplicial space
\w{\Ud:=\map\sb{\C}(\W{n-1},\oW{n})} is Reedy fibrant, since \w{\W{n-1}} is Reedy
cofibrant. Moreover, since
\w{\HiT{\Wn{k}{n-1}}\cong V\sb{k}} for all \w{0\leq k<n} by \S \ref{dscr}(a),
the attaching map \w{\odz{n}:\oV{n}\to V\sb{n-1}} yields a homotopy class
\begin{myeq}\label{eqmoorecyc}
\begin{split}
\alpha~\in&~[\Wn{n-1}{n-1},\,\oW{n}]
~=~\pi\sb{0}\fMT\Wn{n-1}{n-1}\lin{\oW{n}}~=~V\sb{n-1}\lin{\oW{n}}\\
~=&~\Hom\sb{\TAlg}(\HiT{\oW{n}},\,V\sb{n-1})
~=~\Hom\sb{\TAlg}(\oV{n},\,V\sb{n-1})~,
\end{split}
\end{myeq}
\noindent where the next to last equality follows from Lemma \ref{lfreeta},
as extended in \wref[.]{eqextyoneda}

This $\alpha$ is a Moore chain in \w{\pi\sb{0}\Ud} by Definition \ref{dscwo},
so by Lemma \ref{lmoore}(a), it can be represented by a
map \w[,]{\Fk{n-1}:C\sp{n-1}\W{n-1}\to\oW{n}} which induces
\w{\ak{n-2}:Z\sp{n-2}\W{n-1}\to\oW{n}} by Lemma \ref{linducea} with
\w[.]{\ak{n-2}\circ w\sp{n-2}~=~\Fk{n-1}\circ\delta\sp{n-2}}
Moreover, by \wref{eqattach} $\alpha$ is in fact a Moore \wwb{n-1} cycle, so
we have a nullhomotopy \w{\hFk{n-2}:C\sp{n-2}\to P\oW{n}} for
\w{\ak{n-2}\circ w\sp{n-2}} as in \wref[\vsm.]{eqhhvanish}

\noindent\textbf{V.~Step $\mathbf{k}$ of the inner induction \wb[\vsn:]{1<k\leq n-2}}

\noindent Let \w[.]{\bA:=\Omega\sp{n-k-2}\oW{n}} By assumption (Step \textbf{III}),
we have a nullhomotopy
\w{\hFk{k}:C\sp{k}\W{n-1}\to P\bA} for \w[,]{\ak{k}\circ w\sp{k}} and
\w{\hFk{k}\circ\delta\sp{k-1}} determines
\w{\hak{k-1}:Z\sp{k-1}\W{n-1}\to\Omega\bA}
satisfying \wref[,]{eqhikak} which is thus a \wwb{k-1}cycle for the Reedy
fibrant bisimplicial set \w[.]{\map\sb{\C}(\W{n-1},\,\Omega\bA)}

Since \w{\W{n-1}} realizes \w{\Vd} through simplicial dimension
\w[,]{n-1>k} by \S \ref{dscr}(a),
\w{\hak{k-1}\circ w\sp{k-1}\circ v\sp{k-1}:\W{n-1}\to\Omega\bA}
represents a \wwb{k-1}cycle \w{[\hak{k-1}]} for \w[,]{\Vd\lin{\Omega\bA}}
as in \wref[.]{eqmoorecyc}
Because \w{\Vd\to\Gamma} is a resolution, and thus acyclic, there is a
Moore chain \w{\gamma\sb{k}} in \w{C\sb{k}\Vd\lin{\Omega\bA}} with
\w[.]{\partial\sb{0}\sp{V\sb{k}}(\gamma\sb{k})=[\hak{k-1}]} Since
\w[,]{V\sb{k}\lin{\Omega\bA}\cong\pi\sb{0}\fMT\Wn{k}{n-1}\lin{\Omega\bA}=
[\Wn{k}{n-1},\,\Omega\bA]}
by Lemma \ref{lmoore}(a) we can represent \w{\gamma\sb{k}}
by a map \w[,]{g\sp{k}:C\sp{k}\W{n-1}\to\Omega\bA}
while by Lemma \ref{lmoore}(b) we have a homotopy
\begin{myeq}\label{eqhkmone}
G:g\sp{k}\circ \delta\sp{k-1}\sim\hak{k-1}\circ w\sp{k}:
C\sp{k-1}\W{n-1}\to\Omega\bA~.
\end{myeq}
\noindent Next, the concatenation of homotopies
provides an action of \w{\Hom(\bC,\,\Omega\bA)} on \w{\Hom(\bC,\,P\bA)}
(see \cite[\S 1]{SpanS}), which we use to define a new nullhomotopy
\begin{myeq}\label{eqnewnullh}
\Fk{k}~:=~\hFk{k}\star(\iota\circ g\sp{k})\sp{-1}~:~C\sp{k}\W{n-1}~\to~P\bA
\end{myeq}
\noindent for \w{\ak{k}\circ w\sp{k}} (where \w{\iota:\Omega\bA\hra P\bA} is
the inclusion).

By Lemma \ref{linducea}, \w{\Fk{k}} induces a map
\w{\ak{k-1}:Z\sp{k-1}\W{n-1}\to\Omega\bA} satisfying
\begin{myeq}\label{eqhiikak}
\begin{split}
\iota&\circ\wa\circ w\sp{k-1}~=~\wF\circ\delta\sp{k-1}~=~
(\hFk{k}\star(\iota\circ g\sp{k})\sp{-1})\circ\delta\sp{k-1}\\
~=&~(\hFk{k}\circ\delta\sp{k-1})\star(\iota\circ g\sp{k}\circ\delta\sp{k-1})\sp{-1}
~=~\iota\circ\left[(\hak{k-1}\circ w\sp{k-1})\star
(g\sp{k}\circ\delta\sp{k-1})\sp{-1})\right]
\end{split}
\end{myeq}
\noindent by \wref[,]{eqikak} \wref[,]{eqnewnullh} \wref[,]{eqhikak} and
the fact that the H-space structure $\star$ and \w{(-)\sp{-1}}
commutes with precomposition of maps into \w[.]{\Omega\bA}

Since $\iota$ is a monomorphism, we conclude that
\begin{myeq}\label{eqnullhtpy}
\begin{split}
\ak{k-1}\circ w\sp{k-1}~=&~(\hak{k-1}\circ w\sp{k-1})\star
(g\sp{k}\circ\delta\sp{k-1})\sp{-1})\\
~\sim&~(\hak{k-1}\circ w\sp{k-1})\star(\hak{k-1}\circ w\sp{k-1})\sp{-1}~
\sim~0~,
\end{split}
\end{myeq}
\noindent by \wref[,]{eqhkmone} so \w{\ak{k-1}} satisfies
\wref{eqhhvanish} for \w[\vs.]{k-1}

In case (2), the last two steps of the downward induction are no different
from those for \w[,]{k\geq 2} if we set \w[,]{\tWn{-1}{n}=\Wn{-1}{n-1}:=\bY} with
\w{\tdz{-1}:\tWn{-1}{n}\to\tWn{0}{n}} as the coaugmentation \w[.]{\tve{n}}
However, in case (1) we no longer have an object \w{\Wn{k-1}{n-1}}
in $\C$ for \w[,]{k=0} so we must modify our construction somewhat,
using the language of \Tma[s,] as follows\vsm:

\noindent\textbf{VI.~Step $\mathbf{k=1}$ of the descending
induction\vsn:}

By the descending induction hypotheses \textbf{III} for \w{k=1} we have
some nullhomotopy \w{\hFk{1}:\ak{1}\circ w\sp{1}\sim 0} and \w{\hak{0}} with
\w{\iota\circ\hak{0}\circ w\sp{0}=\hFk{1}\circ \dz{0}} by \wref[.]{eqhikak}
Let \w[.]{\bB:=\Omega\sp{n-3}\oW{n}}

We can think of \w{\ak{1}\circ w\sp{1}\circ v\sp{1}} as a $0$-simplex
\w[,]{\uak{1}\in\fMT\Wn{1}{n-1}\lin{\bB}\sb{0}} of
\w{\hak{0}\circ w\sp{0}\circ v\sp{0}} as a $0$-simplex
\w[,]{\uhak{0}\in\fMT\Wn{0}{n-1}\lin{\Omega\bB}\sb{0}}
and of \w{\hFk{1}\circ v\sp{1}} as a $1$-simplex
\w[,]{\uhFk{1}\in\fMT\Wn{1}{n-1}\lin{\bB}\sb{1}}
implicitly using  the natural isomorphism
 \w{i\sb{P}:\fMT\Wn{0}{n-1}\lin{P\bB}\cong P\fMT\Wn{0}{n-1}\lin{\bB}}
of  \S \ref{rgalg}(b),  and the inclusion of \w{j:(PK)\sb{i}\subseteq K\sb{i+1}}
for any Kan complex $K$  and \w{i\geq 0} (see \cite[\S 23.3]{MayS}). Thus
\w{\hFk{1}:\ak{1}\circ w\sp{1}\sim 0} means \w{d\sb{0}\uhFk{1}=\uak{1}} and
\w{d\sb{1}\uhFk{1}=0} (simplicial face maps in the mapping space).
The fact that the domain of \w{\hFk{1}} is \w{C\sp{1}\W{n-1}} implies that
\begin{myeq}\label{eqhfkvanishkone}
(d\sp{1})\sp{\ast}(\uhFk{1})~=~0 \hs\text{in}\hs
\fMT\Wn{0}{n-1}\lin{\bB}\sb{1}~,
\end{myeq}
\noindent and equation \wref{eqhikak} becomes:
\begin{myeq}\label{eqhikaone}
j\circ \iota\circ i\sb{\Omega}(\uhak{0})~=~(d\sp{0})\sp{\ast}(\uhFk{1})
\hs\text{in}\hs  \fMT\Wn{0}{n-1}\lin{\bB}\sb{1}
\end{myeq}
\noindent again using  \w[,]{i\sb{P}} the isomorphism
 \w{i\sb{\Omega}} of  \S \ref{rgalg}(c), the inclusion
\w[,]{\iota:\Omega K\hra PK} and \w{j:(PK)\sb{0}\subseteq K\sb{1}}
as above for \w[.]{K= \fMT\Wn{0}{n-1}\lin{\bB}}

Moreover, we have an augmentation of simplicial \Tma[s]
\w{\vn{n-1}:\fMT\W{n-1}\to\fX} by Definition \ref{dscr}(b) in the outer induction
hypothesis. Thus \w{\vn{n-1}} satisfies the simplicial identity
\begin{myeq}\label{eqsimpidaug}
\vn{n-1}\circ(d\sp{0})\sp{\ast}~=~\vn{n-1}\circ(d\sp{1})\sp{\ast}~.
\end{myeq}
\noindent In order to apply Lemma \ref{linducea}, since $\bY$ is not available in
cosimplicial dimension \w[,]{-1} we need to verify that
\w{\uhak{0}} is a ``Moore $0$-cycle'' in the augmented simplicial \Tma
\w{\fMT\W{n-1}\to\fX} \wwh that is:
\begin{myeq}\label{eqfhakvanishz}
\vn{n-1}(\uhak{0})~=~0 \hs\text{in}\hs  \fX\lin{\Omega\bB}\sb{0}~,
\end{myeq}
\noindent so after post-composition with the monic maps
$$
 \fX\lin{\Omega\bB}\sb{0}~\xra{\iota\sb{\Omega}}~\Omega\fX\lin{\bB}\sb{0}
~\xra{\iota}~P\fX\lin{\bB}\sb{0}~\xra{j}~\fX\lin{\bB}\sb{1}
$$
\noindent it suffices to show:
\begin{myeq}\label{eqhavanishz}
j\circ\iota\circ\iota\sb{\Omega}\circ\vn{n-1}(\uhak{0})~=~0~,
\end{myeq}
\noindent Note that the following diagram of sets commutes:
\mydiagram[\label{eqzeroonecoface}]{
& \uhFk{1}\in \fMT\Wn{1}{n-1}\lin{\bB}\sb{1}
\ar[d]^{(\dz{1})\sp{\ast}} \ar[rd]^{(d\sb{1}\sp{1})\sp{\ast}} &\\
\uhak{0}\in\fMT\Wn{0}{n-1}\lin{\Omega\bB}\sb{0} \ar[d]^{\vn{n-1}}
\ar[r]^{j\circ\iota\circ\iota\sb{\Omega}}&
\fMT\Wn{0}{n-1}\lin{\bB}\sb{1} \ar[d]^{\vn{n-1}}
& \fMT\Wn{0}{n-1}\lin{\bB}\sb{1} \ar[ld]^{\vn{n-1}}\\
\fX\lin{\Omega\bB}\sb{0} \ar[r]^{j\circ\iota\circ\iota\sb{\Omega}}&
\fX\lin{\bB}\sb{1} &
}
\noindent where the left hand square commutes by naturality of
\w[,]{j\circ\iota\circ\iota\sb{\Omega}} and the right hand side
commutes by \wref[.]{eqsimpidaug} By \wref[,]{eqhikaone} the simplices
\w{\uhFk{1}} and \w{\uhak{0}} thus map to the same value in
\w[,]{\fX\lin{B}\sb{1}} which is zero by \wref[,]{eqhfkvanishkone}
proving that \wref[,]{eqhavanishz} and thus \wref[,]{eqfhakvanishz}
in fact hold.

By the outer induction hypothesis that Definition \ref{dscr}(b)  hold,
the augmented simplicial abelian group
\begin{myeq}\label{eqacaugabgp}
\Vd\lin{\Omega\bB}~\xra{\vare}~
\pi\sb{0}\fX\lin{\Omega\bB}~=~\Gamma\lin{\Omega\sp{n-2}\oW{n}}
\end{myeq}
\noindent is realized by \w{\W{n-1}} through simplicial dimension \w[,]{n-1}
so that
\begin{myeq}\label{eqrealizeasag}
V\sb{k}\lin{\Omega\bB}~\cong~\pi\sb{0}\fMT\Wn{k}{n-1}\lin{\Omega\bB}~=~
[\Wn{k}{n-1},\,\Omega\bB]
\end{myeq}
\noindent for \w{k=1} and $0$.

Since by \wref{eqfhakvanishz} \w{[\uhak{0}]} is a $0$-cycle in
\wref[,]{eqacaugabgp} which is acyclic, there is a Moore $1$-chain
\w{\gamma\sb{1}\in V\sb{1}\lin{\Omega\bB}} with
\w[.]{\partial\sb{0}\sp{V\sb{1}}(\gamma\sb{1})~=~[\hak{0}]}
As in \wref{eqnewnullh} of Step \textbf{V}, this can be used
to produce a map \w{\Fk{1}:C\sp{1}\W{n-1}\to P\bB} such that
\begin{myeq}\label{eqdonefone}
\Fk{1}\circ d\sp{1}~=~0~,
\end{myeq}
\noindent  yielding \w{\ak{0}:\Wn{0}{n-1}\to\Omega\bB} (by Step \textbf{III})
having a nullhomotopy \w{\hFk{0}:\ak{0}\sim 0} (again, as in Step \textbf{V}).
Moreover, \wref{eqdonefone} and \wref{eqsimpidaug} together imply that
(after replacing \w{\hFk{1}} by \w{\Fk{1}} in \wref[)]{eqzeroonecoface}
one can deduce that
\begin{myeq}\label{eqfakvanishz}
\vn{n-1}(\uak{0})~=~0\hsp \text{in}\hsm \fX\lin{\Omega\bB}~,
\end{myeq}
\noindent as required\vsm.

\noindent\textbf{VII.~Step $\mathbf{k=0}$ of the descending induction\vsn:}

For \w[,]{k=0} let \w[.]{\bC:=\Omega\sp{n-2}\oW{n}}  The nullhomotopy
\w{\hFk{0}:\ak{0}\circ w\sp{0}\sim 0} is represented as in Step \textbf{VI}
by a $1$-simplex \w{\uhFk{0}\in\fMT\Wn{0}{n-1}\lin{\bC}\sb{1}}
with \w[,]{d\sb{0}\uhFk{0}=\uak{0}} and \w[.]{d\sb{1}\uhFk{0}=0}
Thus for \w{\vn{n-1}(\uhFk{0})\in\fX\lin{\bC}\sb{1}}
we have \w[,]{d\sb{0}(\vn{n-1}(\uhFk{0}))=\vn{n-1}(\uak{0})=0} by
\wref[,]{eqfakvanishz} as well as \w[.]{d\sb{1}(\vn{n-1}(\uhFk{0}))=0}

Therefore, as in \wref{eqhikaone} we obtain \w{\uhak{}\in\fX\lin{\Omega\bC}\sb{0}}
with \w[.]{j\circ\iota\circ\iota\sb{\Omega}(\uhak{})~=~\vn{n-1}(\uhFk{0})}

Note that this argument fails if we do not have the mapping algebra $\fX$
(with $\bT$ as in \S \ref{rgalg}), which allows us to obtain an element
\w{[\uhak{}]\in\Gamma\lin{\Omega\bC}=\Gamma\lin{\Omega\sp{n-1}\oW{n}}}
from the nullhomotopy \w[.]{\hFk{0}:\Wn{0}{n-1}\to P\bC}

Again we use the acyclicity of \wref{eqacaugabgp} to deduce the existence of
a Moore $0$-chain \w{\gamma\sb{0}\in V\sb{0}\lin{\Omega\bC}} with
\begin{myeq}\label{eqzerovar}
\vare(\gamma\sb{0})~=~[\uhak{}]\in\Gamma\lin{\Omega\bC}~.
\end{myeq}
\noindent We no longer require Lemma \ref{lmoore} to deduce that
we can represent \w{\gamma\sb{0}} by a map
\w[,]{g\sp{0}:\Wn{0}{n-1}\to\Omega\bC} corresponding to a
$0$-simplex \w[,]{\uetk{0}\in\fMT\Wn{0}{n-1}\lin{\Omega\bC}\sb{0}}
and thus a $1$-simplex
\begin{myeq}\label{eqhnullh}
\uH~:=~(j\circ\iota\circ\iota\sb{\Omega})(\uetk{0})~\in~
\fMT\Wn{0}{n-1}\lin{\bC}\sb{1}\hsm\text{with}\hsm d\sb{0}\uH=d\sb{1}\uH=0~.
\end{myeq}
\noindent Since \w{\fMT\Wn{0}{n-1}\lin{\bC}} is a homotopy group object in
\w[,]{\Sa} with homotopy group structure $\star$ and inverse \w{(-)\sp{-1}}
induced from those of \w[,]{\bC\in\Theta\subset\C} we may define a new $1$-simplex
\begin{myeq}\label{eqzerohtpy}
\uFk{0}:=\uhFk{0}\star\uH\sp{-1}\in\fMT\Wn{0}{n-1}\lin{\bC}\sb{1}\hs
\text{with}\hs d\sb{0}\uFk{0}=\uak{0} \hs \text{and}\hs d\sb{1}\uFk{0}=0~.
\end{myeq}
\noindent Since \w[,]{\vn{n-1}(\uak{0})=0} the $1$-simplex \w{\vn{n-1}(\uFk{0})\in\fX\lin{\bC}\sb{1}}
is in the image of the composite inclusion
$$
\fX\lin{\Omega\bC}\sb{0}~\xra{\iota\sb{\Omega}}~\Omega\fX\lin{\bC}\sb{0}~\xra{\iota}
~P\fX\lin{\bC}\sb{0}~\xra{j}~\fX\lin{\bC}\sb{1}~.
$$
\noindent Thus, there is a $0$-simplex \w{\uak{}\in\fX\lin{\Omega\bC}\sb{0}} with
\begin{myeq}\label{eqaugfa}
(j\circ\iota\circ\iota\sb{\Omega})(\uak{})~=~\vn{n-1}(\uFk{0})~.
\end{myeq}
\noindent Moreover, since \w{(j\circ\iota\circ\iota\sb{\Omega})} and \w{\vn{n-1}}
commute with the homotopy group structure,
$$
(j\circ\iota\circ\iota\sb{\Omega})(\uak{})~=~
\vn{n-1}(\uhFk{0})\star(\vn{n-1}(\uH)\sp{-1})~=~
[(j\circ\iota\circ\iota\sb{\Omega})(\uhak{})]\star
[(j\circ\iota\circ\iota\sb{\Omega})(\vn{n-1}\uetk{0})]\sp{-1}~,
$$
\noindent by \wref{eqhnullh} and \wref[.]{eqaugfa} Since
\w{j\circ\iota\circ\iota\sb{\Omega}} is monic, this implies that:
\begin{myeq}\label{eqzeroht}
\uak{}~=~\uhak{}\star(\vn{n-1}\uetk{0})\sp{-1}~.
\end{myeq}
\noindent By \wref{eqzerovar} there is a $1$-simplex \w{G\in\fX\lin{\Omega\bC}}
with \w{d\sb{0}G=\uhak{}} and \w{d\sb{1}G=\vn{n-1}(\uetk{0})} in
\w[.]{\fX\lin{\Omega\bC}\sb{0}} Therefore,
\w{G':=G\star(s_{0}\vn{n-1}(\uetk{0}))\sp{-1}\in\fX\lin{\Omega\bC}\sb{1}}
satisfies
\begin{myeq}\label{eqnegonehtpy}
d\sb{0}G'=\uhak{}\star(\vn{n-1}(\uetk{0}))\sp{-1} \hs\text{and}\hs
d\sb{1}G'=\vn{n-1}(\uetk{0})\star(\vn{n-1}(\uetk{0}))\sp{-1}~.
\end{myeq}

Since \w{\Omega\bC} is a homotopy group object, there is a $1$-simplex
\w{K\in\fX\lin{\Omega\bC}\sb{1}} with
\begin{myeq}\label{eqhtpyinverse}
d\sb{0}K=\vn{n-1}(\uetk{0})\star(\vn{n-1}(\uetk{0}))\sp{-1}\hsm\text{and}\hsm
d\sb{1}K=0~.
\end{myeq}
\noindent Since \w{\fX\lin{\Omega\bC}} is a Kan complex, we have
\w{\sigma\in\fX\lin{\Omega\bC}\sb{2}} with \w{d\sb{0}\sigma=G'} and
\w[,]{d\sb{1}\sigma=K} so if we set \w[,]{\uhFk{}:=d\sb{2}\sigma} we find
\begin{myeq}\label{eqhtpymove}
d\sb{0}\uhFk{}=\uhak{}\star(\vn{n-1}(\uetk{0}))\sp{-1}\hsm\text{and}\hsm
d\sb{1}\uhFk{}=0~.
\end{myeq}
\noindent We deduce from \wref{eqzeroht} that \w{\uhFk{}} is a nullhomotopy
for \w{\uak{}} in \w[,]{\fX\lin{\Omega\bC}} as required in Step \textbf{III}, thus
completing the construction of the restricted cosimplicial object \w[.]{\tWu{n}}

Since making the cosimplicial object \w{\vWu{k}} as in \S \ref{senso}\textbf{(iii)}
Reedy fibrant or cofibrant does not change \w[,]{\vWn{0}{k}} we see by induction
from \wref{eqsmatch} that
\begin{myeq}\label{eqwzero}
\tWn{0}{n}~=~\Wn{0}{n-1}\times P\Omega\bC~=~
\prod\sb{k=0}\sp{n}\ P\Omega\sp{k-1}\oW{k}
\end{myeq}
\noindent (using the convention of \wref[).]{eqconvent} Note that all the factors
on the right hand side of \wref{eqwzero} are in $\bT$, by assumption \ref{rgalg},
so \w{\tWn{0}{n}} is, too.

Thus by Lemma \ref{lfreema}, the augmentation \w{\tvn{n}:\fMT\tWn{0}{n}\to\fX}
is determined by the choice of a suitable $0$-simplex
$$
e=(e',e'')\in\fX\lin{\tWn{0}{n}}\sb{0}=
\fX\lin{\Wn{0}{n-1}\times P\Omega\bC}\sb{0}
=\fX\lin{\Wn{0}{n-1}}\sb{0}\times P\fX\lin{\Omega\bC}\sb{0}
$$
\noindent by \wref[,]{eqdopbg} using \S \ref{rgalg}(a), where the component
\w{e'\in\fX\lin{\Wn{0}{n-1}}} corresponds to the given
\w[.]{\vn{n-1}:\fMT\Wn{0}{n-1}\to\fX}

On the other hand, as before  \w{e''} corresponds to the
nullhomotopy \w[,]{\uhFk{}\in\fX\lin{\Omega\bC}\sb{1}} which actually lands in
\w[.]{P\fX\lin{\Omega\bC}\sb{0}}

The cosimplicial identity
\begin{myeq}\label{eqsimpidaugm}
\tvn{n}\circ(\dz{0})\sp{\ast}~=~\tvn{n}\circ(d\sp{1}\sb{0})\sp{\ast}~:~
\fMT\tWn{1}{n}\to\fX
\end{myeq}
\noindent may be verified using Lemma \ref{lfreema} by representing both sides
of the equation by elements in
$$
\fX\lin{\tWn{1}{n}}\sb{0}~=~
\fX\lin{\Wn{1}{n-1}}\sb{0}\times\fX\lin{P\bC}\sb{0}~,
$$
\noindent where the components in \w{\fX\lin{\Wn{1}{n-1}}\sb{0}} agree
because \w{(\bve{n-1})\sp{\ast}:\fMT\W{n-1}\to\fX} is an augmented simplicial \Tma[,]
by the outer induction hypothesis. Here we are using the last assumption in
\S \ref{rgalg} to deduce that \w[.]{\Wn{1}{n-1}\in\bT}

Thus it remains to identify the corresponding components in
\w[.]{\fX\lin{P\bC}\sb{0}}
On the one hand, \w{\qk{1}{n}\circ\dz{0}:\tWn{0}{n-1}\to P\bC}
is the map \w[,]{\Fk{0}\circ\psn{0}{n}} and
\w{\tvn{n}\circ(\psn{0}{n})\sp{\ast}:\fMT\Wn{0}{n-1}\to\fX} is the given
\w[,]{\vn{n-1}} by construction, and since
\w{(\Fk{0})\sp{\ast}:\fMT P\bC\to\fMT\Wn{0}{n-1}} is represented
according to Lemma \ref{lfreema} by \w[,]{\uFk{0}\in\fMT\Wn{0}{n-1}\lin{P\bC}}
we see that the composite \w{\tvn{n}\circ(\qk{1}{n}\circ\dz{0})\sp{\ast}} is
represented in \w{\fX\lin{P\bC}\sb{0}} by
$$
\vn{n-1}(\uFk{0})~=~(j\circ\iota\sb{\Omega})(\uak{})~,
$$
\noindent according to \wref[.]{eqaugfa}

On the other hand,
\w{\qk{1}{n}\circ d\sp{1}\sb{0}:\tWn{0}{n-1}\to P\bC} equals
\w[,]{\iota\circ p\circ\qk{0}{n}}
so
$$
\tvn{n}\circ(d\sp{1}\sb{0})\sp{\ast}~=~
\tvn{n}\circ(\qk{0}{n})\sp{\ast}\circ(\iota\circ p)\sp{\ast}
$$
\noindent where
\w{\tvn{n}\circ(\qk{0}{n})\sp{\ast}:\fMT P\Omega\bC\to\fX}
is represented by \w[,]{\uhFk{}\in\fX\lin{\Omega\bC}\sb{1}}
so \w{\tvn{n}\circ(\qk{0}{n})\sp{\ast}\circ(\iota\circ p)\sp{\ast}} is
represented by
$$
(j\circ\iota\sb{\Omega})(d\sb{0}\uhFk{})~=~
(j\circ\iota\sb{\Omega})(\uak{})~\in~\fX\lin{P\bC}\sb{0}
$$
\noindent as above\vsn.

This completes the $n$-th step of the outer induction in case (1)\vsm.

\noindent\textbf{VIII.~Making $\mathbf{\cW}$ cofibrant in case (2)\vsn:}

Note that when we represent the attaching map
\w{\odz{n}:\oV{n}\to\cZl{n-1}\Vd} for \w{\Vd} by a map
\w{\phi:\oV{n}\olS{n-1}\to\cMls\Vd} as in \S \ref{rscwo},
$\phi$ extends to a map of \emph{augmented} \wwb{n-1}truncated
chain complexes in \w[,]{\TAlg} ending in \w[.]{\cMl{-1}\Vd:=\Gamma}
Moreover, in case (2) the cochain map \w{F:\cMs\W{n-1}\to\bDs} realizing $\phi$
which we have constructed in the outer induction above also extends to a map
of coaugmented \wwb{n-1}truncated cochain complexes, with
\w[,]{\cM{-1}\W{n-1}:=\bY} \w{\dif{-1}} given by \w[,]{\bve{n-1}} and
\w{\Fk{-1}} given by Lemma \ref{lvanish}.

In the left Reedy model category structure on the category \w{\Chnm{\C}{n-1}}
of \wwb{n-1}truncated coaugmented cochain complexes in $\C$,
factor $F$ as a cofibration \w{G:\cMs\W{n-1}\hra\bEs} followed by a
trivial fibration \w[.]{T:\bEs\epic\bDs} Applying the functorial construction
of \S \ref{dscr}(c) to $G$ yields \w{\W{n}} as required for the cofibrant
sequential realization \w[.]{\cW} The vertical trivial fibrations
of \wref{eqmodpathloopve} are induced by $T$.
\end{proof}

\end{document}